\documentclass[final,floating,letterpaper]{siamltex1213}
\usepackage{latexsym,amsmath,amssymb}
\allowdisplaybreaks[4]

\usepackage{comment}
\usepackage[compress]{cite}
\usepackage{algorithm}
\usepackage{algorithmic}

\algsetup{indent=2em}
\usepackage{mathbbold}
\usepackage{bbm}
\usepackage{amsfonts}
\usepackage[T1]{fontenc}
\usepackage{latexsym,amssymb,amsmath,amsfonts}
\usepackage{graphicx}
\usepackage{subfigure}
\usepackage{epsf}
\usepackage{epstopdf}
\usepackage{amssymb}
\usepackage{hyperref}

\usepackage{bm}% bold math
\newtheorem{remark}[theorem]{Remark}

\renewcommand{\theequation}{\arabic{section}.\arabic{equation}}

\newcommand{\R}{{\mathbb R}}

\newcommand{\C}{{\mathbb C}}

\newcommand{\be}{\begin{eqnarray}}
\newcommand{\ben}{\begin{eqnarray*}}
\newcommand{\en}{\end{eqnarray}}
\newcommand{\enn}{\end{eqnarray*}}
\newcommand{\ba}{\backslash}
\newcommand{\pa}{\partial}

\newcommand{\ov}{\overline}

\newcommand{\ggrad}{{\rm grad\,}}        %%%%%%%%%%%%%%%%%%%%%%%%%%%%%%%%%%%%%%%
\newcommand{\ddiv}{{\rm div\,}}

\newcommand{\g}{\gamma}

\newcommand{\Om}{\Omega}
\newcommand{\om}{\omega}

\newcommand{\al}{\alpha}
\newcommand{\la}{\lambda}

\newcommand{\hx}{\hat{x}}
\title{Direct sampling methods for inverse elastic scattering problems}

\author{Xia Ji\thanks{LSEC, Academy of Mathematics and Systems Science, Chinese Academy of Sciences, Beijing 100190, China ({\tt jixia@lsec.cc.ac.cn}).}
\and
Xiaodong Liu\thanks{Institute of Applied Mathematics, Academy of Mathematics and Systems Science, Chinese Academy of Sciences, 100190 Beijing, China({\tt xdliu@amt.ac.cn}).}
\and
Yingxia Xi\thanks{LSEC, Academy of Mathematics and Systems Science, Chinese Academy of Sciences, Beijing 100190, China ({\tt yxiaxi@lsec.cc.ac.cn}).}
}
\begin{document}

\maketitle

%%%%% Begin Abstract %%%%%%%%%%%
\begin{abstract}
We consider the inverse elastic scattering of incident plane compressional and shear waves from the knowledge of the far field patterns.
Specifically, three direct sampling methods for location and shape reconstruction are proposed using the different component of the far field patterns.
Only inner products are involved in the computation, thus the novel sampling methods are very simple and fast to be implemented.
With the help of the factorization of the far field operator, we give a lower bound of the proposed indicator functionals for sampling points inside the scatterers.
While for the sampling points outside the scatterers, we show that the indicator functionals decay like the Bessel functions as the sampling
point goes away from the boundary of the scatterers.
We also show that the proposed indicator functionals continuously dependent on the far field patterns, which further implies that the novel sampling
methods are extremely stable with respect to data error.
For the case when the observation directions are restricted into the limited aperture, we firstly introduce some data retrieval techniques to obtain those data that can not be measured
directly and then use the proposed direct sampling methods for location and shape reconstructions.
Finally, some numerical simulations in two dimensions are conducted with noisy data, and the results further verify the effectiveness and robustness of
the proposed sampling methods, even for multiple multiscale cases and limited-aperture problems.
\end{abstract}
%%%%% end %%%%%%%%%%%

%%%%% AMS/PACs/Keywords %%%%%%%%%%%
%\pacs{}
%\ams{}


\pagestyle{myheadings}
\thispagestyle{plain}
\markboth{X. Ji, X. Liu and Y. Xi}
{Direct sampling methods for inverse elastic scattering problems}
\section{Introduction}

Scattering of elastic waves plays an important role in such different areas as seismic imaging, nondestructive testing, geophysical exploration,  and medical diagnosis.
In the last twenty years, non-iterative sampling methods for location and shape reconstruction in inverse elastic scattering problems have attracted a lot of interest.
Typical examples include the Linear Sampling Method by Arens \cite{Arens2001}, the Factorization Method by Alves and Kress \cite{AlvesKress2002}, Arens \cite{Arens2001},
Charalambopoulos \& Kirsch et al. \cite{TYG}, and recently by Hu, Kirsch \& Sini \cite{HuKirschSini},
and the method based on topological derivative by Guizina \& Chikivhev \cite{GuzinaChikichev}.
The basic idea is to design an indicator which is big inside the underlying scatterer and relatively small outside.
Recently, in \cite{ChenHuang}, Chen \& Huang introduced the Reverse Time Migration where the indicator functional is defined as the imaginary part of the cross-correlation
of the weighted elastic Green function and the weighted back-propagated elastic wave field. Different to the previous sampling methods, the indicator functional decays as the
sampling points go away from the boundary.
We also refer to Li, Wang et al. \cite{LiWangWangZhuo} and Bao \& Hu et al. \cite{BaoHuSunYin} for iterative methods by using multiple frequency data.
For the readers interested in a more comprehensive treatment of the inverse elastic scattering problems, we suggest consulting Bonnet and Constantinescu \cite{BonnetConstantinescu}  on
this subject.

The purpose of this paper is to introduce some novel indicator functionals for inverse elastic scattering problems with the far field patterns. These indicator functionals
have a lower bound for sampling points inside the scatterers and decay like the Bessel functions as the sampling point goes away from the boundary of the scatterers.
Besides, some data retrieval techniques will be introduced for limited aperture problems.
We will focus on the case of two-dimensional scattering and present both the theoretical and numerical results for this case,
illustrating the fast, effective and robust reconstructions, even for the limited aperture problems.

We begin with the notations used throughout this paper. All vectors will be denoted in bold script.
For a vector $\mathbf{x}:=(x_1, x_2)^{T}\in\R^2$, we introduce the two unit vectors $\hat{\mathbf{x}}:=\mathbf{x}/|\mathbf{x}|$ and $\hat{\mathbf{x}}^{\perp}$ obtained by
rotating $\hat{\mathbf{x}}$ anticlockwise by $\pi/2$.
For simplicity, we write $\pa_i$ for the usual partial derivative $\frac{\pa}{\pa x_i},\,i=1,2$.
Then, in addition to the usual differential operators $\ggrad:=(\pa_1, \pa_2)^T$ and $\ddiv:=(\pa_1, \pa_2)$,
we define two auxiliary differential operators $\ggrad^{\perp}$ and $\ddiv^{\perp}$ by
$\ggrad^{\perp}:= (-\pa_2, \pa_1)^T$ and $\ddiv^{\perp}:= (-\pa_2, \pa_1)$, respectively. It is easy to deduce the differential indentities $\ddiv^{\perp}\ggrad = \ddiv\ggrad^{\perp}=0$.

Let $\Om\subset\R^2$ be an open and bounded Lipschitz domain such that the exterior $\R^2\ba\ov{\Om}$ of $\Om$ is connected.
Here and throughout the paper we denote $\ov{\Om}$ the closure of the set $\Om$. A confusion with the complex conjugate $\ov{z}$ of $z\in\C$ is not expected.
The propagation of time-harmonic elastic wave equation in an isotropic homogeneous media outside $\Om$ is governed by the reduced Navier equation
\be\label{elastic}
\Delta^{\ast} \mathbf{u}+\omega^2 \mathbf{u}=0\quad \mbox{in}\,\,\R^2\ba\ov{\Om}, \quad \Delta^{\ast}:= \mu\Delta +(\lambda+\mu) \ggrad\ddiv,
\en
where $\mathbf{u}$ denotes the total displacement field and $\omega$ is the circular frequency, $\lambda$ and $\mu$ are Lam$\acute{\mbox{e}}$ constants satisfying
$\mu>0, \,2\mu+\lambda>0$.
It is well known that any solution $\mathbf{u}$ of the Navier equation \eqref{elastic} has a decomposition in the form
\ben
\mathbf{u}=\mathbf{u}_p+\mathbf{u}_s,
\enn
where
\ben
\mathbf{u}_p:=-\frac{1}{k^2_p}\ggrad\ddiv\mathbf{u} \quad\mbox{and}\quad \mathbf{u}_s:=-\frac{1}{k^2_s}\ggrad^{\perp}\ddiv^{\perp}\mathbf{u},
\enn
are known as the compressional (longitudinal) and shear (transversal) parts of $\mathbf{u}$, respectively.
Here, $k_p:=\om/\sqrt{\la+2\mu}$ and $k_s:=\om/\sqrt{\mu}$ denote the compressional wave number and the shear wave number, respectively.
It is clear that
\ben
\Delta \mathbf{u}_p +k_p^2\mathbf{u}_p=0,\quad \ddiv^{\perp}\mathbf{u}_p=0,\quad \Delta \mathbf{u}_s +k_s^2\mathbf{u}_s=0,\quad \ddiv\mathbf{u}_s=0 \quad \mbox{in}\,\,\R^2\ba\ov{\Om}.
\enn

For a rigid body $\Om$, the total displacement field $\mathbf{u}$ satisfies the first (Dirichlet) boundary condition
\be\label{Dirichlet}
\mathbf{u}=0\quad\mbox{on}\,\pa \Om.
\en
For a cavity, we impose the second (Neumann) boundary condition
\be\label{Neumann}
\mathbb{T}_{\nu}\mathbf{u}=0\quad\mbox{on}\,\pa \Om,
\en
where $\mathbb{T}_{\nu}$ denotes the surface traction operator defined by
\ben
\mathbb{T}_{\nu}:=2\mu\nu\cdot\ggrad+\la\nu\ddiv-\mu\nu^{\perp}\ddiv^{\perp}
\enn
in terms of the exterior unit normal vector $\nu$ on $\pa \Om$.

Let $\mathbb{S}:=\{\mathbf{x}\in \R^2:|\mathbf{x}|=1\}$ denote the unit circle in $\R^2$.
In elastic scattering, an important case is the scattering of a plane wave with incident direction $\mathbf{d}\in \mathbb{S}$ which takes the form
\ben
\mathbf{u}^{in}=a_{p}\mathbf{u}_{p}^{in}+a_{s}\mathbf{u}_{s}^{in},\quad a_p,\,a_s\in \C,
\enn
where $\mathbf{u}_{p}^{in}:=\mathbf{d}e^{ik_{p}\mathbf{x}\cdot \mathbf{d}}$ is a plane compressional wave and $\mathbf{u}_{s}^{in}:=\mathbf{d}^{\perp}e^{ik_{s}\mathbf{x}\cdot \mathbf{d}}$ is a plane shear wave, respectively.
The scatterer $\Om$ gives rise to a scattered field $\mathbf{u}^{sc}=\mathbf{u}-\mathbf{u}^{in}$. The scattered field $\mathbf{u}^{sc}$ is a solution to \eqref{elastic}
and has a decomposition $\mathbf{u}^{sc}=\mathbf{u}_{p}^{sc}+\mathbf{u}_{s}^{sc}$ with
$\mathbf{u}_p^{sc}:=-\frac{1}{k^2_p}\ggrad\ddiv\mathbf{u}^{sc}$ and $\mathbf{u}_s^{sc}:=-\frac{1}{k^2_s}\ggrad^{\perp}\ddiv^{\perp}\mathbf{u}^{sc}$.
The scattered field satisfies the Kupradze's radiation conditions
\be\label{KupradzeRC}
\frac{\partial \mathbf{u}_p^{sc}}{\partial r}-ik_p\mathbf{u}_p^{sc}=o(r^{-\frac{1}{2}}),\qquad
\frac{\partial \mathbf{u}_s^{sc}}{\partial r}-ik_s\mathbf{u}_s^{sc}=o(r^{-\frac{1}{2}}),
\en
uniformly in all directions $\hat{\mathbf{x}}\in \mathbb{S}$ as $r:=|\mathbf{x}|\rightarrow\infty$.
For the unique solvability of the scattering problems \eqref{elastic}-\eqref{KupradzeRC} in the space $[H^{1}_{loc}(\R^2\ba\ov{\Om})]^{2}$ we refer to
Kupradze \cite{Kupradze1965} and Li, Wang et al. \cite{LiWangWangZhuo}.

It is well known that every radiating solution to the Navier equation has an asymptotic behaviour of the form
\be\label{usasymptotic}
\mathbf{u}^{sc}(\mathbf{x})
&=&\frac{k_p}{\om}\frac{e^{i\pi/4}}{\sqrt{8\pi\om}}\frac{e^{ik_p|\mathbf{x}|}}{\sqrt{|{\mathbf{x}}|}}u^\infty_p(\hat{\mathbf{x}})\hat{\mathbf{x}}\cr
&&+\frac{k_s}{\om}\frac{e^{i\pi/4}}{\sqrt{8\pi\om}}\frac{e^{ik_s|\mathbf{x}|}}{\sqrt{|{\mathbf{x}}|}}u^\infty_s(\hat{\mathbf{x}}){\mathbf{x}}^{\perp}+O(|\mathbf{x}|^{-3/2}), \quad \mathbf{x}\to \infty
\en
uniformly in all direction $\hat{\mathbf{x}}\in \mathbb{S}$. The functions $u^\infty_p$ and $u^\infty_s$ are known as the compressional and shear far field pattern of $\mathbf{u}^{sc}$,
and are analytic functions on $\mathbb{S}$.
We want to remark here that, to simplify the subsequent representations, the coefficients in \eqref{usasymptotic} are slightly different to those given in \cite{Arens2001, Sevroglou}.
We also have the asymptotic behaviour \cite{Arens2001}
\be\label{Tusasymptotic}
\mathbb{T}_{\mathbf{\hat{x}}}\mathbf{u}^{sc}(\mathbf{x})
&=&i\om\frac{e^{i\pi/4}}{\sqrt{8\pi\om}}\frac{e^{ik_p|\mathbf{x}|}}{\sqrt{|{\mathbf{x}}|}}u^\infty_p(\hat{\mathbf{x}})\hat{\mathbf{x}}\cr
&&+i\om\frac{e^{i\pi/4}}{\sqrt{8\pi\om}}\frac{e^{ik_s|\mathbf{x}|}}{\sqrt{|{\mathbf{x}}|}}u^\infty_s(\hat{\mathbf{x}}){\mathbf{x}}^{\perp}+O(|\mathbf{x}|^{-1}), \quad \mathbf{x}\to \infty,
\en
where $\mathbb{T}_{\mathbf{\hat{x}}}$ denotes the surface traction operator on a circle centered on the origin of radius $|\mathbf{x}|$.

Throughout this paper, we will denote the pair of far field patterns
\ben
[u^\infty_p(\hat{\mathbf{x}}, \mathbf{d}, a_p, a_s);\, u^\infty_s(\hat{\mathbf{x}}, \mathbf{d}, a_p, a_s)]
\enn
of the corresponding scattered field by $\mathbf{u}^\infty(\hat{\mathbf{x}}, \mathbf{d}, a_p, a_s)$,
indicating the dependence on the observation direction $\hat{\mathbf{x}}\in \mathbb{S}$,
the incident direction $\mathbf{d}\in \mathbb{S}$ and the pair of coefficients $(a_p, a_s)\in \C^2$.
We want to remark that the dependence on the coefficients $(a_p, a_s)$ is linear.
For convenience, we use the following notations
\ben
u^\infty_{pp}(\hat{\mathbf{x}}, \mathbf{d}):=u^\infty_{p}(\hat{\mathbf{x}}, \mathbf{d}, 1, 0),\quad u^\infty_{sp}(\hat{\mathbf{x}}, \mathbf{d}):=u^\infty_{p}(\hat{\mathbf{x}}, \mathbf{d}, 0, 1),\\
u^\infty_{ps}(\hat{\mathbf{x}}, \mathbf{d}):=u^\infty_{s}(\hat{\mathbf{x}}, \mathbf{d}, 1, 0),\quad u^\infty_{ss}(\hat{\mathbf{x}}, \mathbf{d}):=u^\infty_{s}(\hat{\mathbf{x}}, \mathbf{d}, 0, 1).\\
\enn
Let $\mathbb{S}_0\subset\mathbb{S}$ be a subset of $\mathbb{S}$ with nonempty interior. Then, we are interested in the following three inverse elastic scattering problems.
\begin{itemize}
  \item {\bf IP-FF}: Determine the location and shape of the scatterer $\Om$ from the knowledge of the full far field pattern $\mathbf{u}^\infty(\hat{\mathbf{x}}, \mathbf{d}, a_p, a_s)$ for all
        observation directions $\hat{\mathbf{x}}\in \mathbb{S}_0\subset \mathbb{S}$, all incident directions $\mathbf{d}\in \mathbb{S}$ and for $(a_p, a_s)=(1,0)$ and $(a_p, a_s)=(0,1)$.
  \item {\bf IP-PP}: Determine the location and shape of the scatterer $\Om$ from the knowledge of the compressional far field pattern $u^\infty_{pp}(\hat{\mathbf{x}}, \mathbf{d})$ for all
        observation directions $\hat{\mathbf{x}}\in \mathbb{S}_0\subset \mathbb{S}$, and all incident directions $\mathbf{d}\in \mathbb{S}$.
  \item {\bf IP-SS}: Determine the location and shape of the scatterer $\Om$ from the knowledge of the shear far field pattern $u^\infty_{ss}(\hat{\mathbf{x}}, \mathbf{d})$ for all
        observation directions $\hat{\mathbf{x}}\in \mathbb{S}_0\subset \mathbb{S}$, and all incident directions $\mathbf{d}\in \mathbb{S}$.
\end{itemize}

%There is already a vast literature on {\bf IP-FF}. We refer to the uniqueness result in H\"{a}hner and Hsiao \cite{HahnerHsiao}, the numerical methods in
%Alves and Kress \cite{AlvesKress2002}, Arens \cite{Arens2001}, Charalambopoulos, Kirsch \& Anagnostopoulos et al. \cite{TYG}, Hu, Kirsch \& Sini \cite{HuKirschSini},
%and Sevroglou \cite{Sevroglou}.
%Clearly, the measurement data in {\bf IP-PP} and {\bf IP-SS} is much less than those in {\bf IP-FF}.
%We refer to Gintides \& Sini \cite{GintidesSini} and Hu, Kirsch \& Sini \cite{HuKirschSini} for the corresponding uniqueness results.
%In \cite{HuKirschSini}, Hu, Kirsch \& Sini also studied the Factorization Method for the above three inverse problems.

Clearly, the measurement data in {\bf IP-PP} and {\bf IP-SS} is much less than those in {\bf IP-FF}.
For uniqueness of {\bf IP-FF}, we refer to H\"{a}hner and Hsiao \cite{HahnerHsiao}, while the corresponding result of {\bf IP-PP} and {\bf IP-SS}
can be found in the recent works by Gintides \& Sini \cite{GintidesSini} and Hu, Kirsch \& Sini \cite{HuKirschSini}.
Due to analyticity, for uniqueness it suffices to know the far field pattern on the subset $\mathbb{S}_0\subset\mathbb{S}$.
However, as one would expect, the quality of the reconstructions decreases drastically for this so called limited-aperture problem. In particular, the traditional
sampling type methods studied in Alves and Kress \cite{AlvesKress2002}, Arens \cite{Arens2001}, Charalambopoulos \& Kirsch et al. \cite{TYG},
Hu, Kirsch \& Sini \cite{HuKirschSini}, and Sevroglou \cite{Sevroglou} fail to work.
In this paper, we will introduce some data retrieval techniques to compute those data that can not be measured directly. Combining this and using the proposed sampling methods,
the limited-aperture problems are desired to be solved partially.

This paper is organized as follows.
In the next section, the theoretical analysis of the proposed reconstruction scheme will be established. A lower bound of
the indicators for the sampling points inside the scatterers is obtained with the help of an inf-criterion characterization.
The decay behavior of the indicators will then be studied for sampling points away from the boundary of the scatterers.
A stability statement will also be established to reflect the important feature of the reconstruction scheme.
For the limited-aperture problems, a data retrieval scheme will be introduced to numerically obtain the far field patterns that can not be measured directly.
One then may combine the data retrieval technique and the previous sampling methods for inverse elastic scattering problems.
Some numerical simulations in two dimensions will be presented in the last section to indicate the efficiency and robustness of the proposed methods.

\section{Novel direct sampling methods and their mathematical basis}

%The plane wave scattering is an important case,  we assume the polarization verctor $\mathbf{P}\in S^{n-1}$, we take the form
%\begin{equation}
%\mathbf{v(x)}= e^{-i k_p \mathbf{d\cdot x}}(\mathbf{d}\cdot \mathbf{P})\mathbf{d}+ e^{-i k_s \mathbf{d\cdot x}}(\mathbf{d}^{\perp}\cdot\mathbf{P})\mathbf{d}^{\perp},
%\end{equation}
%we denote the far-field pairs $(u^\infty_p, u^\infty_s)$ by $\mathbf{u}^\infty(\cdot, \cdot,\mathbf{v})$.

For any $\mathbf{z}\in\R^2$ and any polarization $\mathbf{q}\in\mathbb{S}$ define the two functions
$\phi^{p}_\mathbf{z}\in L^2(\mathbb{S})$ and $\phi^{s}_\mathbf{z}\in L^2(\mathbb{S})$ by
\be\label{phizps}
\phi^{p}_\mathbf{z}(\bbtheta):=e^{-ik_{p}\mathbf{z}\cdot\bbtheta}(\mathbf{q}\cdot\bbtheta)\quad\mbox{and} \quad \phi^{s}_\mathbf{z}(\bbtheta):=e^{-ik_{s}\mathbf{z}\cdot\bbtheta}(\mathbf{q}\cdot\bbtheta^{\perp}),\quad \bbtheta\in\mathbb{S},
\en
respectively.
Then for {\bf IP-FF}, {\bf IP-PP} and {\bf IP-SS} we introduce the indicator functionals
\be
\label{Indicatorff}
&&I_{\bf\rm FF}(\mathbf{z}):=\left|\int_{\mathbb{S}}(\phi^{p}_\mathbf{z}(\mathbf{d}), \phi^{s}_\mathbf{z}(\mathbf{d}))\int_{\mathbb{S}}
\left(
    \begin{array}{cccc}
       u_{pp}^\infty(\mathbf{\hx,d}) \quad  u_{ps}^\infty(\mathbf{\hx,d}) \\
       u_{sp}^\infty(\mathbf{\hx,d}) \quad  u_{ss}^\infty(\mathbf{\hx,d}) \\
      \end{array}
  \right)
\left(
           \begin{array}{c}
             \overline{\phi^{p}_\mathbf{z}(\mathbf{\hx})} \\
             \overline{\phi^{s}_\mathbf{z}(\mathbf{\hx})} \\
           \end{array}
         \right)
ds(\mathbf{\hx})ds(\mathbf{d})\right|,\qquad\\
\label{Indicatorpp}
&&I_{\bf\rm PP}(\mathbf{z}):=\left|\int_{\mathbb{S}}\phi^{p}_\mathbf{z}(\mathbf{d})\int_{\mathbb{S}}u_{pp}^\infty(\mathbf{\hx,d})\ov{\phi^{p}_\mathbf{z}(\mathbf{\hx})}ds(\mathbf{\hx})ds(\mathbf{d})\right|,\\
\label{Indicatorss}
&&I_{\bf\rm SS}(\mathbf{z}):=\left|\int_{\mathbb{S}}\phi^{s}_\mathbf{z}(\mathbf{d})\int_{\mathbb{S}}u_{ss}^\infty(\mathbf{\hx,d})\ov{\phi^{s}_\mathbf{z}(\mathbf{\hx})}ds(\mathbf{\hx})ds(\mathbf{d})\right|,\quad \mathbf{z}\in\R^2,
\en
respectively.

\subsection{Lower bound estimate of $I_{\bf\rm FF}(\mathbf{z})$, $I_{\bf\rm PP}(\mathbf{z})$ and $I_{\bf\rm SS}(\mathbf{z})$ for $\mathbf{z}\in\Om$}

We will use the short-hand $\mathcal{L}^2:=[L^2(\mathbb{S})]^2$.
For any $\mathbf{g}\in \mathcal{L}^2$, we adopt the notation $\mathbf{g}=(g_p; g_s)$ with
\ben
g_p(\mathbf{d}):=\mathbf{g(d)\cdot d}\quad\mbox{and}\quad g_s(\mathbf{d}):=\mathbf{g(d)\cdot d^{\perp}},\quad  \mathbf{d}\in\mathbb{S}.
\enn
The Hilbert space $\mathcal{L}^2$ will be equipped with the inner product
\be\label{innerproduct}
(\mathbf{g},\mathbf{h})_{\mathbb{S}}
:= \int_{\mathbb{S}}g_p(\mathbf{d})\ov{h_p(\mathbf{d})}ds(\mathbf{d})+ \int_{\mathbb{S}}g_s(\mathbf{d})\ov{h_s(\mathbf{d})}ds(\mathbf{d}),
\quad \mathbf{g},\mathbf{h}\in \mathcal{L}^2.
\en
It is clear that there holds the decomposition
\ben
\mathbf{g(d)=g_p(d)+g_s(d)}, \quad \mathbf{g_p(d)}:=g_p(\mathbf{d})\mathbf{d},\quad \mathbf{g_s(d)}:=g_s(\mathbf{d})\mathbf{d}^{\perp},
\enn
where $\mathbf{g_p(d)}\in \mathcal{L}^2_p:=\{\mathbf{g_p}\in\mathcal{L}^2: \mathbf{g_p(d)\cdot d^{\perp}}=0,\,\mathbf{d}\in \mathbb{S}\}$
and $\mathbf{g_s(d)}\in \mathcal{L}^2_s:=\{\mathbf{g_s}\in\mathcal{L}^2: \mathbf{g_s(d)\cdot d}=0,\, \mathbf{d}\in \mathbb{S}\}$.
For later use, we introduce the orthogonal projection operators $\mathbf{\mathcal{P}_p}: \mathcal{L}^2\to \mathcal{L}^2_p$ and $\mathbf{\mathcal{P}_s}: \mathcal{L}^2\to \mathcal{L}^2_s$,
i.e., for all $\mathbf{g(d)}\in \mathcal{L}^2$, $\mathbf{\mathcal{P}_p} (\mathbf{g(d)})=\mathbf{g_p(d)}$ and $\mathbf{\mathcal{P}_s}(\mathbf{g(d)})=\mathbf{g_s(d)}$.
Their adjoint operators $\mathbf{\mathcal{P}^{\ast}_p}: \mathcal{L}^2_p\to \mathcal{L}^2$ and $\mathbf{\mathcal{P}^{\ast}_s}: \mathcal{L}^2_s\to \mathcal{L}^2$ are just the inclusions
from $\mathcal{L}^2_p$ and $\mathcal{L}^2_s$, respectively, to $\mathcal{L}^2$.

%%%%%%%%%%
% 上面Herglotz函数定义不同于Tilo文章。Tilo文章使得后面远场分解有问题，见Tilo文章 Theorem 4.3
%%%%%%%%%%%%%%%%%%

Consider the elastic Herglotz wave function
\be\label{eHerglotz}
\mathbf{v_g(x)}:=\int_{\mathbb{S}}\left\{e^{ik_p\mathbf{x}\cdot\mathbf{d}}g_p(\mathbf{d})\mathbf{d}+e^{ik_s\mathbf{x}\cdot\mathbf{d}}g_s(\mathbf{d})\mathbf{d}^{\perp}\right\}ds(\mathbf{d}),\quad \mathbf{x}\in \R^2,
\en
with a vector Herglotz kernel $\mathbf{g}=(g_p; g_s)\in \mathcal{L}^2$.
We now introduce the far field operator $\mathbf{F}:\mathcal{L}^2\rightarrow \mathcal{L}^2$ defined by
\be\label{ffoperator}
\mathbf{(Fg)(\hx)}
&:=&\int_{\mathbb{S}}\mathbf{u}^{\infty}\left(\mathbf{\hx},\mathbf{d},g_p(\mathbf{d}), g_s(\mathbf{d})\right)ds(\mathbf{d})\cr
&=&\int_{\mathbb{S}}
\left(
    \begin{array}{cccc}
      u_{pp}^\infty(\mathbf{\hx,d}) \quad u_{sp}^\infty(\mathbf{\hx,d}) \\
      u_{ps}^\infty(\mathbf{\hx,d}) \quad u_{ss}^\infty(\mathbf{\hx,d}) \\
      \end{array}
  \right)\left(
           \begin{array}{c}
             g_p(\mathbf{d}) \\
             g_s(\mathbf{d}) \\
           \end{array}
         \right)
ds(\mathbf{d}),
\en
i.e., $\mathbf{Fg}$ is the far field pattern for the scattering of elastic Herglotz wave function with kernel $\mathbf{g}$.
The far field operator $\mathbf{F}$ plays an essential role in the investigations of the sampling type methods for inverse problems,
we refer to \cite{Arens2001} for a survey on the state of the art of its properties and applications.

For any $\mathbf{z}\in\R^2$, recall the two test functions $\phi^{p}_\mathbf{z}\in L^2(\mathbb{S})$ and $\phi^{s}_\mathbf{z}\in L^2(\mathbb{S})$, we define
$\bbphi_\mathbf{z}\in \mathcal{L}^2$ by
\be\label{phiz}
\mathbf{\bbphi_z}(\bbtheta):=(\phi^{p}_\mathbf{z}; \phi^{s}_\mathbf{z})=e^{-ik_p\mathbf{z}\cdot\bbtheta}(\mathbf{q}\cdot\bbtheta)\bbtheta+e^{-ik_s\mathbf{z}\cdot\bbtheta}(\mathbf{q}\cdot\bbtheta^{\perp})\bbtheta^{\perp},\quad \bbtheta\in \mathbb{S}.
\en
By interchanging orders of integration, we may rewrite our indicator $I_{\rm FF}$ given by (\ref{Indicatorff}) in a very simple form
\be\label{IndLiunewformFF}
I_{\bf\rm FF}(\mathbf{z})=|( \mathbf{F\bbphi_z}, \mathbf{\bbphi_z})_{\mathbb{S}}|,\quad \mathbf{z}\in\R^2,
\en
where $(\cdot, \cdot)_{\mathbb{S}}$ is the inner product of the space $\mathcal{L}^2$ given in \eqref{innerproduct}.
Similarly, by defining  $\mathbf{F_p}:=\mathbf{\mathcal{P}_pF\mathcal{P}^\ast_p}$ and $\mathbf{F_s}:=\mathbf{\mathcal{P}_sF\mathcal{P}^\ast_s}$,
we found that the indicators  $I_{\bf\rm PP}$ and $I_{\bf\rm SS}$ given by (\ref{Indicatorpp}) and (\ref{Indicatorss}) can be written as a very simple form,
\be\label{IndLiunewformPPSS}
I_{\bf\rm PP}(\mathbf{z}):= \Big|( \mathbf{F_p\bbphi^p_z}, \mathbf{\bbphi^p_z})_{\mathbb{S}}\Big| \quad\mbox{and}\quad
I_{\bf\rm SS}(\mathbf{z}):= \Big|( \mathbf{F_{s}\bbphi^s_z}, \mathbf{\bbphi^s_z})_{\mathbb{S}}\Big|, \quad \mathbf{z}\in\R^2,
\en
with $\mathbf{\bbphi^p_z}:=(\phi^{p}_\mathbf{z}; 0)$ and $\mathbf{\bbphi^s_z}:=(0; \phi^{s}_\mathbf{z})$, respectively.

Recall the Green's tensor
\ben
\mathbf{\Phi(x,y)}
&=&\frac{i}{4\mu}H_0^{(1)}(k_s\vert \mathbf{x-y} \vert)I\cr
&&+\frac{i}{4\om^2}\grad_{\mathbf{x}}\grad_{\mathbf{x}}^{T}(H_0^{(1)}(k_s\vert\mathbf{x-y} \vert)-H_0^{(1)}(k_p\vert\mathbf{x-y}\vert)), \quad \mathbf{x\neq y}
\enn
of the Navier equation in $\R^2$ in terms of the identity matrix $I$ and the Hankel function $H^{(1)}_0$ of the first kind of order zero.
For any $\mathbf{y}\in\R^2$ and any polarization $\mathbf{q}\in\mathbb{S}$, an elastic point source with source point $\mathbf{y}$ and polarization $\mathbf{q}$ is given by
$\mathbf{\Phi(x,y)q},\,\mathbf{x}\in\R^2\ba\{\mathbf{y}\}$.
From the asymptotics for the Hankel function $H_0^{(1)}$  it follows that the far field pattern $\mathbf{\Phi^{\infty}(\cdot,y; q)}$
of the point source $\mathbf{\Phi(x,y)q},\,\mathbf{x}\in\R^2\ba\{\mathbf{y}\}$ is given by
\be
\mathbf{\Phi}^\infty_{p}(\mathbf{\hx, y; q})&=&  e^{-ik_p\mathbf{\hx\cdot y}}(\mathbf{q}\cdot\mathbf{\hx}) = \phi^{p}_\mathbf{z}\mathbf{(\hx)},\quad\mathbf{\hx}\in \mathbb{S},\\
\mathbf{\Phi}^\infty_{s}(\mathbf{\hx, y; q})&=&  e^{-ik_s\mathbf{\hx\cdot y}}(\mathbf{q}\cdot\mathbf{\hx}^{\perp})= \phi^{s}_\mathbf{z}\mathbf{(\hx)},\quad \mathbf{\hx}\in \mathbb{S}.
\en
Now we introduce the elastic single-layer operator $\mathbf{S}: [H^{-1/2}(\pa\Om)]^2 \rightarrow [H^{1/2}(\pa\Om)]^2$, given by
\be\label{singlelayerpotential}
(\mathbf{S}\mathbf{\bbphi})(\mathbf{x}):=\int_{\pa\Om}\mathbf{\Phi(x,y)\bbphi(y)}ds(\mathbf{y}),\quad \mathbf{x}\in\pa\Om.
\en
Then the following property of the single-layer operator $\mathbf{S}$ is important for our subsequent analysis.

\begin{lemma}\label{Scoercive}
Assume that $\om^2$ is not a Dirichlet eigenvalue of $-\Delta^{\ast}$ in $\Om$. Then there exist $c>0$ with
\ben
|\langle \bbphi, \mathbf{S}\bbphi \rangle| \geq c\|\phi\|^2_{[H^{-1/2}(\pa\Om)]^2},\quad \bbphi\in [H^{-1/2}(\pa\Om)]^2,
\enn
where $\langle\cdot, \cdot\rangle$ denotes the duality pairing in $\langle[H^{-1/2}(\pa\Om)]^2 , [H^{1/2}(\pa\Om)]^2\rangle$.
\end{lemma}
\begin{proof}
This property follows immediately from Lemma 1.17 of \cite{KirschGrinberg} in combination with Lemma 4.2 of \cite{Arens2001}.
\end{proof}

Define the data-to-pattern operator $\mathbf{G} : [H^{1/2}(\pa\Om)]^2 \rightarrow \mathcal{L}^2$ by
\be
\mathbf{Gh}:= \mathbf{u}^\infty,
\en
where $\mathbf{u}^\infty$ is the far field pattern of the solution to the Dirichlet boundary value problem
with boundary data $\mathbf{h}$. By a standard argument, we have that the data-to-pattern operator $\mathbf{G} : [H^{1/2}(\pa\Om)]^2 \rightarrow \mathcal{L}^2$ is compact, injective with dense range
in $\mathcal{L}^2$.
To characterize the scatterer $\Om$ by the corresponding date-to-pattern operator, we have the following lemma.
\begin{lemma}\label{Gproperty}
For any $\mathbf{z}\in\R^2$ and $\mathbf{q}\in\mathbb{S}$, define the function $\mathbf{\bbphi_z}$ by \eqref{phiz}. Then the following holds.
\begin{itemize}
  \item $\mathbf{z}\in\Om$ if and only if\, $\mathbf{\bbphi_z}(\mathbf{\hx; q})\in\mathcal {R}(\mathbf{G})$.
  \item $\mathbf{z}\in\Om$ if and only if\, $\mathbf{\bbphi_z^{p}}(\mathbf{\hx; q})\in\mathcal {R}(\mathbf{\mathcal{P}_pG})$.
  \item $\mathbf{z}\in\Om$ if and only if\, $\mathbf{\bbphi_z^{s}}(\mathbf{\hx; q})\in\mathcal {R}(\mathbf{\mathcal{P}_sG})$.
\end{itemize}
\end{lemma}
\begin{proof}
The first result has been proved by Arens in Theorem 4.7 of \cite{Arens2001}. The other two can be done analogously in combination with the definitions of the projection operators.
\end{proof}

We introduce the Hergoltz wave operator $\mathbf{H}: \mathcal{L}^2 \rightarrow [H^{1/2}(\pa\Om)]^2$ by setting
\ben
(\mathbf{H}\mathbf{g})(\mathbf{x}):=\int_{\mathbb{S}}\left\{e^{ik_p\mathbf{x}\cdot\mathbf{d}}g_p(\mathbf{d})\mathbf{d}+e^{ik_s\mathbf{x}\cdot\mathbf{d}}g_s(\mathbf{d})\mathbf{d}^{\perp}\right\}ds(\mathbf{d}),\quad \mathbf{x}\in \pa\Om.
\enn
$\mathbf{H}\mathbf{g}$  is the trace on $\pa\Om$ of the elastic Herglotz wave function \eqref{eHerglotz} with vector Herglotz kernel $\mathbf{g}=(g_p; g_s)\in \mathcal{L}^2$.
Its adjoint $\mathbf{H}^{\ast}: [H^{-1/2}(\pa\Om)]^2 \rightarrow \mathcal{L}^2$ is given by
\ben
(\mathbf{H}^{\ast}\bbphi)(\mathbf{d}):=\left(\int_{\pa\Om}e^{-ik_p\mathbf{x}\cdot\mathbf{d}}\mathbf{d}\cdot\bbphi(\mathbf{x})ds(\mathbf{x}); \int_{\pa\Om}e^{-ik_s\mathbf{x}\cdot\mathbf{d}}\mathbf{d}^{\perp}\cdot\bbphi(\mathbf{x})ds(\mathbf{x})\right),\quad \mathbf{d}\in \mathbb{S}.
\enn
We note that $\mathbf{H}^{\ast}\bbphi$ is exactly the far field pattern of the elastic single-layer potential $\mathbf{S\bbphi}$, hence
\ben
\mathbf{H}^{\ast} = \mathbf{GS} \quad\mbox{and therefore}\quad \mathbf{H} = \mathbf{S^{\ast}G^{\ast}},
\enn
where $\mathbf{G}^*: \mathcal{L}^2\rightarrow [H^{-1/2}(\partial \Omega)]^2$ and $\mathbf{S}^*: [H^{-1/2}(\partial \Omega)]^2\rightarrow [H^{1/2}(\partial \Omega)]^2$ denote the adjoints of
$\mathbf{G}$ and $\mathbf{S}$, respectively.
On the other hand, we observe that $\mathbf{Fg}$ is the far field pattern of the solution of the exterior Dirichlet problem with boundary data $-(\mathbf{H}\mathbf{g})|_{\pa\Om}$,
which implies $\mathbf{F=-GH}$. Combining the previous operator equality, we deduce the factorization of the far field operator
\be\label{ffactorization}
\mathbf{F}=-\mathbf{GS^{\ast}G^\ast},
\en
Consequently, we have the factorizations of the operators $\mathbf{F}_p$ and $\mathbf{F}_s$,
\be\label{ppssfactorization}
\mathbf{F}_p=-\mathbf{(\mathcal{P}_pG)S^{\ast}(\mathcal{P}_p G)^\ast},\quad
\mathbf{F}_s=-\mathbf{(\mathcal{P}_sG)S^{\ast}(\mathcal{P}_s G)^\ast}.
\en

For all $\mathbf{z}\in \R^2$ and polarization $\mathbf{q}\in\mathbb{S}$, define a subspace $\mathbf{A_z}\subset \mathcal{L}^2$ by
\ben
\mathbf{A_z}:=\{\mathbf{\bbpsi}\in \mathcal{L}^2:\,(\mathbf{\bbpsi }, \mathbf{\bbphi_z })_{\mathbb{S}}=1\},
\enn
where $\mathbf{\bbphi_z}$ is the test function given in \eqref{phiz} and $(\cdot, \cdot)_{\mathbb{S}}$ is the inner product of $\mathcal{L}^2$ defined by \eqref{innerproduct}.
Now we are in a position to state the main result of this subsection.

\begin{theorem}\label{theorem1} Consider the inverse elastic scattering by a rigid body $\Om$. Assume that $\om^2$ is not a Dirichlet eigenvalue of $-\Delta^{\ast}$ in $\Om$.
Let $\mathbf{q}\in\mathbb{S}$ be the polarization.
Then $\mathbf{z}\in\Omega$, if and only if,
\be\label{infcreterior}
\inf\{\mathbf{|(F\bbpsi, \bbpsi)_{\mathbb{S}}|:\,\bbpsi\in A_z}\}>0.
\en
Furthermore, we have the lower bound estimate
\be\label{estimateff}
I_{\bf\rm FF}(\mathbf{z})\geq \frac{c}{\|\mathbf{\Phi(\cdot,z)q}\|^2_{[H^{1/2}(\partial\Omega)]^2}},\quad \mathbf{z}\in\Omega,
\en
for some constant $c>0$ which is independent of $\mathbf{z}$.
Similar, we have
\be\label{estimateppss}
I_{\bf\rm PP}(\mathbf{z})\geq \frac{c}{\|\mathbf{\Phi(\cdot,z)q}\|^2_{[H^{1/2}(\partial\Omega)]^2}}\quad\mbox{and}\quad
I_{\bf\rm SS}(\mathbf{z})\geq \frac{c}{\|\mathbf{\Phi(\cdot,z)q}\|^2_{[H^{1/2}(\partial\Omega)]^2}},\quad \mathbf{z}\in\Omega,\qquad
\en
for some constant $c>0$ which is independent of $\mathbf{z}$.
\end{theorem}
\begin{proof}
First, \eqref{infcreterior} follows directly by applying Theorem 1.16 in \cite{KirschGrinberg} to the factorization \eqref{ffactorization}, Lemma \ref{Gproperty} and noting the fact that the operator $\mathbf{S}$
is coercive by Lemma \ref{Scoercive}.
Furthermore, note that $\mathbf{\bbphi_z}(\mathbf{\hx; q})=G(-\Phi(\cdot,z)q|_{\pa\Om})$, using Theorem 1.16 of \cite{KirschGrinberg} again we deduce that
\ben
\inf\{\mathbf{|(F\bbpsi, \bbpsi)_{\mathbb{S}}|:\,\bbpsi\in A_z}\}\geq \frac{c}{\|\mathbf{\Phi(\cdot,z)q}\|^2_{[H^{1/2}(\partial\Omega)]^2}}
\enn
for some constant $c>0$ which is independent of $\mathbf{z}$.
Second, we observe that
\be\label{Phiphi}
(\mathbf{\bbphi_z}, \mathbf{\bbphi_z})_{\mathbb{S}}=2\pi,
\en
which implies $\mathbf{\bbpsi_0}:=\frac{1}{2\pi}\mathbf{\bbphi_z}\in \mathbf{A_z}$. From this we derive that
\ben
I_{\bf\rm FF}(\mathbf{z})
&=& \mathbf{|(F\bbphi_z, \bbphi_z)_{\mathbb{S}}|}\cr
&=& 4\pi^2\mathbf{|(F\bbpsi_0, \bbpsi_0)_{\mathbb{S}}|}\cr
&\geq& 4\pi^2\inf\{\mathbf{|(F\bbpsi, \bbpsi)_{\mathbb{S}}|:\,\bbpsi\in A_z}\}\cr
&\geq& \frac{c}{\|\mathbf{\Phi(\cdot,z)q}\|^2_{[H^{1/2}(\partial\Omega)]^2}}
\enn
for some constant $c>0$ which is independent of $\mathbf{z}$.
Thus, the lower bound estimate \eqref{estimateff} for the indicator $I_{\bf\rm FF}$ in $\Om$ has been proved. The other two lower bound estimates in \eqref{estimateppss} can be shown
analogously using the factorizations \eqref{ppssfactorization}.
\end{proof}

We conclude this subsection by a remark that the analogous result of Theorem \ref{theorem1} holds for the Neumann boundary condition. Our analysis rely on the factorization of the
far field operator. The data-to-pattern operator $\mathbf{G} : [H^{-1/2}(\pa\Om)]^2 \rightarrow \mathcal{L}^2$ is now defined to map $\mathbf{h}\in[H^{-1/2}(\pa\Om)]^2$ into the
far field pattern $\mathbf{u}^\infty=\mathbf{Gh}$ of the exterior Neumann boundary value problem with boundary data $\mathbf{h}$.
We introduce the hypersingular integral operator $\mathbf{N}: [H^{1/2}(\pa\Om)]^2 \rightarrow [H^{-1/2}(\pa\Om)]^2 $ by
\ben
\mathbf{(N\bbphi)(x)}:=\mathbb{T}_{\mathbf{\nu(x)}}\int_{\pa\Om}(\mathbb{T}_{\mathbf{\nu(y)}}\mathbf{\Phi(x,y)})^{T}\bbphi(\mathbf{y})ds(\mathbf{y}),\quad \mathbf{x}\in \pa\Om,
\enn
for $\mathbf{\bbphi}\in [H^{1/2}(\pa\Om)]^2$. Then one can derive the factorization of the far field operator for Neumann problem in the form
\ben
\mathbf{F}=-\mathbf{GN^{\ast}G^{\ast}}.
\enn
Based on this, the analogous results of Lemmas \ref{Scoercive}, \ref{Gproperty} and Theorem \ref{theorem1} can be derived.

Finally, we want to remark that the assumption on interior eigenvalues in Lemma \ref{Scoercive} and Theorem \ref{theorem1} has only its theory interest.
It is well-known that the classical sampling type methods for solving inverse scattering
problems fail if the wave number is an eigenvalue of a corresponding interior eigenvalue problem.
We refer to a modification proposed by Kirsch \& Liu\cite{KirschLiuIP14avoid} to avoid the eigenvalues.
However, our sampling methods are independent of the interior eigenvalues from the numerical point of view.

\subsection{Indicator behavior for the sampling points away from the boundary $\partial \Omega$}

We have known from the previous subsection that the values of the indicator functionals can not be small arbitrarily for sampling points inside.
In this subsection, we study the behavior of the indicator functionals for sampling points away from the boundary.

To simplify the subsequent representations, we introduce
\ben
\mathbf{J(w)}:=\frac{\mathbf{ww}^{T}}{|\mathbf{w}|^2},\quad \mathbf{w}\in\R^2.
\enn
Straightforward calculations show that $\mathbf{I-J(w)}=\mathbf{J(w^{\perp})}$.
For the scattering problem (\ref{elastic})-\eqref{KupradzeRC}, the far-field pattern $(u^\infty_p,u^\infty_s)$ has the following representation \cite{Arens2001}
\be
\label{upinfty}u^\infty_p\mathbf{(\hat{x},d)\hat{x}}&=& \int_{\pa\Om}\Big\{[\mathbb{T}_{\mathbf{\nu(y)}}\mathbf{J(\hat{x})}e^{-ik_p\mathbf{\hat{x}\cdot y}}]^T\mathbf{u}^{sc}\mathbf{(y,d)}\cr
&&\qquad\qquad\qquad\quad -\mathbf{J(\hat{x})}e^{-ik_p\mathbf{\hat{x}\cdot y}}\mathbb{T}_{\mathbf{\nu(y)}}\mathbf{u}^{sc}\mathbf{(y,d)}\Big\}ds(\mathbf{y}),\\
\label{usinfty}u^\infty_s\mathbf{(\hat{x},d)\hat{x}^{\perp}}&=& \int_{\pa\Om}\Big\{[\mathbb{T}_{\mathbf{\nu(y)}}\mathbf{J(\hat{x}^{\perp})}e^{-ik_s\mathbf{\hat{x}\cdot y}}]^T\mathbf{u}^{sc}\mathbf{(y,d)}\cr
&&\qquad\qquad\qquad\quad -\mathbf{J(\hat{x}^{\perp})}e^{-ik_s\mathbf{\hat{x}\cdot y}}\mathbb{T}_{\mathbf{\nu(y)}}\mathbf{u}^{sc}\mathbf{(y,d)}\Big\}ds(\mathbf{y}).
\en
By straightforward calculations, it can be seen that, for all $\mathbf{\hat{x}}\in \mathbb{S},\,y\in \pa\Om$,
\ben
[\mathbb{T}_{\mathbf{\nu(y)}}(\mathbf{J(\hat{x})}e^{-ik_p\mathbf{\hat{x}\cdot y}})]\mathbf{\hat{x}}=\mathbb{T}_{\mathbf{\nu(y)}}(\mathbf{\hat{x}}e^{-ik_p\mathbf{\hat{x}\cdot y}})
\enn
and
\ben
[\mathbb{T}_{\mathbf{\nu(y)}}(\mathbf{J(\hat{x}^{\perp})}e^{-ik_p\mathbf{\hat{x}\cdot y}})]\mathbf{\hat{x}^{\perp}}=\mathbb{T}_{\mathbf{\nu(y)}}(\mathbf{\hat{x}^{\perp}}e^{-ik_p\mathbf{\hat{x}\cdot y}}).
\enn
Inserting this into \eqref{upinfty}-\eqref{usinfty}, we find that
\be
\label{upinfty2}u^\infty_p\mathbf{(\hat{x},d)}&=& \int_{\pa\Om}\Big\{\mathbf{u}_{p}^{in}(\mathbf{y,-\hat{x}})\cdot\mathbb{T}_{\mathbf{\nu(y)}}\mathbf{u}^{sc}\mathbf{(y,d)}\cr
&&\qquad\qquad\qquad\quad -[\mathbb{T}_{\mathbf{\nu(y)}}\mathbf{u}_{p}^{in}(\mathbf{y,-\hat{x}})]\cdot\mathbf{u}^{sc}\mathbf{(y,d)}\Big\}ds(\mathbf{y}),\\
\label{usinfty2}u^\infty_s\mathbf{(\hat{x},d)}&=& \int_{\pa\Om}\Big\{\mathbf{u}_{s}^{in}(\mathbf{y,-\hat{x}})\cdot\mathbb{T}_{\mathbf{\nu(y)}}\mathbf{u}^{sc}\mathbf{(y,d)}\cr
&&\qquad\qquad\qquad\quad -[\mathbb{T}_{\mathbf{\nu(y)}}\mathbf{u}_{s}^{in}(\mathbf{y,-\hat{x}})]\cdot\mathbf{u}^{sc}\mathbf{(y,d)}\Big\}ds(\mathbf{y}).
\en

Now we introduce the following auxiliary functions
\ben
G_{pp}\mathbf{(z,d)}:=\int_{\mathbb{S}}u^\infty_{pp}\mathbf{(\hat{x},d)}\ov{\phi^{p}_\mathbf{z}(\mathbf{\hat{x}})}ds(\mathbf{\hat{x}}),\quad
G_{sp}\mathbf{(z,d)}:=\int_{\mathbb{S}}u^\infty_{sp}\mathbf{(\hat{x},d)}\ov{\phi^{p}_\mathbf{z}(\mathbf{\hat{x}})}ds(\mathbf{\hat{x}}),\\
G_{ps}\mathbf{(z,d)}:=\int_{\mathbb{S}}u^\infty_{ps}\mathbf{(\hat{x},d)}\ov{\phi^{s}_\mathbf{z}(\mathbf{\hat{x}})}ds(\mathbf{\hat{x}}),\quad
G_{ss}\mathbf{(z,d)}:=\int_{\mathbb{S}}u^\infty_{ss}\mathbf{(\hat{x},d)}\ov{\phi^{s}_\mathbf{z}(\mathbf{\hat{x}})}ds(\mathbf{\hat{x}}).
\enn
Then the indicators $I_{\bf\rm FF}$, $I_{\bf\rm PP}$ and $I_{\bf\rm SS}$ involve the terms
\be
\label{G1} \int_{\mathbb{S}}G_{pp}\mathbf{(z,d)}\phi^{p}_\mathbf{z}(\mathbf{d})ds(\mathbf{d}), \quad
\int_{\mathbb{S}}G_{ps}\mathbf{(z,d)}\phi^{p}_\mathbf{z}(\mathbf{d})ds(\mathbf{d}), \\
\int_{\mathbb{S}}G_{sp}\mathbf{(z,d)}\phi^{s}_\mathbf{z}(\mathbf{d})ds(\mathbf{d}),\quad
\int_{\mathbb{S}}G_{ss}\mathbf{(z,d)}\phi^{s}_\mathbf{z}(\mathbf{d})ds(\mathbf{d}).\label{G2}
\en
Noting that $G_{pp}, G_{ps},G_{sp}$ and $G_{ss}\in L^2(\mathbb{S})$, by the well known Riemann-Lebesgue Lemma, we obtain that all the expressions in (\ref{G1})-(\ref{G2}) go to $0$
as $|\mathbf{z}|\rightarrow\infty$, so we have
\be
I_{\bf\rm FF}(\mathbf{z})\rightarrow 0, \quad I_{\bf\rm PP}(\mathbf{z})\rightarrow 0,\quad\mbox{and}\quad I_{\bf\rm SS}(\mathbf{z})\rightarrow 0 \quad \mbox{as}\,\, \mathbf{z}\rightarrow \infty.
\en

For the behaviors of the indicators as the sampling points away from the boundary of the scatterers we proceed by recalling the spherical harmonics
\be\label{sphericalharmonic}
Y_\alpha^\beta(\mathbf{\hx}):=\g e^{i\beta\varphi}, \quad \mathbf{\hx}:=(\cos\varphi, \sin\varphi)^{T}\in \mathbb{S},\,\, \g:=\sqrt{\frac{1}{2\pi}},
\en
for $\beta=\pm\alpha, \alpha=0,1,2,\cdots,$ which form a complete orthonormal system in $L^2(\mathbb{S})$.
In particular, we are interested in the spherical harmonics $Y_\alpha^\beta(\mathbf{\hx})$ of order $\al=0,1,2,3$ with $\mathbf{\hx}=(\hat{x}_1, \hat{x}_2)^{T}\in\mathbb{S}$.
\be\label{sphhar}
&&Y^0_0(\mathbf{\hx})=\g ,\quad
Y^{-1}_1(\mathbf{\hx})=\g (\hat{x}_1-i\hat{x}_2),\quad
Y^{1}_1(\mathbf{\hx})=\g (\hat{x}_1+i\hat{x}_2),\\
&&Y^{-2}_2(\mathbf{\hx})=\g (\hat{x}_1^2-\hat{x}_2^2-2i\hat{x}_1\hat{x}_2),\quad Y^{2}_2(\mathbf{\hx})=\g (\hat{x}_1^2-\hat{x}_2^2+2i\hat{x}_1\hat{x}_2),\\
&&Y^{-3}_3(\mathbf{\hx})=\g (4\hat{x}_1^3-3\hat{x}_1-3i\hat{x}_2+4i\hat{x}_2^3),\quad
Y^{3}_3(\mathbf{\hx})=\g (4\hat{x}_1^3-3\hat{x}_1+3i\hat{x}_2-4i\hat{x}_2^3).\qquad\label{sphhar1}
\en
Using (\ref{sphhar})-(\ref{sphhar1}) and the fact $\hat{x}_1^2+\hat{x}_2^2=1$, we deduce that
\ben
&&\hat{x}_1=\frac{Y^{1}_1(\mathbf{\hx})+Y^{-1}_1(\mathbf{\hx})}{2\g},\quad \hat{x}_2=\frac{Y^{1}_1(\mathbf{\hx})-Y^{-1}_1(\mathbf{\hx})}{2i\g},\\
&&\hat{x}_1^2=\frac{1}{2}+\frac{Y^{2}_2(\mathbf{\hx})+Y^{-2}_2(\mathbf{\hx})}{4\g},\quad
\hat{x}_1\hat{x}_2=\frac{Y^{2}_2(\mathbf{\hx})-Y^{-2}_2(\mathbf{\hx})}{4i\g},\\
&&\hat{x}_2^2=\frac{1}{2}-\frac{Y^{2}_2(\mathbf{\hx})+Y^{-2}_2(\mathbf{\hx})}{4\g},\\
&&\hat{x}_1^3=\frac{Y^{3}_3(\mathbf{\hx})+Y^{-3}_3(\mathbf{\hx})}{8\g}+\frac{3(Y^{1}_1(\mathbf{\hx})+Y^{-1}_1(\mathbf{\hx}))}{8\g},\\
&&\hat{x}_1^2\hat{x}_2=\frac{Y^{1}_1(\mathbf{\hx})-Y^{-1}_1(\mathbf{\hx})}{8i\g}+\frac{Y^{3}_3(\mathbf{\hx})-Y^{-3}_3(\mathbf{\hx})}{8i\g},\\
&&\hat{x}_1\hat{x}_2^2=\frac{Y^{1}_1(\mathbf{\hx})+Y^{-1}_1(\mathbf{\hx})}{8\g}-\frac{Y^{3}_3(\mathbf{\hx})+Y^{-3}_3(\mathbf{\hx})}{8\g},\\
&&\hat{x}_2^3=\frac{3(Y^{1}_1(\mathbf{\hx})-Y^{-1}_1(\mathbf{\hx}))}{8i\g}-\frac{Y^{3}_3(\mathbf{\hx})-Y^{-3}_3(\mathbf{\hx})}{8i\g}.
\enn
Using further the Funk-Hecke formula \cite{LiuIP17}
\be\label{FunkHeckformula}
\int_{\mathbb{S}}e^{-ik\mathbf{z}\cdot\mathbf{\hx} }Y_\alpha^\beta(\mathbf{\hx})ds(\mathbf{\hx})=\frac{2\pi}{i^\alpha}J_{\alpha}(k|\mathbf{z}|)Y_\alpha^\beta(\mathbf{\hat{z}}),
\en
with $J_\alpha$ being the Bessel functions of order $\alpha$, we found that
\ben
\int_{\mathbb{S}}e^{-ik\mathbf{z\cdot\hx}}\hat{x}^m_1\hat{x}^n_2ds(\mathbf{\hx})=\sum_{\alpha\leq 3}[a^{m,n}_{\alpha}Y_\alpha^\alpha(\mathbf{\hat{z}})+b^{m,n}_{\alpha}Y_\alpha^{-\alpha}(\mathbf{\hat{z}})]J_{\alpha}(k|\mathbf{z}|), m+n\leq 3, m,n\in \mathbb{N},
\enn
for some constants $a^{m,n}_{\alpha},\, b^{m,n}_{\alpha}\in \C$.
Furthermore, using the far field representations \eqref{upinfty2}-\eqref{usinfty2} and interchanging the order of integration,
we deduce
\ben
G_{\rm ii}\mathbf{(z,d)}=\sum_{\alpha\leq 3}c_{\alpha}J_{\alpha}(k|\mathbf{y-z}|),\quad \mathbf{z}\in \R^2, \quad \mathbf{y}\in\pa\Om, \quad {\rm ii}=pp, ps, sp, ss.
\enn
This implies that the four auxiliary functions $G_{\rm ii}$ are superpositions of the Bessel functions $J_{\alpha},\alpha=0,1,2,3$, although
we do not know the coefficients $c_{\alpha}$ which depend on the total fields $\mathbf{u}$ on $\pa\Om$.
Thus, we expect that the indicators $I_{\bf\rm FF}$, $I_{\bf\rm PP}$ and $I_{\bf\rm SS}$ decay like the Bessel functions as the sampling point $\mathbf{z}$ goes away from the boundary $\pa\Om$.
Combining the lower bound estimates given in the previous subsection, we expect that {\em the indicators take their maximum near the boundary of the underlying target and decay
like the Bessel functions as the sampling points go away from the boundary.}
This phenomenon will be verified further in a lot of numerical simulations in Section 3.

We end this subsection by a stability statement, which reflects an important feature of the reconstruction scheme.

\begin{theorem}{\rm (Stability statement).}\label{stability}
\be\label{stabilityestimate}
I_{\bf\rm ii}(\mathbf{z})-I^{\delta}_{\bf\rm ii}(\mathbf{z})\leq c\|\mathbf{u}^{\infty}-\mathbf{u}_{\delta}^{\infty}\|_{\mathcal{L}^2},\quad \mathbf{z}\in\R^2,\quad {\bf\rm ii=FF,PP,SS},
\en
where $I^{\delta}_{\bf\rm ii}(\mathbf{z})$ is the indicator functional with $\mathbf{u}^{\infty}$ replaced by $\mathbf{u}_{\delta}^{\infty}$, $c$ is a constant
independent of sampling point $\mathbf{z}$.
\end{theorem}
\begin{proof}
For simplicity, we restrict ourselves to the case of $I_{\bf\rm PP}$. The other two can be done analogously.
\ben
&&I_{\bf\rm PP}(\mathbf{z})-I^{\delta}_{\bf\rm PP}(\mathbf{z})\cr
&:=&\left|\int_{\mathbb{S}}\phi^{p}_\mathbf{z}(\mathbf{d})\int_{\mathbb{S}}u_{pp}^\infty(\mathbf{\hx,d})\ov{\phi^{p}_\mathbf{z}(\mathbf{\hx})}ds(\mathbf{\hx})ds(\mathbf{d})\right|
   -\left|\int_{\mathbb{S}}\phi^{p}_\mathbf{z}(\mathbf{d})\int_{\mathbb{S}}u_{pp,\delta}^\infty(\mathbf{\hx,d})\ov{\phi^{p}_\mathbf{z}(\mathbf{\hx})}ds(\mathbf{\hx})ds(\mathbf{d})\right|\cr
&\leq& \left|\int_{\mathbb{S}}\phi^{p}_\mathbf{z}(\mathbf{d})\int_{\mathbb{S}}[u_{pp}^\infty(\mathbf{\hx,d})-u_{pp,\delta}^\infty(\mathbf{\hx,d})]\ov{\phi^{p}_\mathbf{z}(\mathbf{\hx})}ds(\mathbf{\hx})ds(\mathbf{d})\right|\cr
&\leq& c||\mathbf{u}^{\infty}-\mathbf{u}_{\delta}^{\infty}||_{\mathcal{L}^2}.
\enn
where we have used the triangle inequality in first inequality and the Cauchy-Schwarz Inequality in the second one.
\end{proof}
\begin{remark}
Theorem \ref{stability} show that our indicator depends continuously on the scattering amplitudes, which makes our indicator extremely stable with respect to noises.
\end{remark}

\subsection{Limited-aperture problems}\label{sec:limited-aperture}

In many cases of practical interest, the far field pattern $\mathbf{u}^{\infty}(\mathbf{\hat{x}, d})$ is restricted to the case when observation direction $\mathbf{\hat{x}}$ and incident direction
$\mathbf{d}$ are on a subset $\mathbb{S}_o$ and $\mathbb{S}_i$, respectively, of the unit sphere $\mathbb{S}$.
%The well known "backscattering" data coincide with taking $\mathbb{S}_o=-\mathbb{S}_i$.
One of the key questions of the inverse problems is how to achieve a reasonable reconstruction with these incomplete data.
Clearly, one may directly replace the integral surface $\mathbb{S}\times\mathbb{S}$ in \eqref{Indicatorff}-\eqref{Indicatorss} by $\mathbb{S}_i\times\mathbb{S}_o$.
However, as one could expect, the quality of the reconstructions decreases drastically
for this so called {\em limited-aperture problem}, and actually deteriorates as the aperture decreases.
A typical feature of the {\em limited-aperture} results is that the "shadow region" is highly elongated down range.
Physically, the information available from the "shadow region" is very weak, in particular for high frequency waves.
These phenomenon will also be observed in the subsequent numerical simulations.
Indeed, {\em limited-aperture} data can present a severe challenge for the numerical methods.

In this subsection, we introduce some techniques to retrieve the far field patterns that can not be measured directly. The corresponding results to acoustic scattering
problems can be found in a recent work \cite{LiuSun17}. For simplicity, we assume that $\mathbb{S}_i=\mathbb{S}$.

For convenience, we introduce the notation
\ben
\{\mathbf{u,v}\}_{\pa\Om}:=\int_{\pa\Om}[\mathbf{u}(\mathbf{y})\cdot \mathbb{T}_{\nu(\mathbf{y})}\mathbf{v}(\mathbf{y})-\mathbf{v}(\mathbf{y})\cdot \mathbb{T}_{\nu(\mathbf{y})}\mathbf{u}(\mathbf{y})]ds(\mathbf{y}).
\enn
\begin{theorem}\label{ThmRR}
For all $\mathbf{\hat{x},d}\in \mathbb{S}$, the far field pattern satisfies the reciprocity relations
\be\label{ReciprocityRelations1}
u^{\infty}_{pp}({\mathbf{\hat{x},d}})=u^{\infty}_{pp}({\mathbf{-d, -\hat{x}}}),
\quad u^{\infty}_{ss}({\mathbf{\hat{x},d}})=u^{\infty}_{ss}({\mathbf{-d, -\hat{x}}})
\en
and
\be\label{ReciprocityRelations2}
u^{\infty}_{ps}({\mathbf{\hat{x},d}})= u^{\infty}_{sp}({\mathbf{-d, -\hat{x}}}).
\en
\end{theorem}
\begin{proof}
We consider firstly the case when the incident field is a plane compressional wave, i.e.,
$\mathbf{u}^{in}=\mathbf{u}_{p}^{in}({\mathbf{x,d}})=\mathbf{d}e^{ik_{p}\mathbf{x}\cdot \mathbf{d}}$.
We further indicate the dependence of the corresponding scattered field and the total field on the incident direction $\mathbf{d}\in\mathbb{S}$ by writing
$\mathbf{u}^{sc}\mathbf{(x,d)}$ and $\mathbf{u}\mathbf{(x,d)}$, respectively.
From the representation \eqref{upinfty2}, we have
\be\label{uppxd}
u^\infty_{pp}\mathbf{(\hat{x},d)}= \{\mathbf{u}_{p}^{in}(\mathbf{\cdot,-\hat{x}}), \mathbf{u}^{sc}\mathbf{(\cdot,d)}\}_{\pa\Om}
\en
and, interchanging the roles of $\mathbf{\hat{x}}$ and $\mathbf{d}$,
\be\label{uppdx}
u^\infty_{pp}\mathbf{(-d,-\hat{x})}= \{\mathbf{u}_{p}^{in}(\mathbf{\cdot,d}), \mathbf{u}^{sc}\mathbf{(\cdot,-\hat{x})}\}_{\pa\Om}.
\en
Applying the Betti formula to $\mathbf{u}^{in}$ in $\Om$ yields
\be
 \{\mathbf{u}_{p}^{in}(\mathbf{\cdot,d}), \mathbf{u}_{p}^{in}(\mathbf{\cdot,-\hat{x}})\}_{\pa\Om}=0.
\en
Using the Betti formula again to $\mathbf{u}^{sc}$ in the region $\{\mathbf{x}\in\R^2\ba\ov{\Om}: |\mathbf{x}|<R\}$ with $R$ large enough, with the help of
\eqref{usasymptotic} and \eqref{Tusasymptotic}, and finally letting $R$ tend to infinity, we deduce that
\be\label{us=0}
 \{\mathbf{u}^{sc}(\mathbf{\cdot,d}), \mathbf{u}^{sc}(\mathbf{\cdot,-\hat{x}})\}_{\pa\Om}=0.
\en
We now subtract \eqref{uppxd} from the sum of \eqref{uppdx}-\eqref{us=0} to obtain
\ben
u^\infty_{pp}\mathbf{(\hat{x},d)}-u^\infty_{pp}\mathbf{(-d,-\hat{x})}
= \{\mathbf{u}(\mathbf{\cdot,-\hat{x}}), \mathbf{u}(\mathbf{\cdot,d})\}_{\pa\Om}
=0,
\enn
where we have used the boundary conditions \eqref{Dirichlet}-\eqref{Neumann} on $\pa\Om$ in the last equality.
Thus we have showed the first reciprocity relation in \eqref{ReciprocityRelations1}, and the other two in \eqref{ReciprocityRelations1}-\eqref{ReciprocityRelations2}
can be proved analogously.
\end{proof}

The first two reciprocity relations in \eqref{ReciprocityRelations1} is an extension of the well-known ones in acoustic and electromagnetic scattering problems \cite{CK}.
The third one in \eqref{ReciprocityRelations2} is more interesting since it establishes a connection between the compressional far field pattern corresponding to a plane shear wave and the shear far field pattern
corresponding to a plane compressional wave.
We refer to \cite{DassiosKiriakiPolyzos} for the reciprocity relations in 3D.

In $\R^2$, we choose an equidistant set of knots $\theta_i:=(i-1)\pi/m,\,i=1,2,\ldots,2m$ from $[0,2\pi)$.
Assume that we have a set of incident plane waves with incident directions
\ben
\mathbf{d}_i:=(\cos\theta_i,\sin\theta_i),\,i=1,2,\ldots,2m.
\enn
The far field patterns are measured in different observation directions
\ben
\mathbf{\hat{x}}_j:=(\cos\theta_j,\sin\theta_j),\,j=1,2,\ldots,2m.
\enn
In the finite case we define the multi-static response (MSR) matrix $\mathbb{F}_{full}\in \C^{4m\times 4m}$ by
\be\label{MSR}
\mathbb{F}_{full}
:= \left(
     \begin{array}{cc}
       \mathbb{F}_{pp} & \mathbb{F}_{ps} \\
       \mathbb{F}_{sp} & \mathbb{F}_{ss} \\
     \end{array}
   \right)
\en
with
\be\label{MSR2}
\mathbb{F}_{\rm ii}
:= \left(
    \begin{array}{cccc}
      u^{1,1}_{\rm ii}\quad u^{1,2}_{\rm ii}\quad \cdots\quad u^{1,2m}_{\rm ii} \\
      u^{2,1}_{\rm ii}\quad u^{2,2}_{\rm ii}\quad \cdots\quad u^{2,2m}_{\rm ii} \\
      \vdots\,\qquad \vdots\,\quad \ddots\,\qquad \vdots \\
      u^{2m,1}_{\rm ii}\,\, u^{2m,2}_{\rm ii}\,\, \cdots\quad u^{2m,2m}_{\rm ii} \\
      \end{array}
  \right), \quad {\rm ii}=pp,\, ps,\, sp,\, ss,
\en
where $u_{\rm ii}^{i,j}=u^\infty_{\rm ii}(\mathbf{\hat{x}}_j, \mathbf{d}_i)$ for $1\leq i, j\leq 2m$ corresponding to $2m$ observation directions $\mathbf{\hat{x}}_j$
and $2m$ incident directions $\mathbf{d}_i$.
%The response matrix can be regarded as a discrete version of the far field operator.
The MSR matrix $\mathbb{F}_{full}$ given in \eqref{MSR} is regarded as the scattering amplitude in {\em full-aperture}.
We can partition both the rows and columns of the $2m$-by-$2m$ MSR submatrix $\mathbb{F}_{\rm ii}$ to obtain a $2$-by-$2$ block matrix
\be\label{Fpartition}
\mathbb{F}_{\rm ii}=\left(
             \begin{array}{cc}
               \mathbb{F}^{11}_{\rm ii} & \mathbb{F}^{12}_{\rm ii} \\
               \mathbb{F}^{21}_{\rm ii} & \mathbb{F}^{22}_{\rm ii} \\
             \end{array}
           \right), \quad {\rm ii}=pp,\, ps,\, sp,\, ss.
\en
We can now use the reciprocity Theorem \ref{ThmRR} to deduce the symmetric structure of the MSR matrix $\mathbb{F}_{full}$.
\begin{corollary}\label{blocksymmetric}
For the block matrix defined in \eqref{Fpartition}, we have
\begin{itemize}
  \item $\mathbb{F}^{11}_{\rm ii}=(\mathbb{F}^{22}_{\rm ii})^{T},\quad \mathbb{F}^{12}_{\rm ii}=(\mathbb{F}^{12}_{\rm ii})^{T}$\quad
        and\quad $\mathbb{F}^{21}_{\rm ii}=(\mathbb{F}^{21}_{\rm ii})^{T}$,\quad {\rm ii}=pp,\,ss;
  \item $\mathbb{F}^{11}_{ps}=(\mathbb{F}^{22}_{sp})^{T}$,\quad $\mathbb{F}^{22}_{ps}=(\mathbb{F}^{11}_{sp})^{T}$,\quad
       $\mathbb{F}^{12}_{ps}=(\mathbb{F}^{12}_{sp})^{T}$\quad and \quad $\mathbb{F}^{21}_{ps}=(\mathbb{F}^{21}_{sp})^{T}$.
\end{itemize}
\end{corollary}

Using Corollary \ref{blocksymmetric}, the far field patterns that can not be measured directly could be retrieved immediately from the symmetric
structure of the MSR matrix $\mathbb{F}_{full}$.
Let $B$ be a ball large enough such that $\ov{\Om}\subset B$ and $k^2$ is not a Dirichlet eigenvalue for $-\Delta$ in $B$. Define the integral operator
$S^{\infty}_k: L^2(\pa B)\rightarrow L^2(\mathbb{S})$ by
\be\label{Sinfty}
(S^{\infty}_{k}\phi)(\mathbf{\hat{x}}):=\int_{\pa B}e^{-ik\mathbf{\hat{x}\cdot y}}\phi(\mathbf{y})ds(\mathbf{y}),\quad \mathbf{\hat{x}}\in \mathbb{S}.
\en
Then one can show that the operator $S^{\infty}_k$ is injective and has dense range, see e.g., Theorem 5.19 in \cite{CK}.
To retrieve {\em full-aperture} data, one may firstly solve the equations
\be\label{Sinftyequations}
S^{\infty}_{k_p}\phi=u_p^{\infty},\quad\mbox{and}\quad S^{\infty}_{k_s}\psi=u_s^{\infty}
\en
by the Tikhonov regularization in $L^2(\mathbb{S}_0)$ and then insert the solutions $\phi$ and $\psi$
into $u_p^{\infty}(\mathbf{\hat{x}}):=(S_{k_p}^{\infty}\phi)(\mathbf{\hat{x}})$ and $u_s^{\infty}(\mathbf{\hat{x}}):=(S_{k_s}^{\infty}\psi)(\mathbf{\hat{x}})$, respectively for
$\mathbf{\hat{x}}\in \mathbb{S}$ to obtain the missing data. Unfortunately, the integral operator $S_k^{\infty}$ has an analytic kernel and therefore the equations
\eqref{Sinftyequations} are severely ill-posed. Following ideas from \cite{LiuSun17}, one may combine this with the special symmetric structure given in Corollary \ref{blocksymmetric}
to retrieve the data that cannot be measured directly.
With the reconstructed full-aperture data, using the indicators given in \eqref{Indicatorff}-\eqref{Indicatorss}, it is desirable to partially solve the limited-aperture problems.

\section{Numerical experiments}
Now we turn to present a variety of numerical examples in two dimensions to illustrate the efficiency and effectiveness of our sampling methods.
There are totally six groups of numerical tests to be considered, and they are
respectively referred to as {\bf Dirichlet, Neumann, Multiple, MultiScalar, ResolutionLimit}  and {\bf Limited-aperture}.

The boundaries of the scatterers used in our numerical experiments are parameterized as follows.
\begin{small}
\begin{eqnarray}
\label{circle}&\mbox{\rm Circle:}&\quad \mathbf{x}(t)\ =(a,b)+\rho\ (\cos t, \sin t),\quad 0\leq t\leq2\pi,\\
\label{peanut}&\mbox{\rm Peanut:}&\quad \mathbf{x}(t)\ =(a,b)+\rho\ \sqrt{3\cos^2 t+1}(\cos t, \sin t),\quad 0\leq t\leq2\pi,\\
\label{pear}&\mbox{\rm Pear:}&\quad \mathbf{x}(t)\ =(a,b)+\rho(2+0.3\cos 3t)\ (\cos t, \sin t),\quad 0\leq t\leq2\pi,\\
\label{kite}&\mbox{\rm Kite:}&\quad \mathbf{x}(t)\ =(a,b)+\rho\ (\cos t+0.65\cos 2t-0.65, 1.5\sin t),\quad 0\leq t\leq2\pi,
\end{eqnarray}
\end{small}
with $(a,b)$ and $\rho$ being the location and size, respectively, of the scatterer which may be different in different examples.
For simplicity, we seek the scattered field in the form
\ben
\mathbf{u}^{sc}(\mathbf{x})=\int_{\pa\Om}\Phi(\mathbf{x,y})\bbpsi(\mathbf{y})ds(\mathbf{y}),\quad \mathbf{x}\in\R^2\ba\ov{\Om},
\enn
where the density $\bbpsi$ is to be determined by using Nystr\"{o}m method with the boundary conditions \eqref{Dirichlet}-\eqref{Neumann}.
Then the corresponding far field patterns are given by
\ben
u^{\infty}_{p}(\mathbf{\hx})&=&\int_{\pa\Om}e^{-ik_p\mathbf{\hx\cdot y}}\bbpsi(\mathbf{y})\cdot\mathbf{\hx}ds(\mathbf{y}),\quad \mathbf{\hx}\in\mathbb{S},\\
u^{\infty}_{s}(\mathbf{\hx})&=&\int_{\pa\Om}e^{-ik_s\mathbf{\hx\cdot y}}\bbpsi(\mathbf{y})\cdot\mathbf{\hx}^{\perp}ds(\mathbf{y}),\quad \mathbf{\hx}\in\mathbb{S},
\enn
which can be evaluated by the trapezoidal rule after solving the integral equation for $\bbpsi$.
We refer to the \cite{BaoXuYin, HsiaoWendland, YinHsiaoXu} and the references therein
for other numerical methods for the exterior elastic scattering problems.

In our simulations, we compute the far field patterns
$\mathbf{u}^\infty(\theta_j, \theta_l)$  with $\theta_j = 2\pi j/N$,
for $N$ equidistantly distributed incident directions and observation directions.
These data are then stored in the matrix $F_\Om \in \C^{N \times N}$.
We further perturb
$F_\Om$ by random noise using
\ben
F_{\Om}^{\delta}\ =\ F_{\Om} +\delta\|F_{\Om}\|\frac{R_1+R_2 i}{\|R_1+R_2 i\|},
\enn
where $R_1$ and $R_2$ are two $N \times N$ matrixes containing pseudo-random values
drawn from a normal distribution with mean zero and standard deviation one.
The value of $\delta$ used in our code is $\delta:=\|F_{\Om}^{\delta} -F_{\Om}\|/\|F_{\Om}\|$ and so presents the relative error.

In the following experiments, we take $N=512$, the relative error level $\delta = 10\%, 30\%$, and the Lam$\acute{\mbox{e}}$ constants $\lambda=1$ and $\mu=1$.
If not stated otherwise, we take the circular frequency $\om=8\pi$.

\subsection{Example Dirichlet}\label{sec:single_Dirichlet}
This example is designed to check the validity of our methods for scatterers with different shapes.
For simplicity, we impose Dirichlet boundary condition on the
underlying scatterers. The Neumann case will be considered in the next example.
The underlying scatterer is chosen to be a kite or a pear shaped domain with $(a,b)=(0,0)$ and $\rho=1$.
The research domain is chosen as $[-6,6]\times[-6,6]$ with $321\times 321$ equally spaced sampling points.
Figure \ref{fig:Dirichlet_noise30} shows the corresponding results.
Clearly, different to the classical sampling methods, all the proposed indicators decay like the Bessel functions as the sampling points away from the boundary.
All the results show that our methods are very robust to noise. We observe that, even up to $30\%$ noise is added, both the location and shape of the underlying scatterer
can be well reconstructed.
This verifies our theory  analyses and can also be observed in the subsequent examples.\\

\begin{figure}
\centering
\subfigure[\textbf{True domain}]{\includegraphics[width=0.24\textwidth]{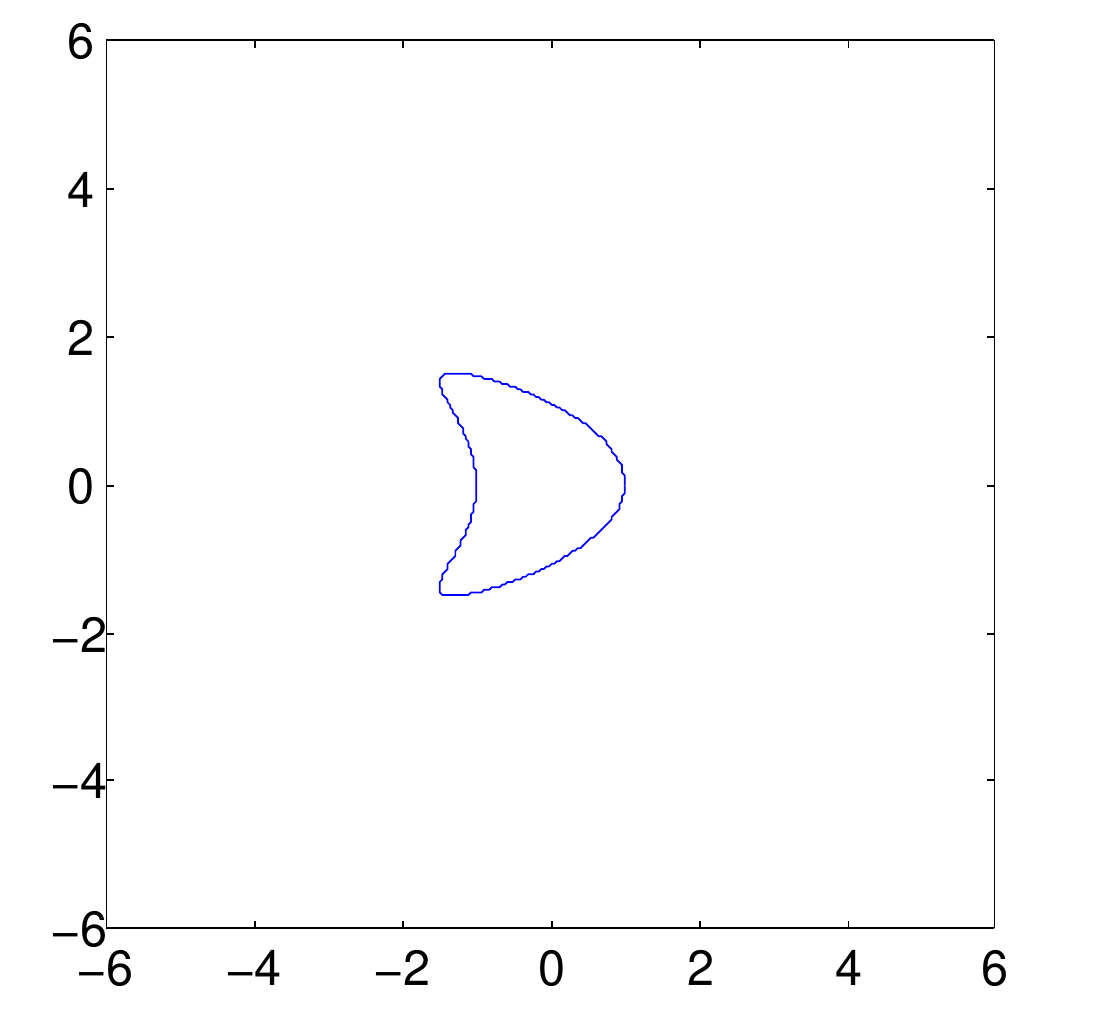}}
\subfigure[\textbf{$I_{\rm SS}^2$}]{\includegraphics[width=0.24\textwidth]{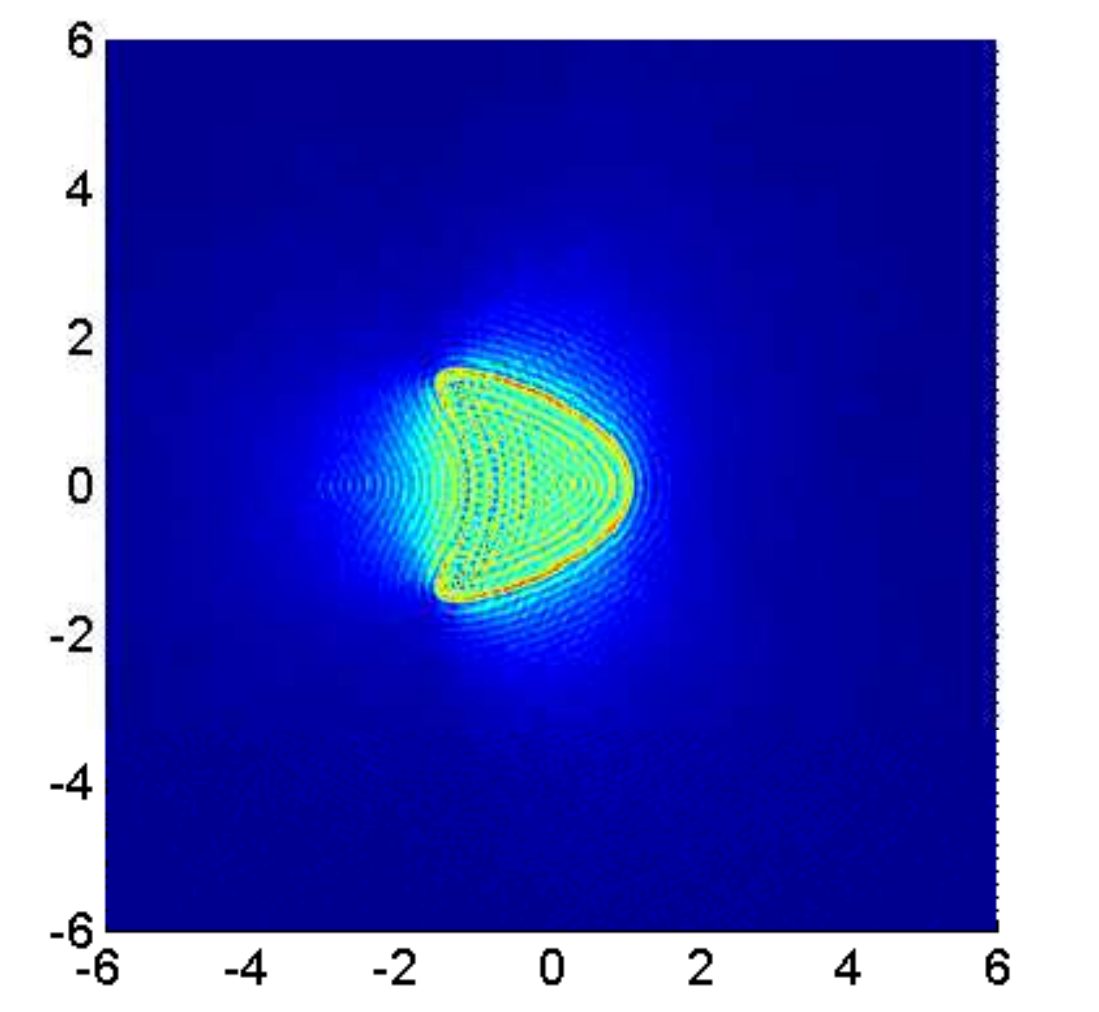}}
\subfigure[\textbf{$I_{\rm PP}^2$}]{\includegraphics[width=0.24\textwidth]{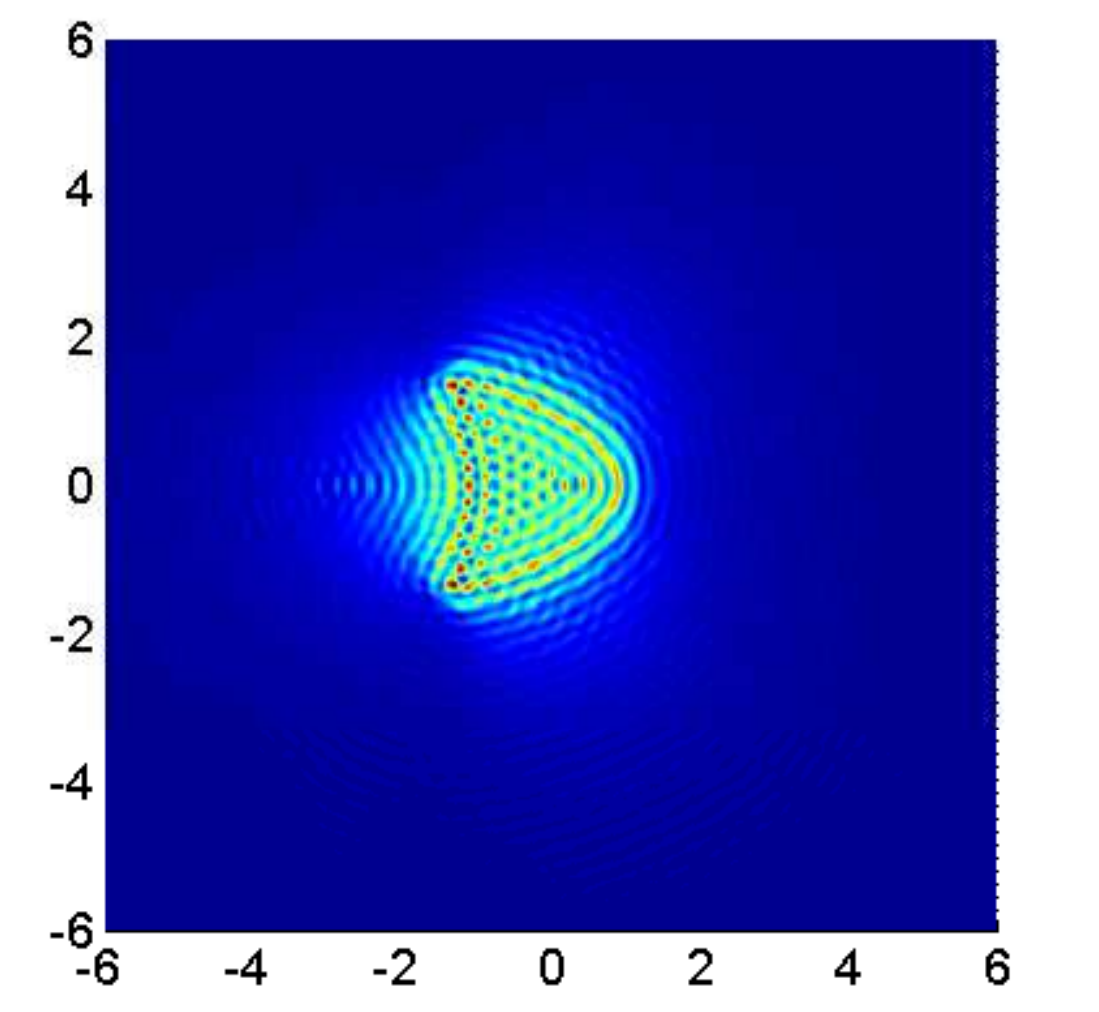}}
\subfigure[\textbf{$I_{\rm FF}^2$}]{\includegraphics[width=0.24\textwidth]{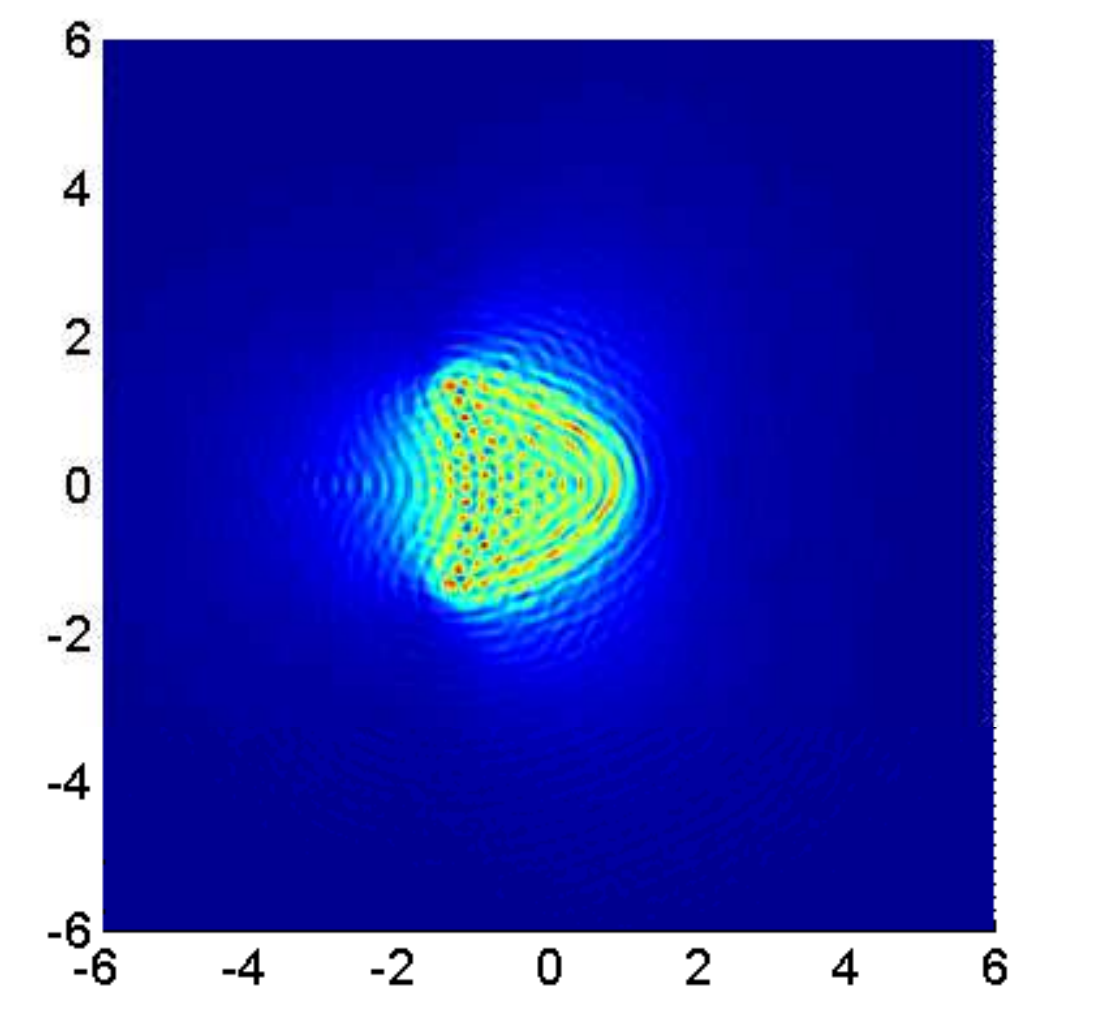}}
\subfigure[\textbf{True domain}]{\includegraphics[width=0.24\textwidth]{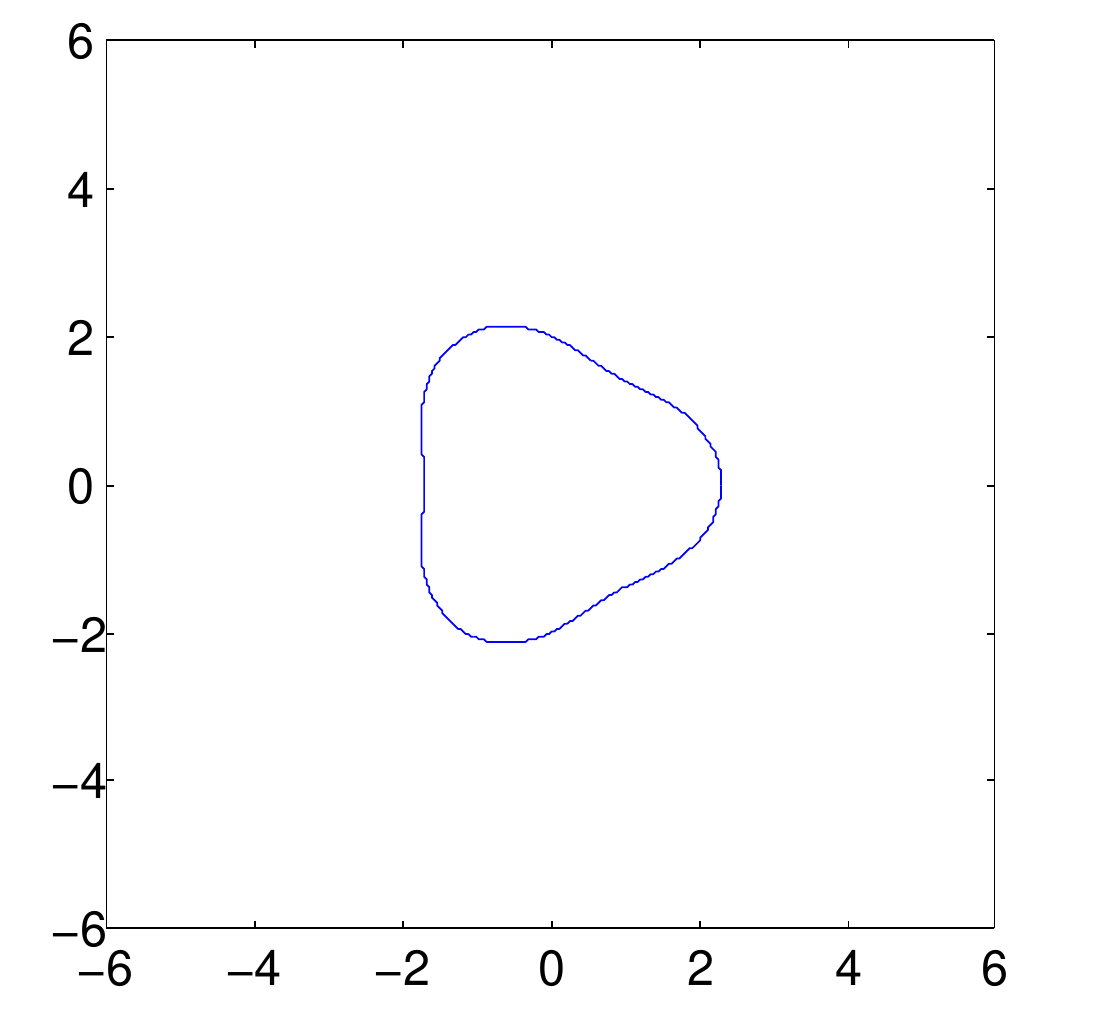}}
\subfigure[\textbf{$I_{\rm SS}^2$}]{\includegraphics[width=0.24\textwidth]{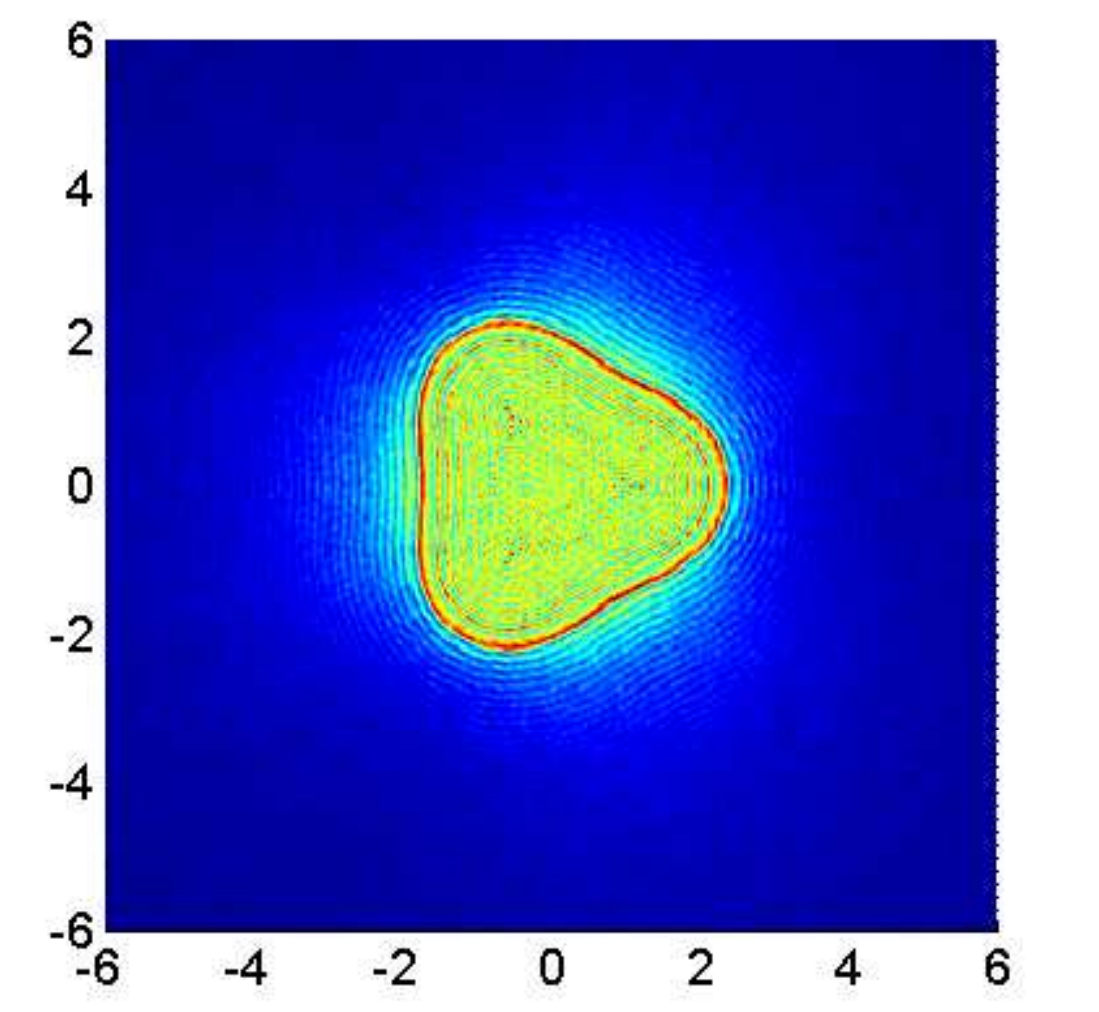}}
\subfigure[\textbf{$I_{\rm PP}^2$}]{\includegraphics[width=0.24\textwidth]{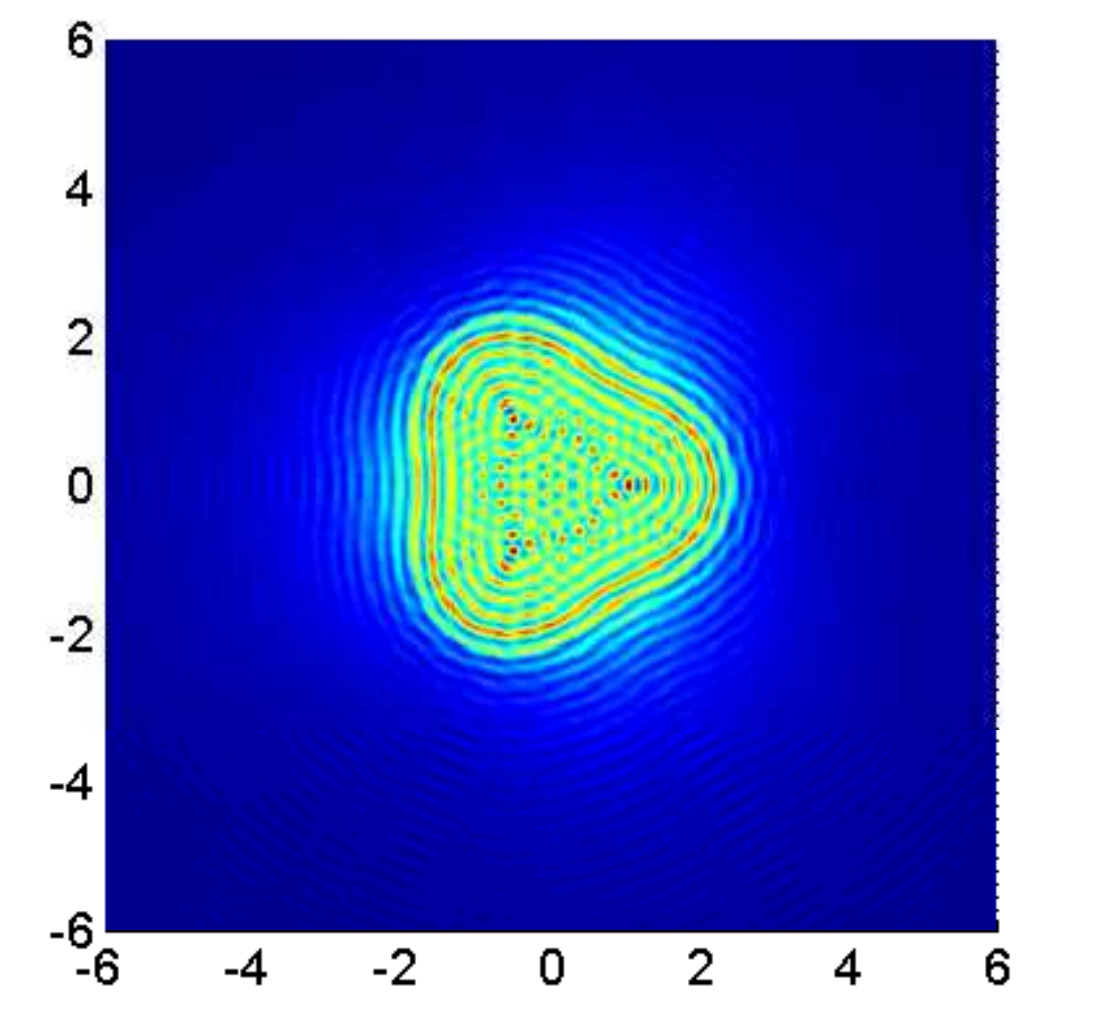}}
\subfigure[\textbf{$I_{\rm FF}^2$}]{\includegraphics[width=0.24\textwidth]{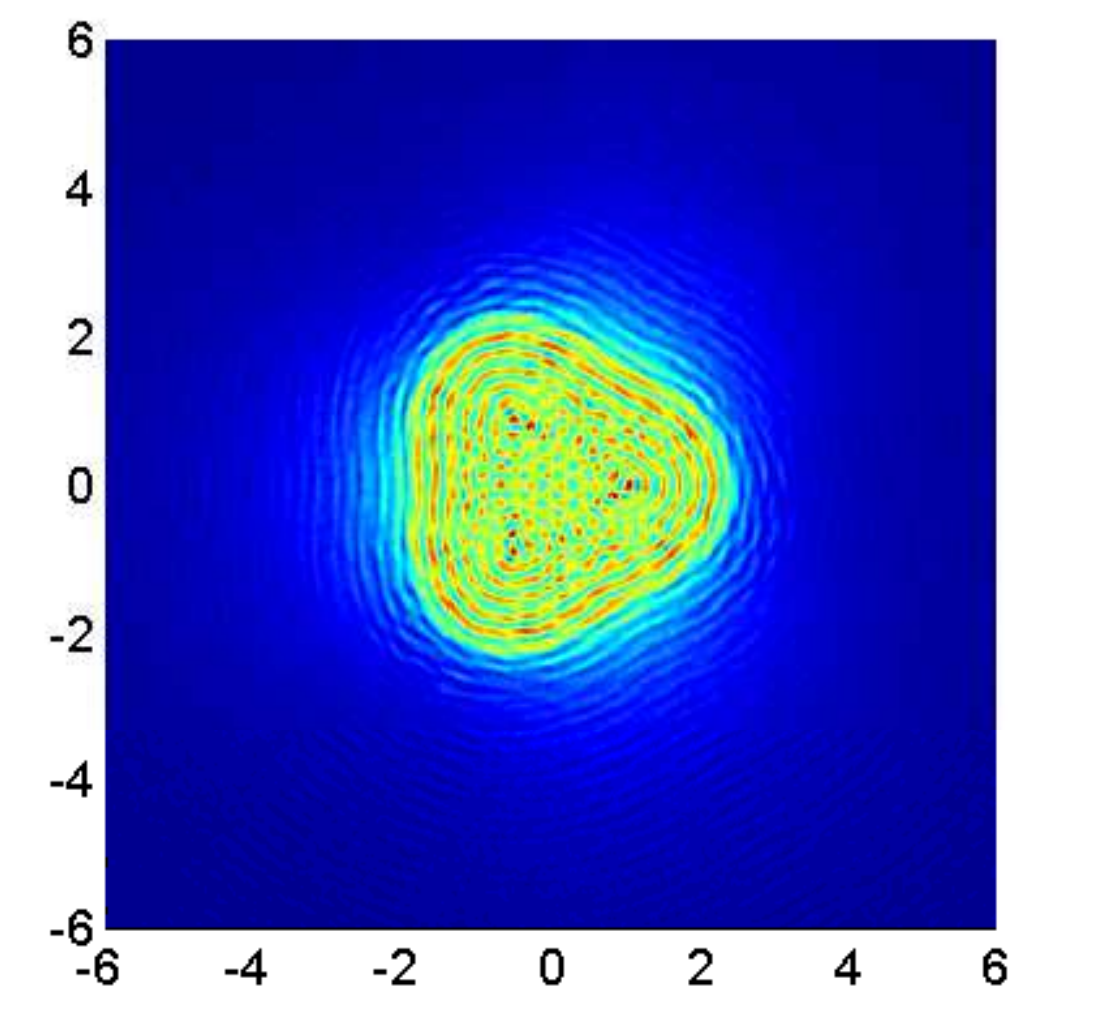}}
\caption{{\bf Example Dirichlet.} Reconstructions of a rigid body with $30\%$ noise. Top row: kite shaped domain. Bottom row: pear shaped domain.}
\label{fig:Dirichlet_noise30}
\end{figure}

\subsection{Example Neumann}

For comparison, we consider the scattering by cavities, i.e., the Neumann boundary condition is imposed. We take the same domains as the previous example.
The reconstructions are shown in Figure \ref{fig:Neumann_noise30}.
From this, we observe that the locations and shapes are still be well captured.
This further shows that our methods are independent of the physical properties of the underlying scatterer.\\

\begin{figure}[htbp]
\centering
\subfigure[\textbf{True domain}]{\includegraphics[width=0.24\textwidth]{kite.pdf}}
\subfigure[\textbf{$I_{\rm SS}^2$}]{\includegraphics[width=0.24\textwidth]{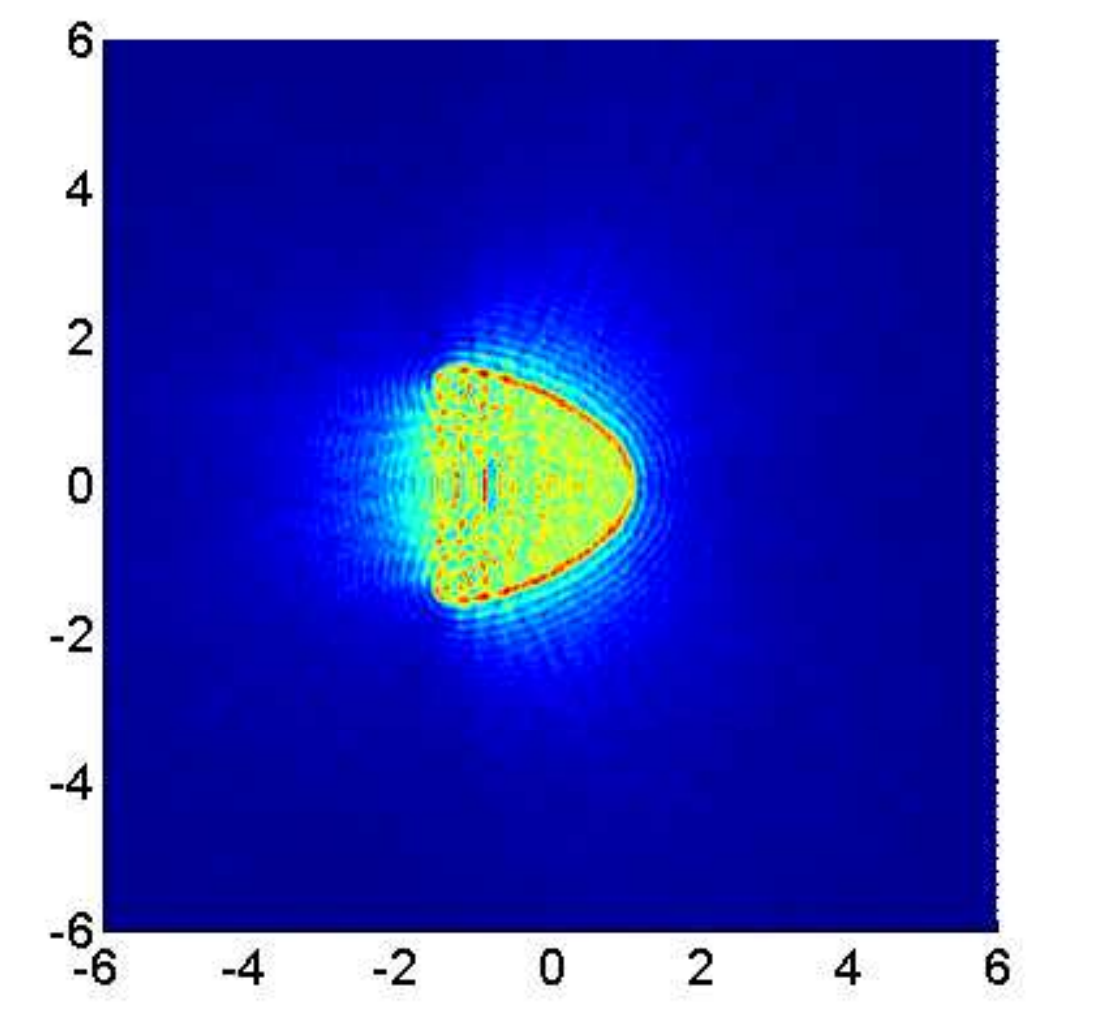}}
\subfigure[\textbf{$I_{\rm PP}^2$}]{\includegraphics[width=0.24\textwidth]{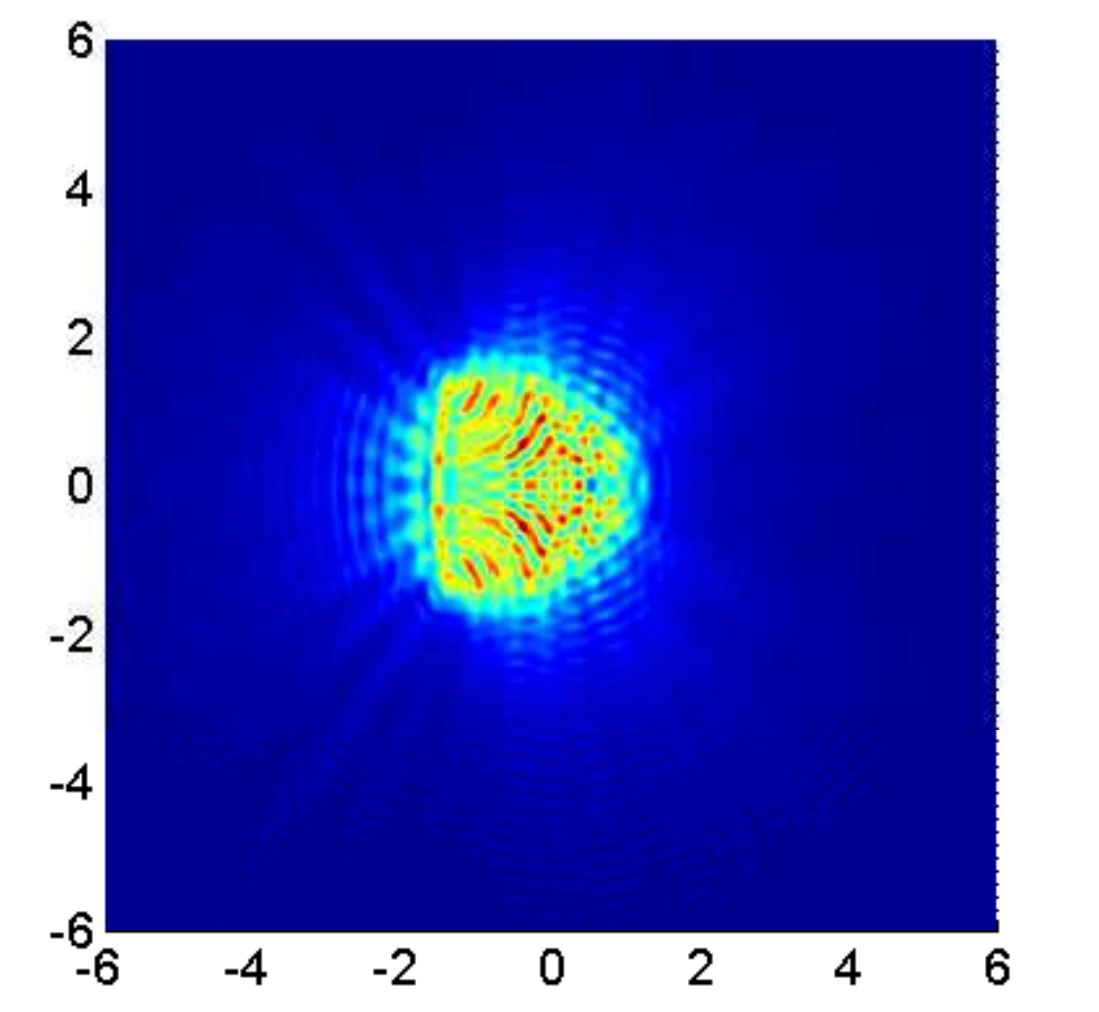}}
\subfigure[\textbf{$I_{\rm FF}^2$}]{\includegraphics[width=0.24\textwidth]{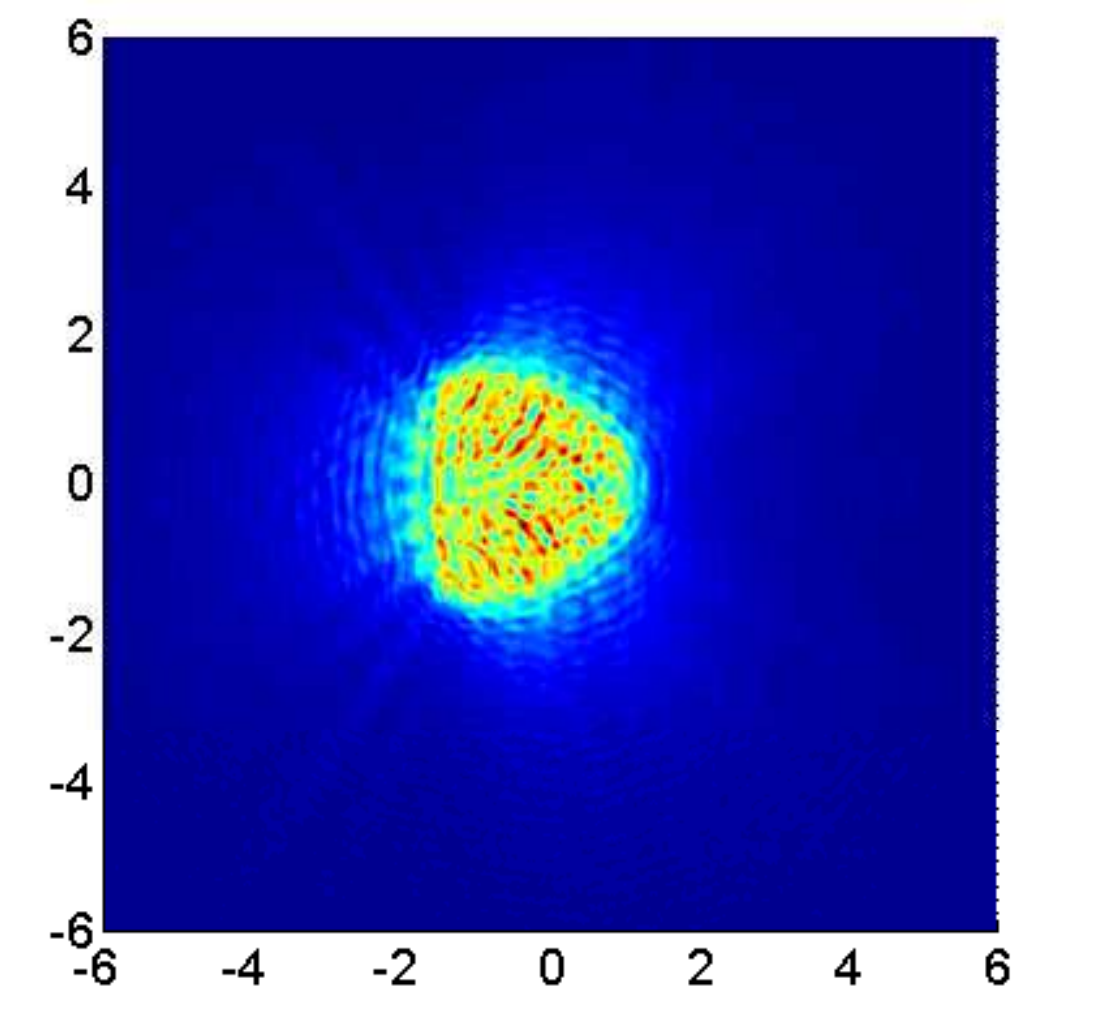}}
\subfigure[\textbf{True domain}]{\includegraphics[width=0.24\textwidth]{pear.pdf}}
\subfigure[\textbf{$I_{\rm SS}^2$}]{\includegraphics[width=0.24\textwidth]{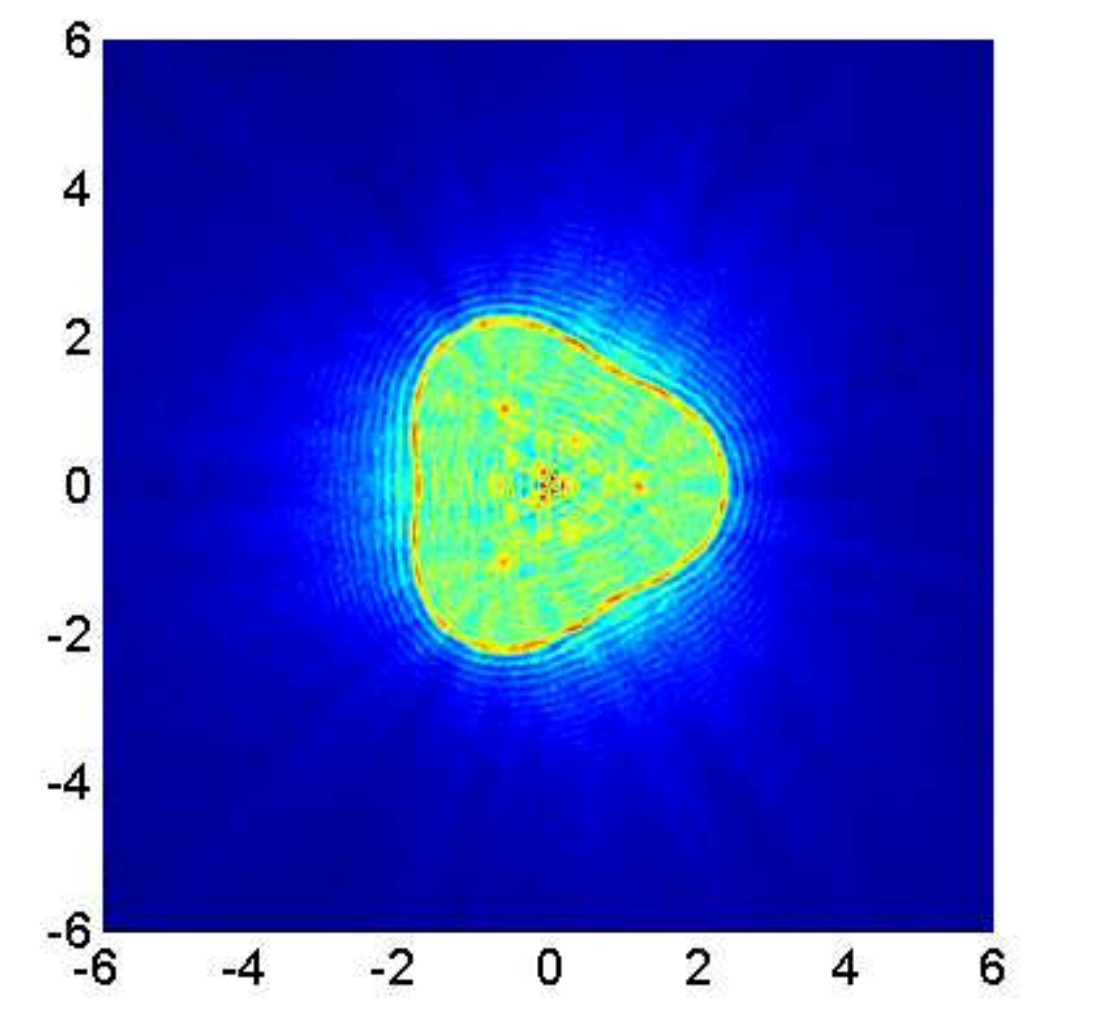}}
\subfigure[\textbf{$I_{\rm PP}^2$}]{\includegraphics[width=0.24\textwidth]{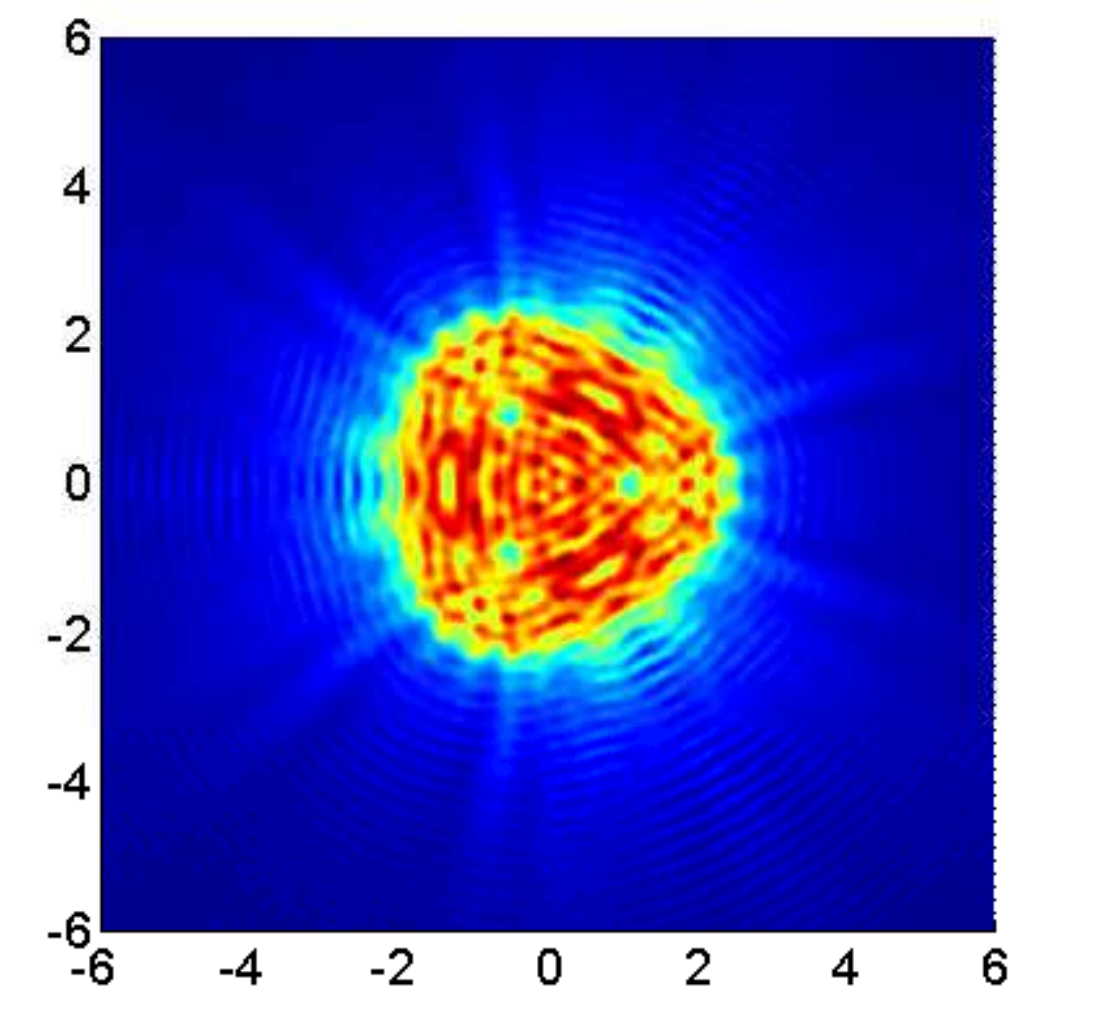}}
\subfigure[\textbf{$I_{\rm FF}^2$}]{\includegraphics[width=0.24\textwidth]{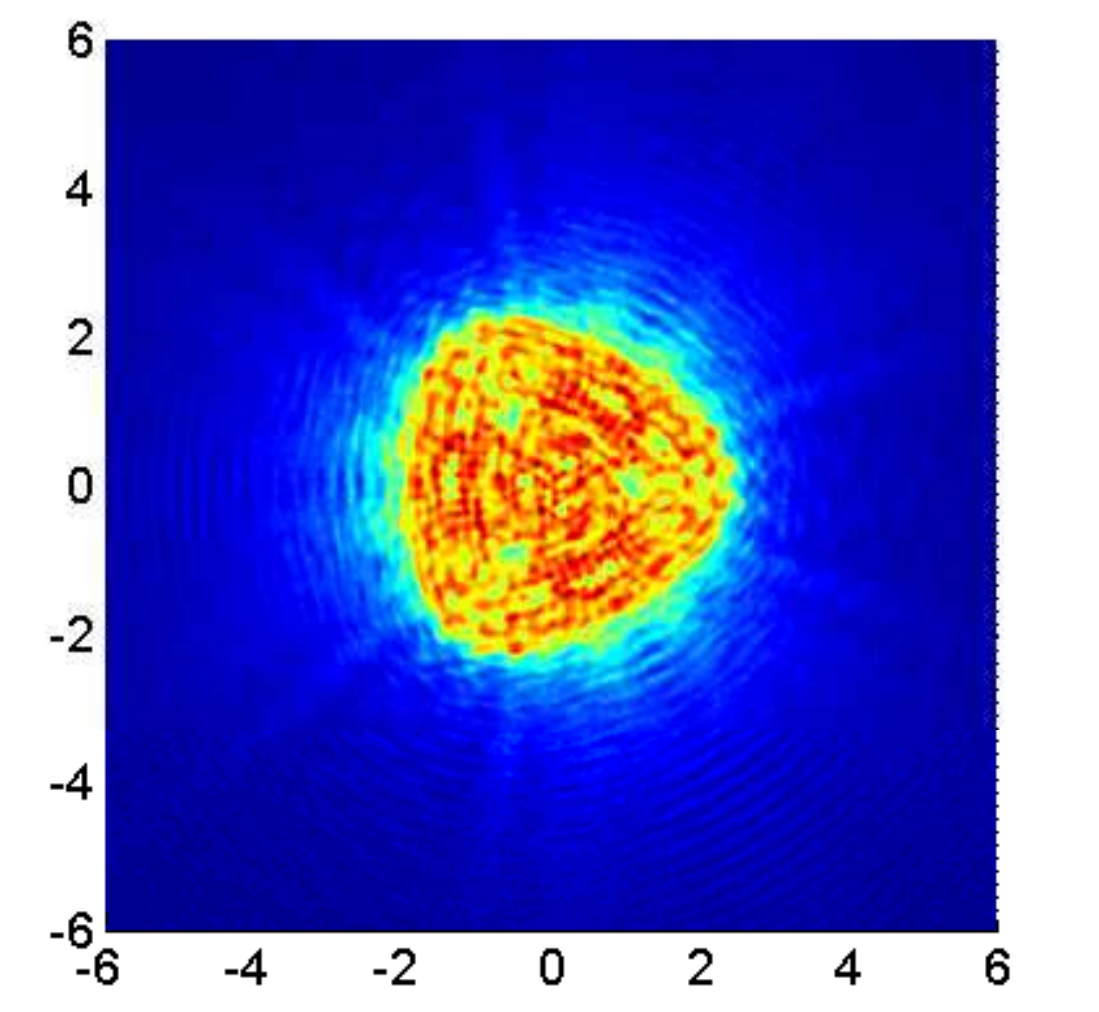}}
\caption{{\bf Example Neumann.} Reconstructions of a cavity with $30\%$ noise. Top row: kite shaped domain. Bottom row: pear shaped domain.}
\label{fig:Neumann_noise30}
\end{figure}

\subsection{Example Multiple}

We consider the scattering by a scatterer with two disjoint components. We set the scatterer to be a combination of a peanut shaped domain centered at
$(3,-3)$ and a kite shaped domain centered at $(-3,3)$.
We impose Dirichlet boundary condition on the boundary $\pa \Om$.
The search domain is the rectangle $[-6,6]\times[-6,6]$ with $641\times 641$ equally spaced sampling points.
The reconstructions are shown in Figure \ref{fig:kite_peanut_noise30}.\\

\begin{figure}
\subfigure[\textbf{True domain}]{\includegraphics[width=0.24\textwidth]{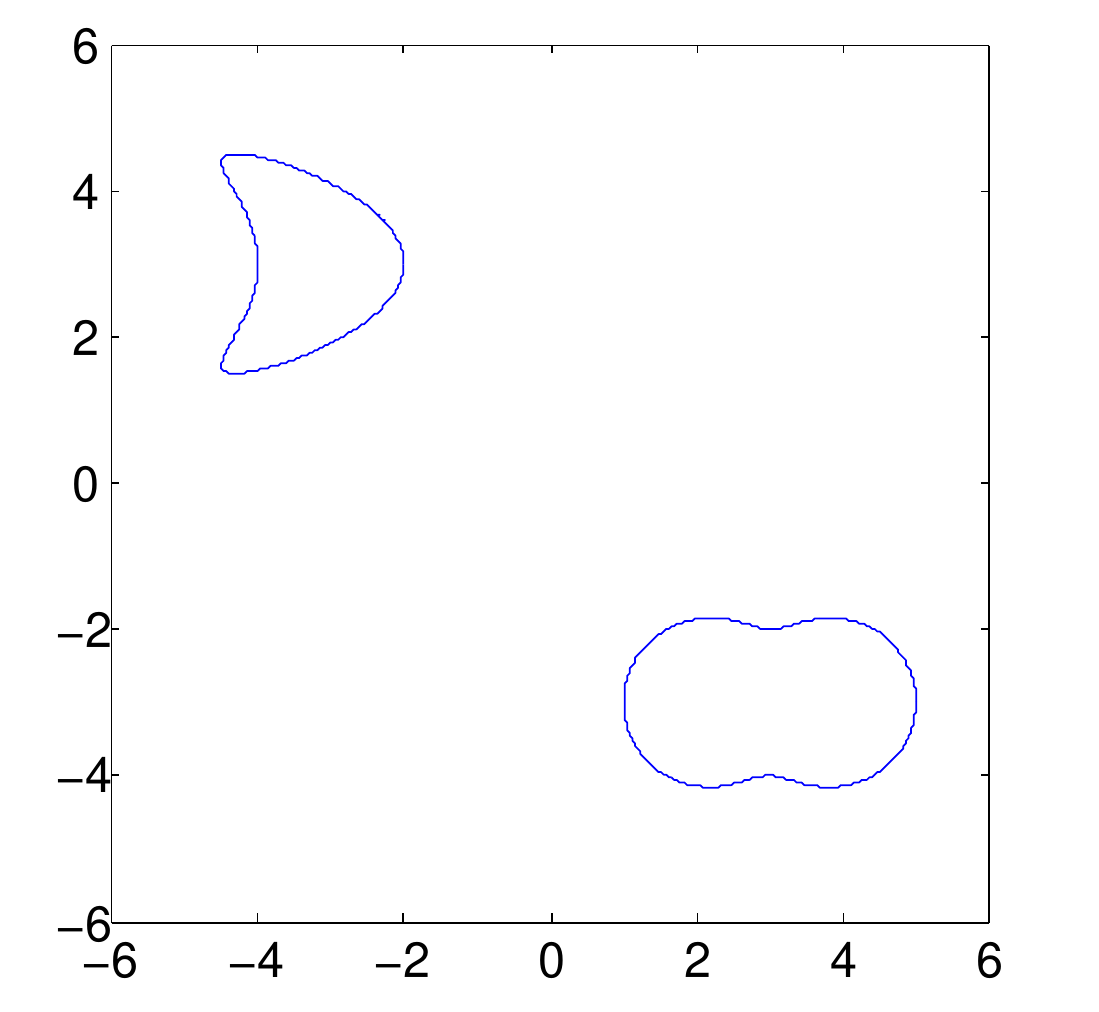}}
\subfigure[\textbf{$I_{\rm SS}^2$}]{\includegraphics[width=0.24\textwidth]{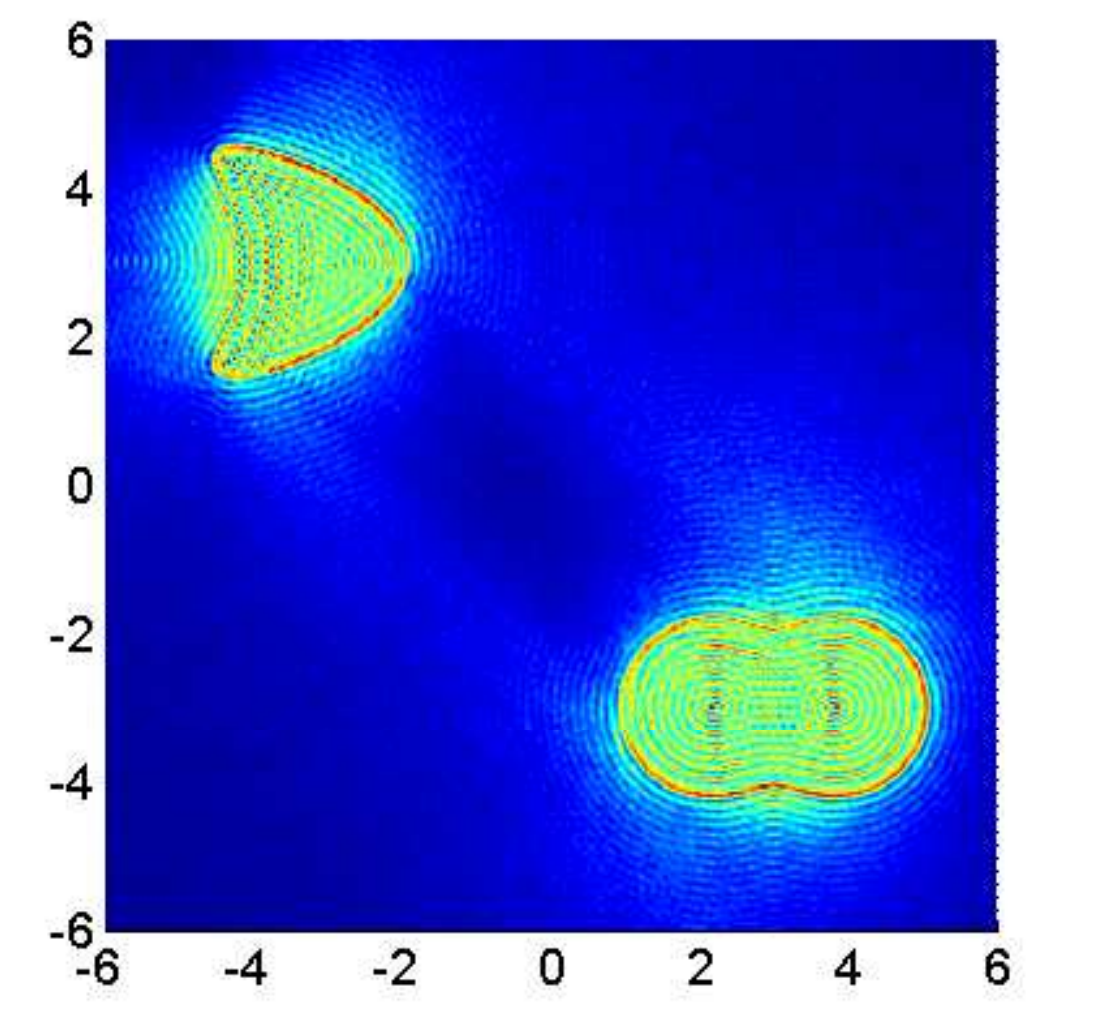}}
\subfigure[\textbf{$I_{\rm PP}^2$}]{\includegraphics[width=0.24\textwidth]{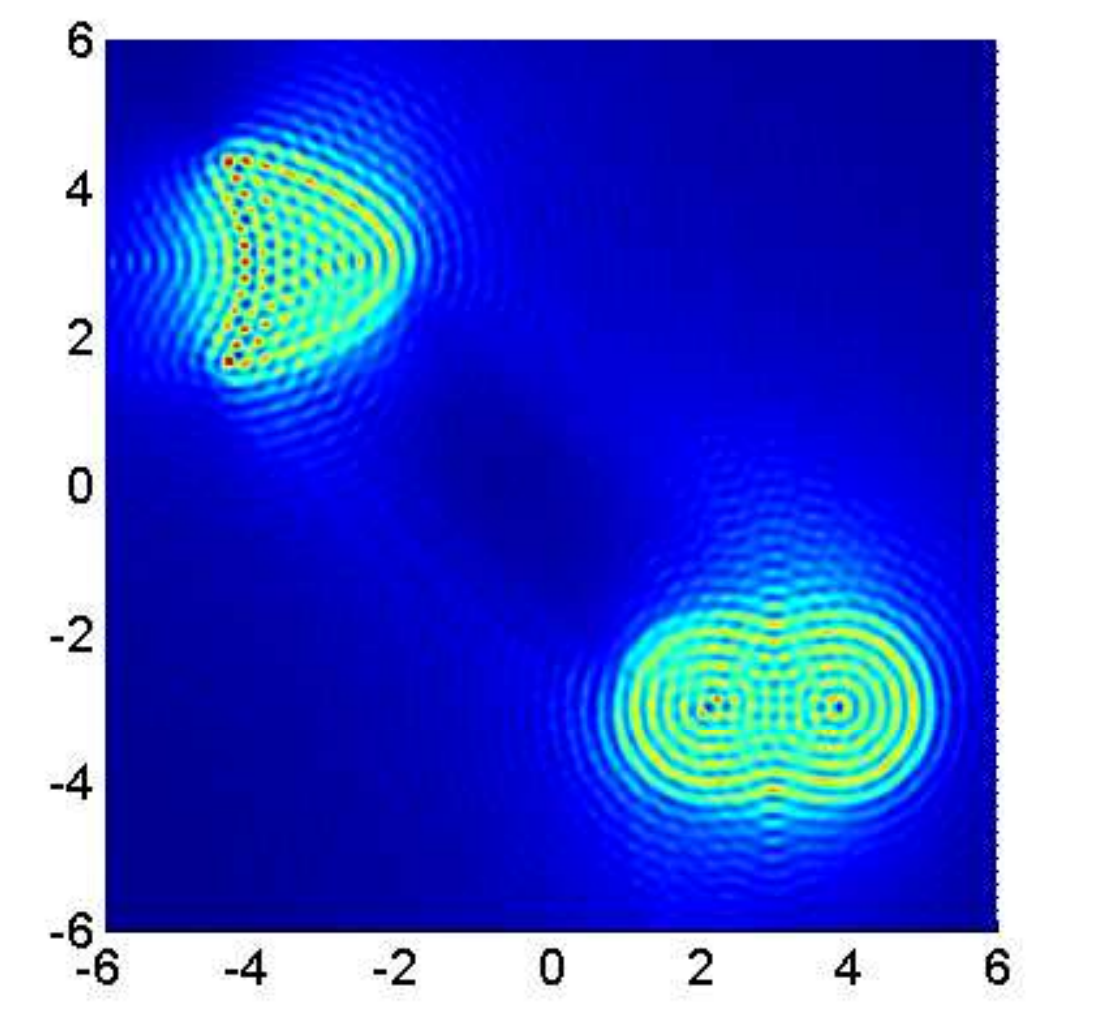}}
\subfigure[\textbf{$I_{\rm FF}^2$}]{\includegraphics[width=0.24\textwidth]{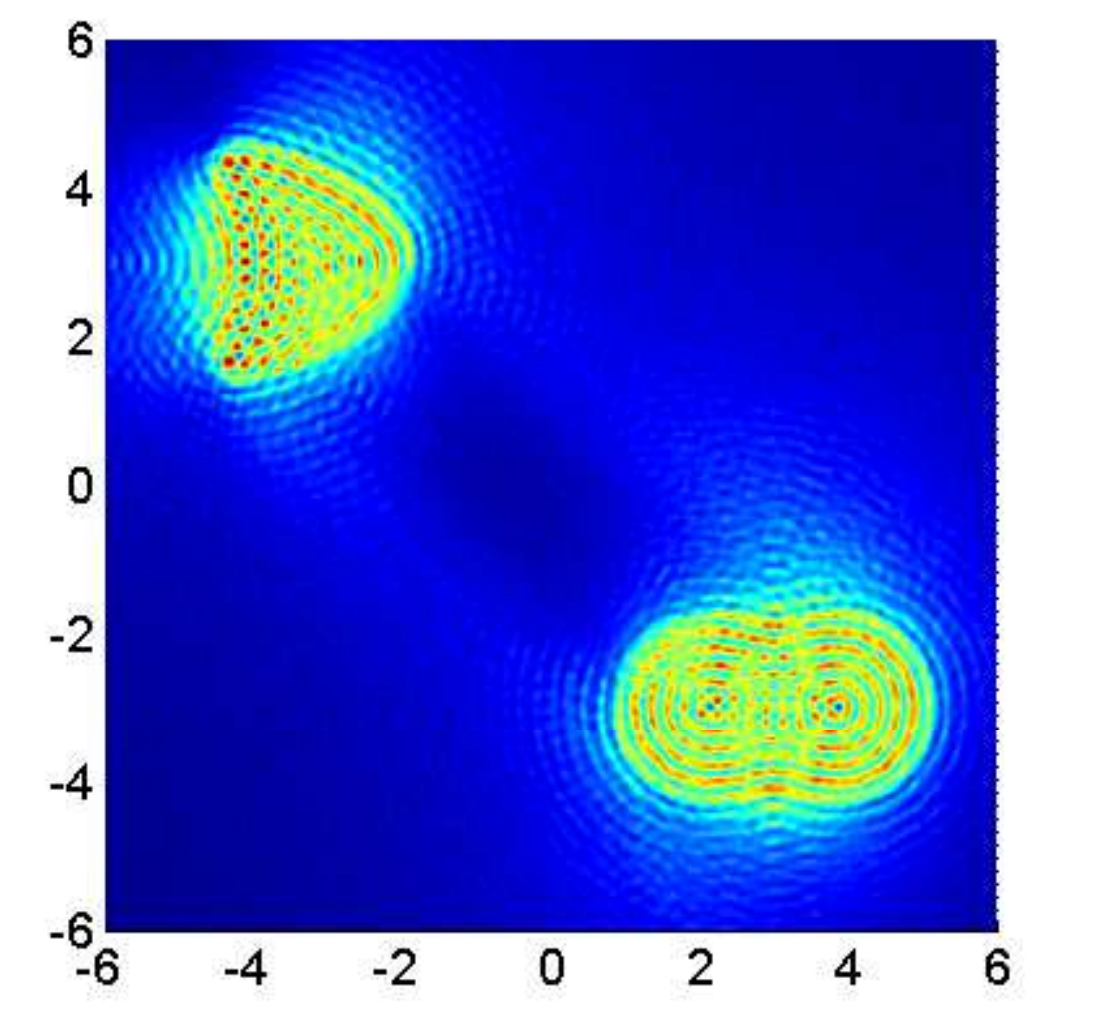}}
\caption{{\bf Example Multiple.} Reconstructions for \textbf{kite-peanut} scatterers with $30\%$ noise.}
\label{fig:kite_peanut_noise30}
\end{figure}

\subsection{Example Multiscalar}
In this example, the underlying scatterer is a combination of a big pear shaped domain with $\rho=2$ centered at $(0,0)$ and a mini disk with radius $\rho=0.1$ centered at $(a,b)=(4,4)$. We set $N=512$ and $\om=8\pi$.
The reconstructions are shown in Figure \ref{fig:pear_circle_n1024_noise30}.
We observe that both parts can be well reconstructed. In particular, the mini disk is also exactly located, even with $30\%$ noise.\\

\begin{figure}
\subfigure[\textbf{True domain}]{\includegraphics[width=0.24\textwidth]{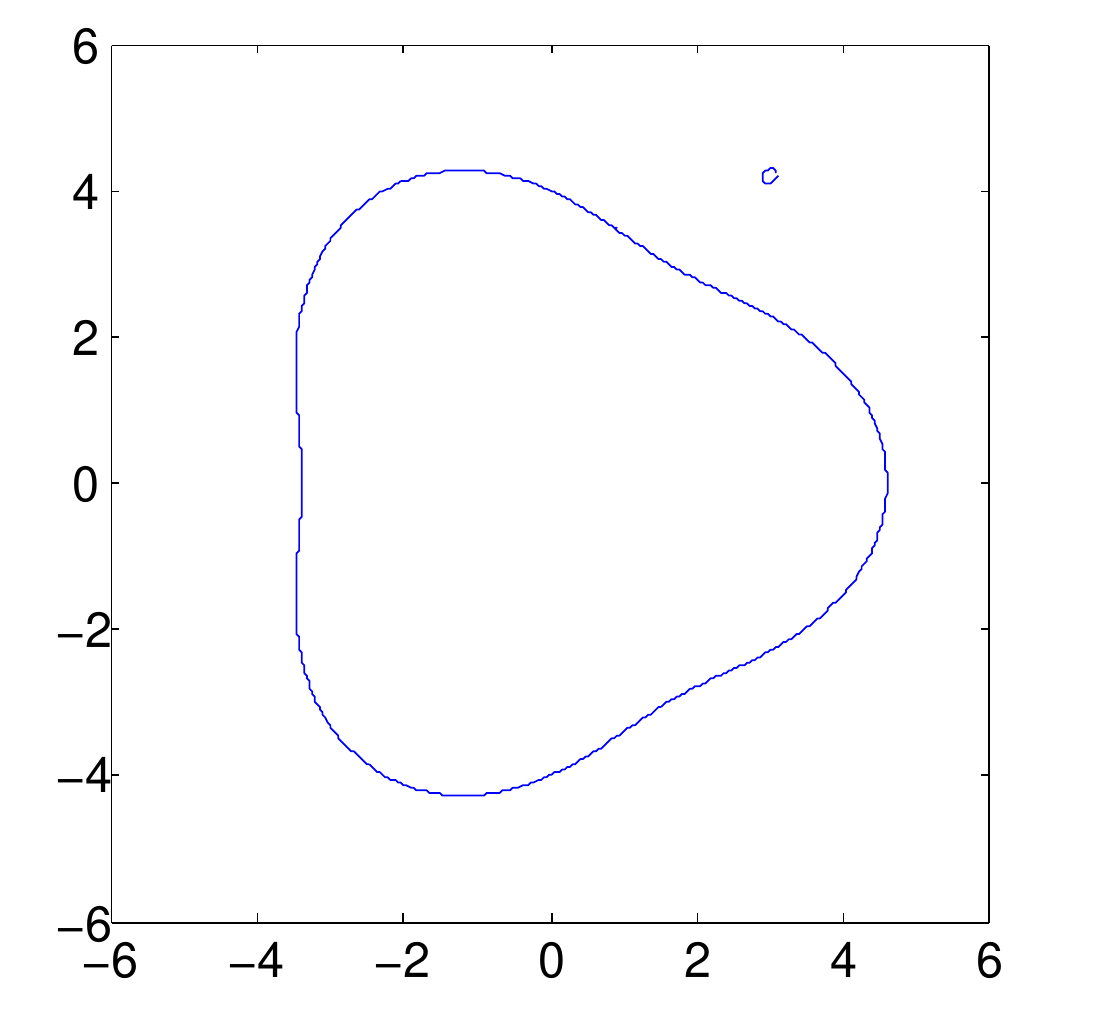}}
\subfigure[\textbf{$I_{\rm SS}^2$}]{\includegraphics[width=0.24\textwidth]{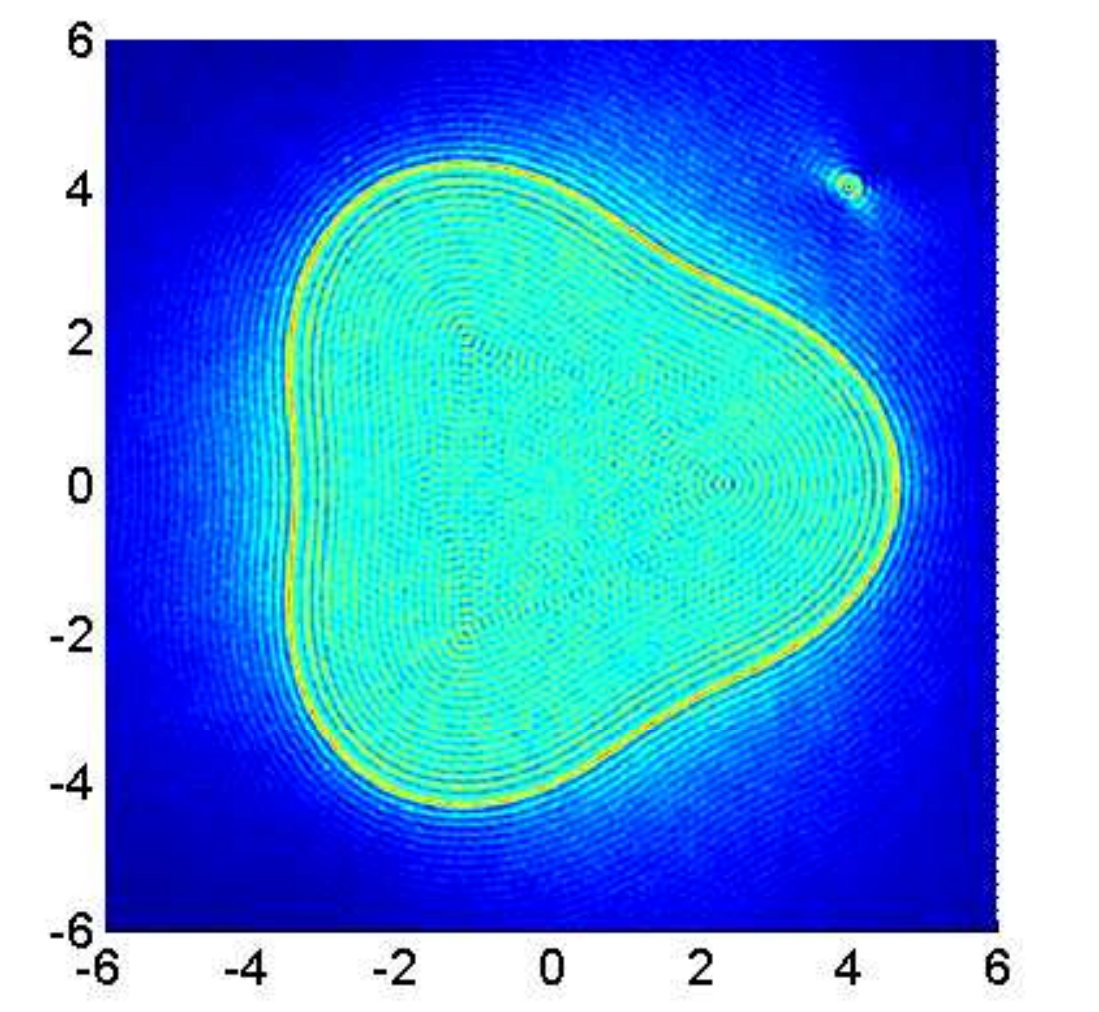}}
\subfigure[\textbf{$I_{\rm PP}^2$}]{\includegraphics[width=0.24\textwidth]{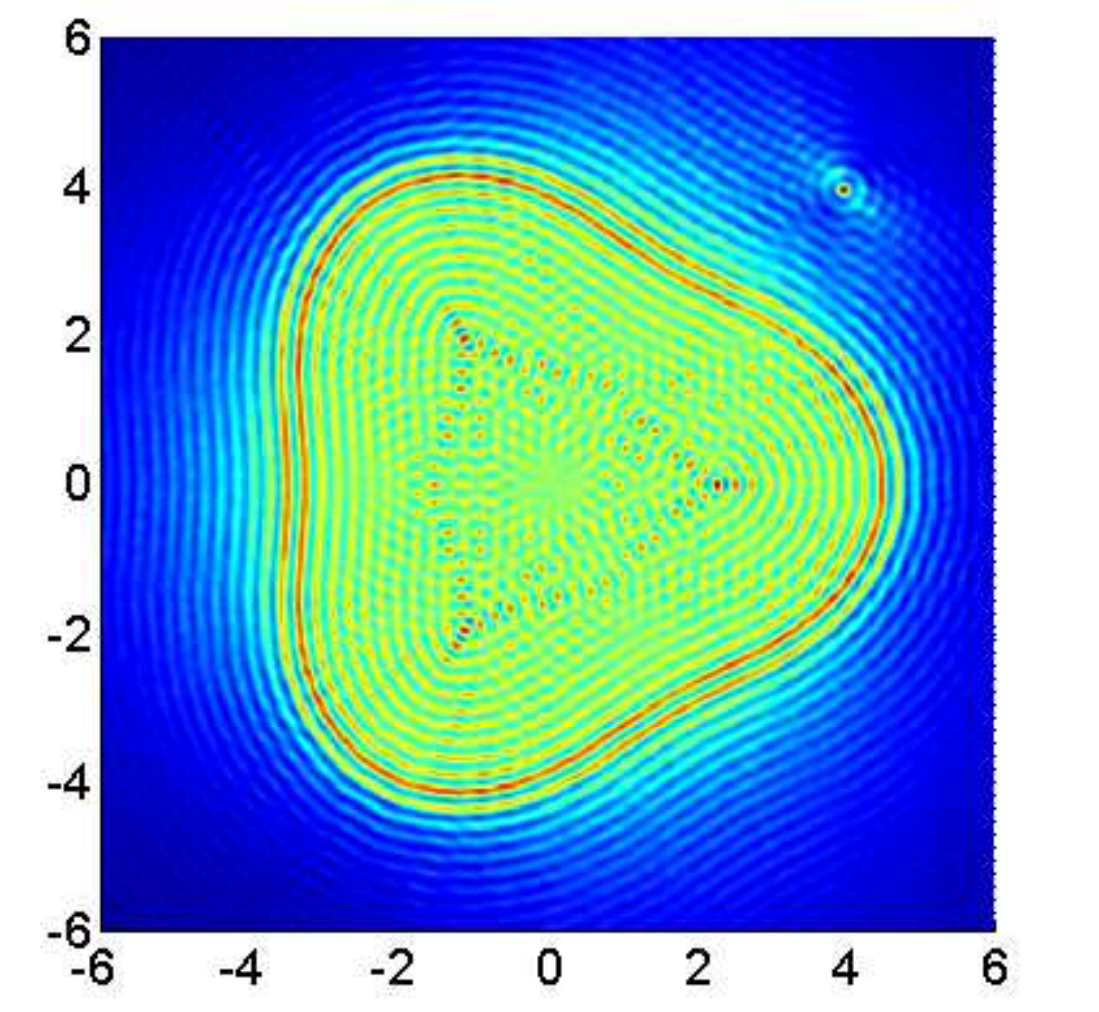}}
\subfigure[\textbf{$I_{\rm FF}^2$}]{\includegraphics[width=0.24\textwidth]{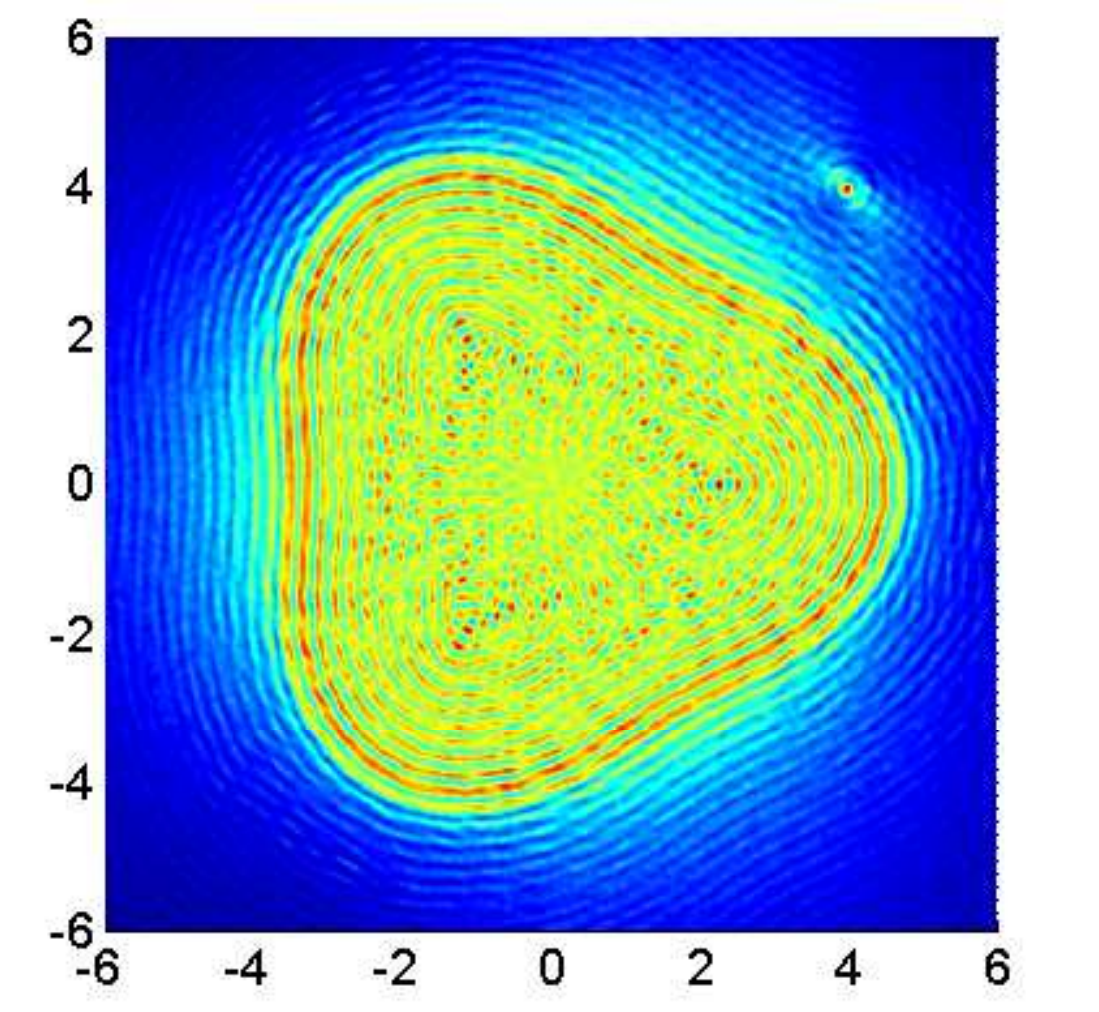}}
\caption{{\bf Example Multiplescalar.} Reconstructions for \textbf{pear-disk} scatterers with $30\%$ noise.}
\label{fig:pear_circle_n1024_noise30}
\end{figure}

\subsection{Example Resolutionlimit}
As shown in Figure \ref{fig:circle_kite_n1024_noise30}(a), the underling scatterer $\Om$ is given as the union of two disjoint obstacles
$\Om=\Om_{1}\cup\Om_{2}$ where $\Om_{1}$ is a big disk centered at $(-2,0)$ with radius $\rho=3$, while $\Om_{2}$ is a kite shaped domain centered at $(2.75,0)$.
These two disjoint components are very close to each other.
Again, we impose Dirichlet boundary condition on the boundary $\pa \Om$.
The results are shown in Figure \ref{fig:circle_kite_n1024_noise30}, from which we observe that the gap appears clearly.\\

\begin{figure}
\centering
\subfigure[\textbf{True domain}]{\includegraphics[width=0.24\textwidth]{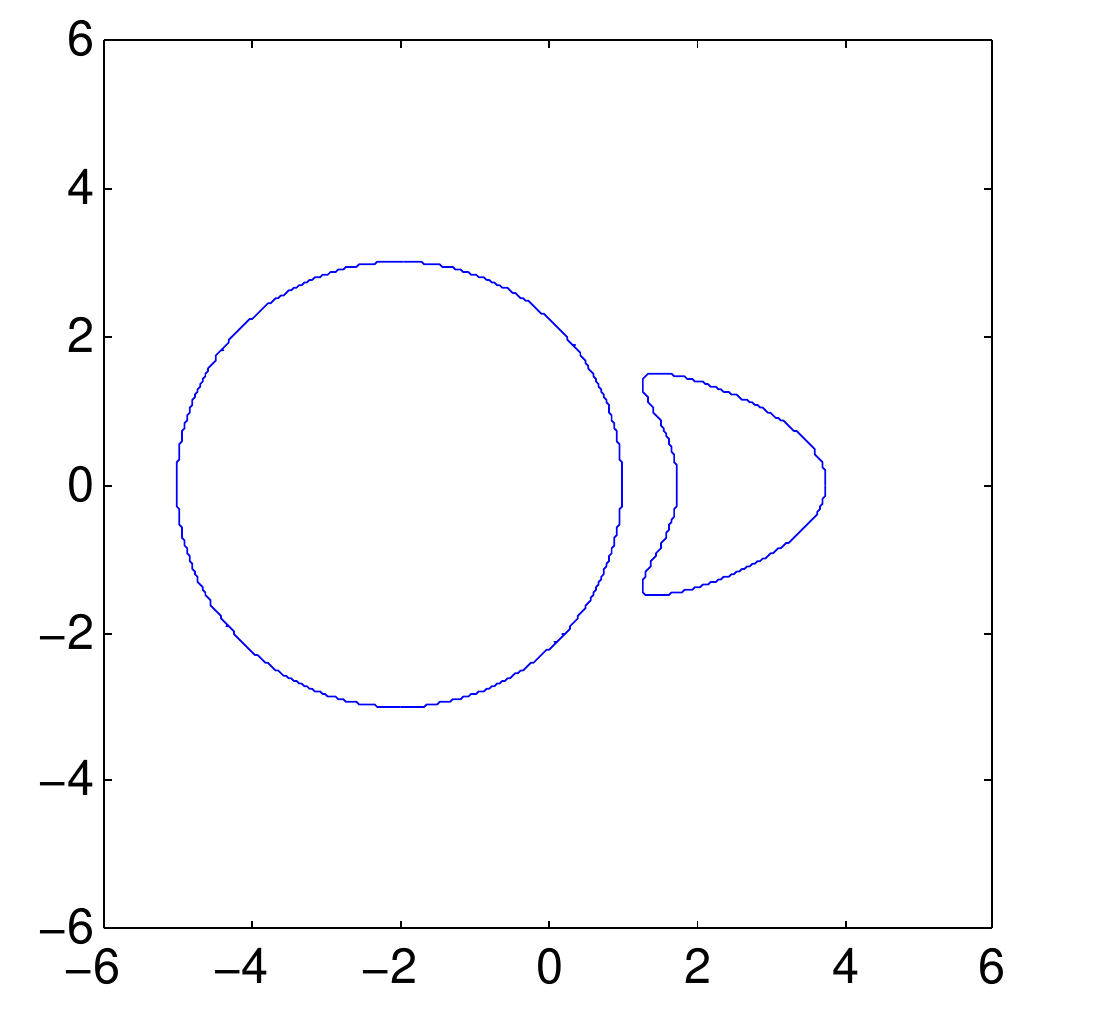}}
\subfigure[\textbf{$I_{\rm SS}^2$}]{\includegraphics[width=0.24\textwidth]{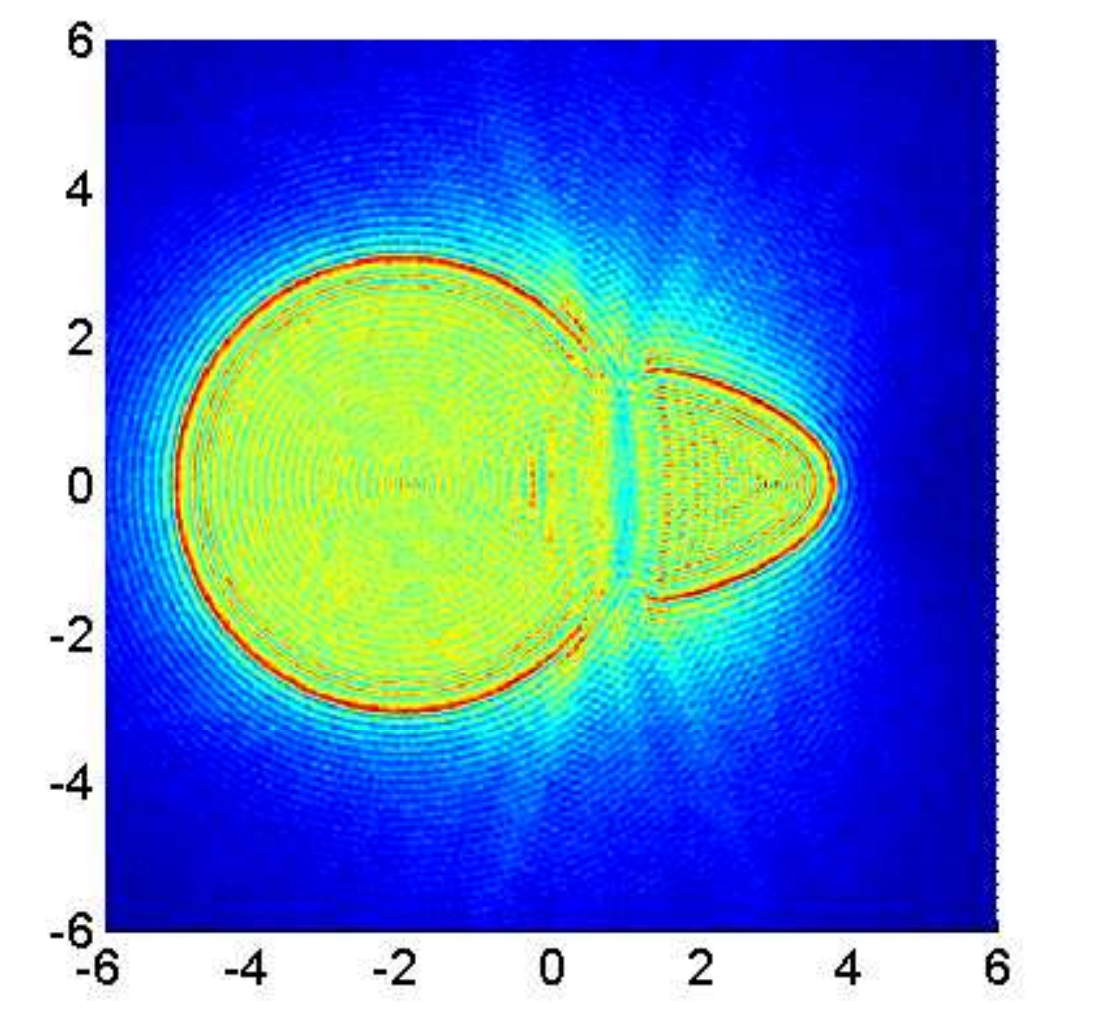}}
\subfigure[\textbf{$I_{\rm PP}^2$}]{\includegraphics[width=0.24\textwidth]{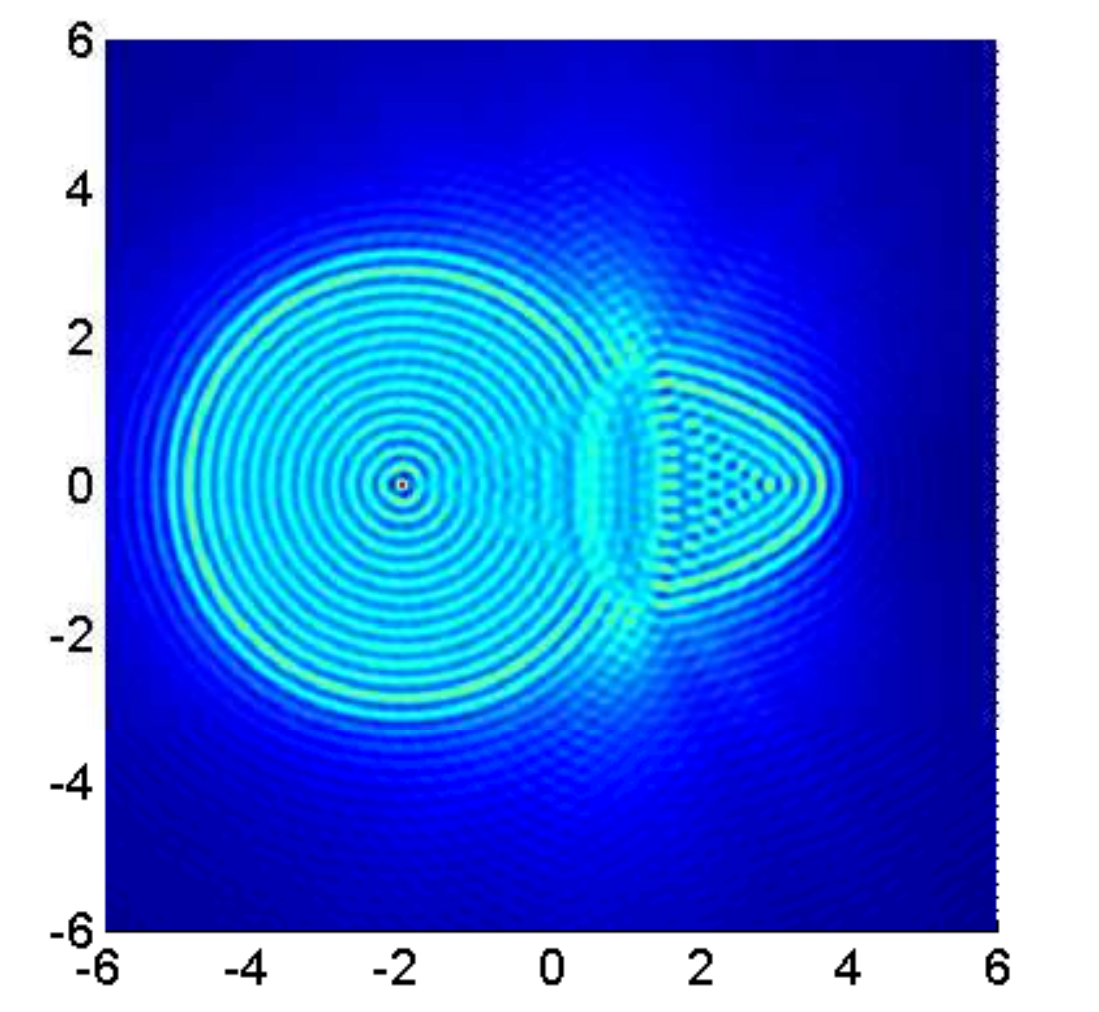}}
\subfigure[\textbf{$I_{\rm FF}^2$}]{\includegraphics[width=0.24\textwidth]{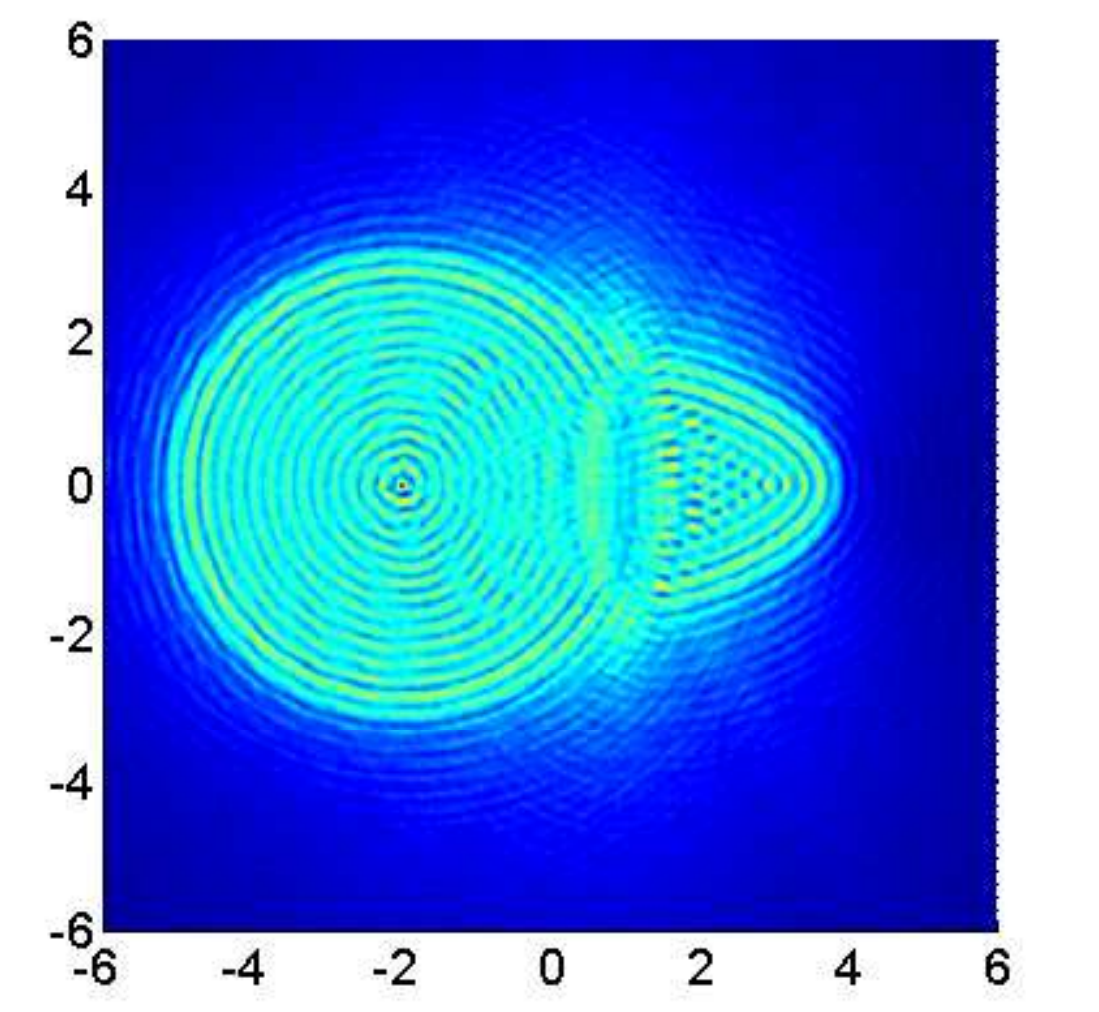}}
\caption{{\bf Example Resolutionlimit.} Reconstructions for \textbf{disk-kite} scatterers with $30\%$ noise.}
\label{fig:circle_kite_n1024_noise30}
\end{figure}
\subsection{Example Limited-aperture}
In the final example, we consider the limited-aperture problems. We take the benchmark test example, a kite shaped rigid body.
The kite shaped domain is non-convex. It has become a criterion to judge the quality of a reconstruction
method whether the concave part of the obstacle can be successfully recovered.
Figure \ref{fig:kite_limited_quarters} shows the reconstructions with different limited-aperture data.
It is clear to note that, even without noise, the results are worse than those with full-aperture data.
However, both the location and shape of the kite are still roughly captured.
As expected, the illuminated part is well reconstructed, which is as good as the one with full-aperture data,
while the "shadow region" is indeed elongated because of the limited data aperture and the high frequency $\om=8\pi$.
We also observe, from the 2nd and 3rd columns of Figure \ref{fig:kite_limited_quarters},
that the concave part can be well reconstructed as long as it is covered by the observation angles.

\begin{figure}
\centering
\subfigure[$(0,\pi/2)$]{\includegraphics[width=0.24\textwidth]{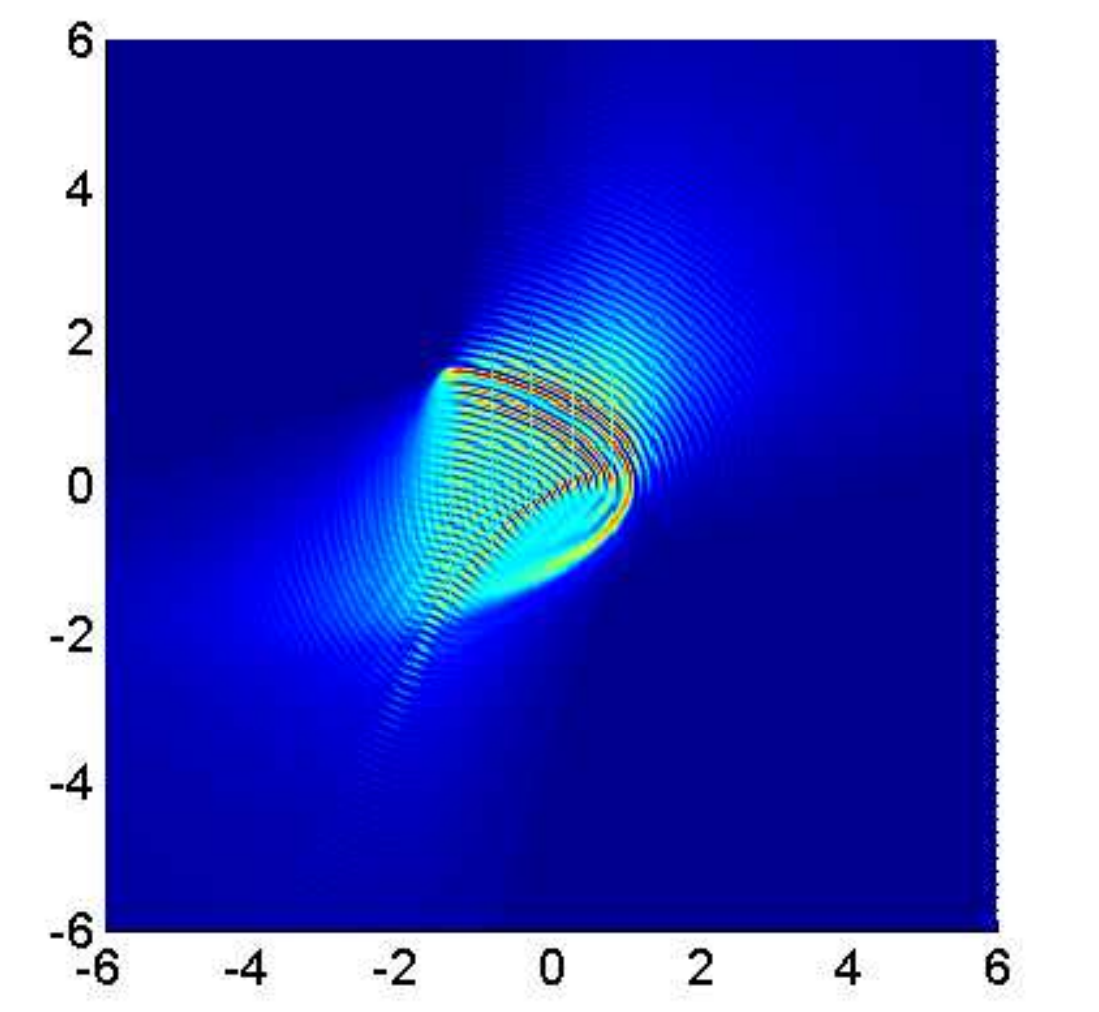}}
\subfigure[$(\pi/2,\pi)$]{\includegraphics[width=0.24\textwidth]{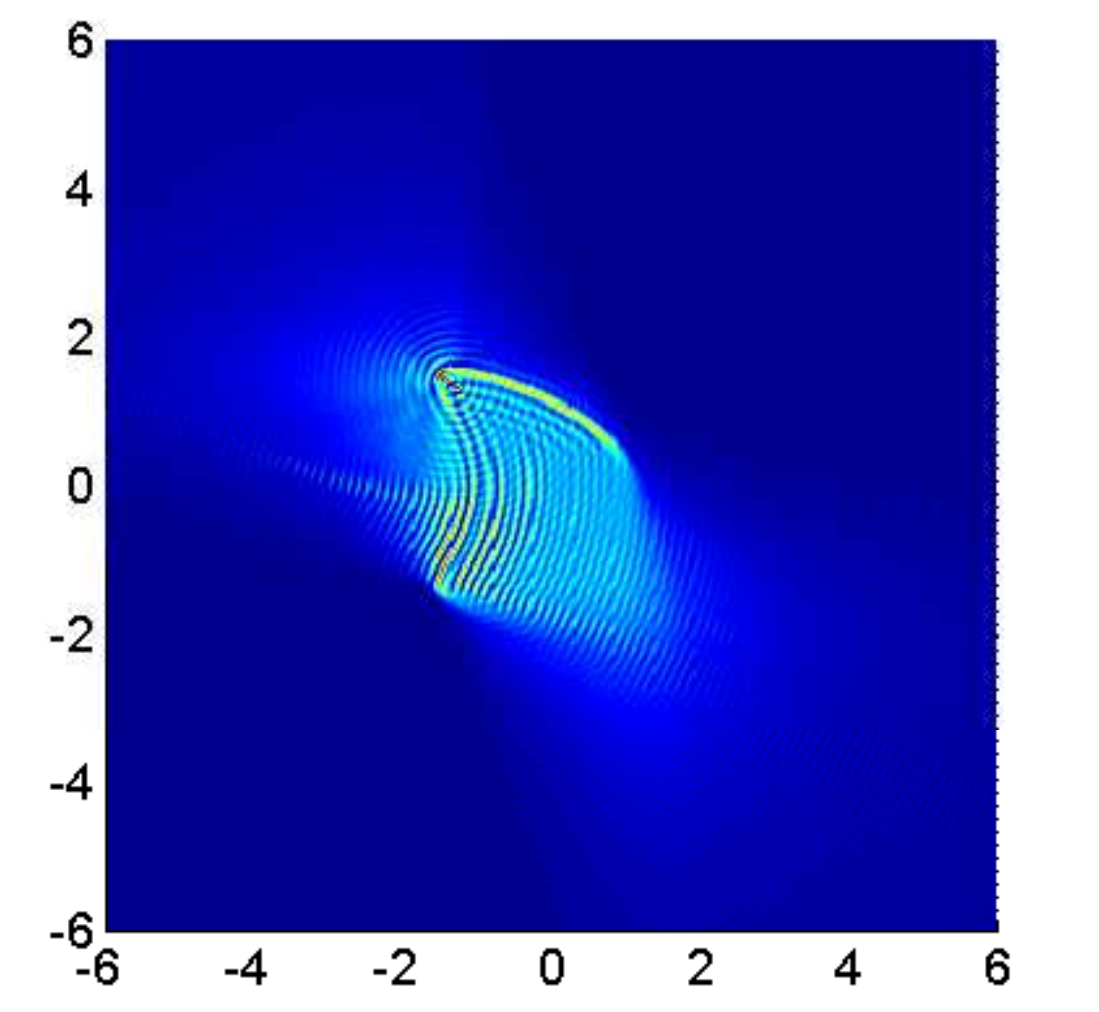}}
\subfigure[$(\pi,3\pi/2)$]{\includegraphics[width=0.24\textwidth]{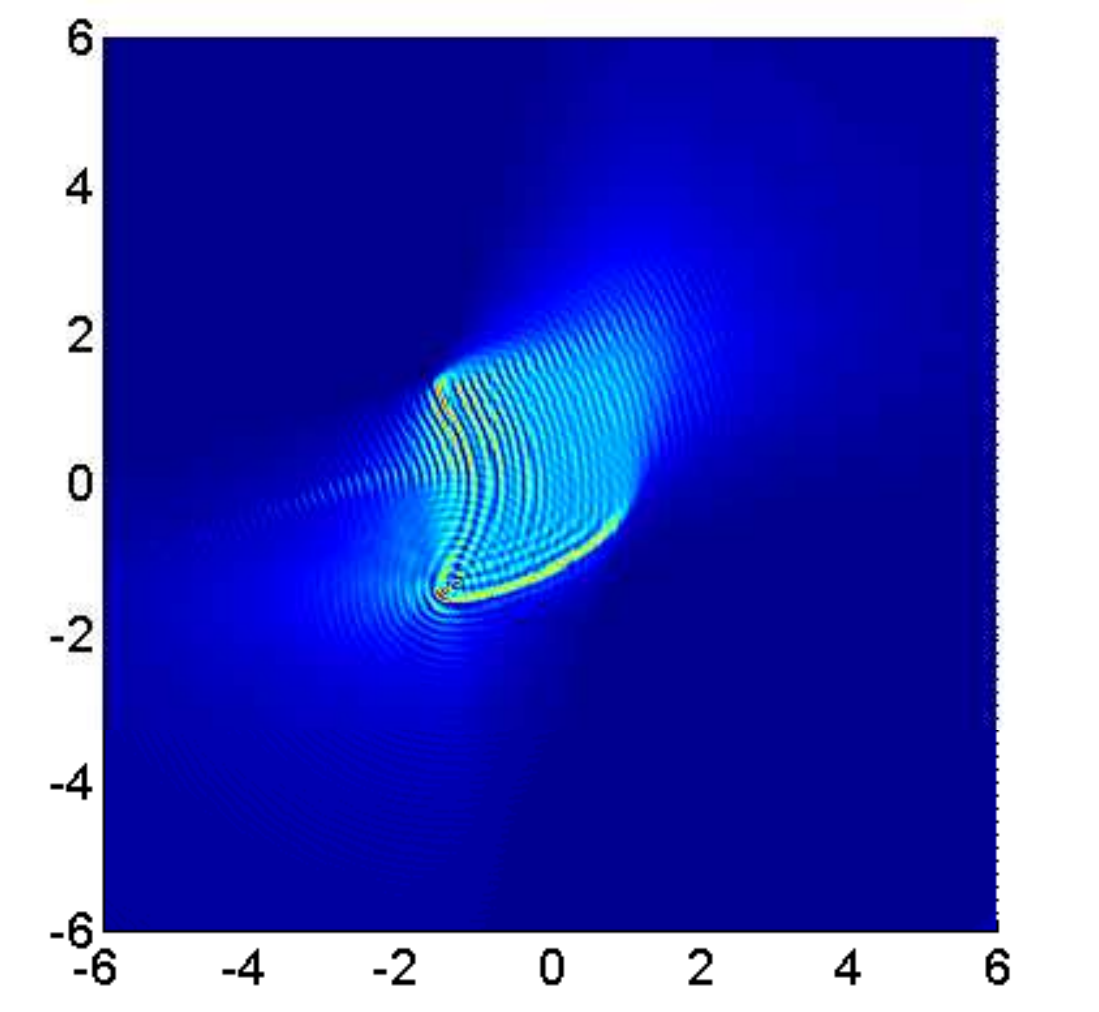}}
\subfigure[$(3\pi/2,2\pi)$]{\includegraphics[width=0.24\textwidth]{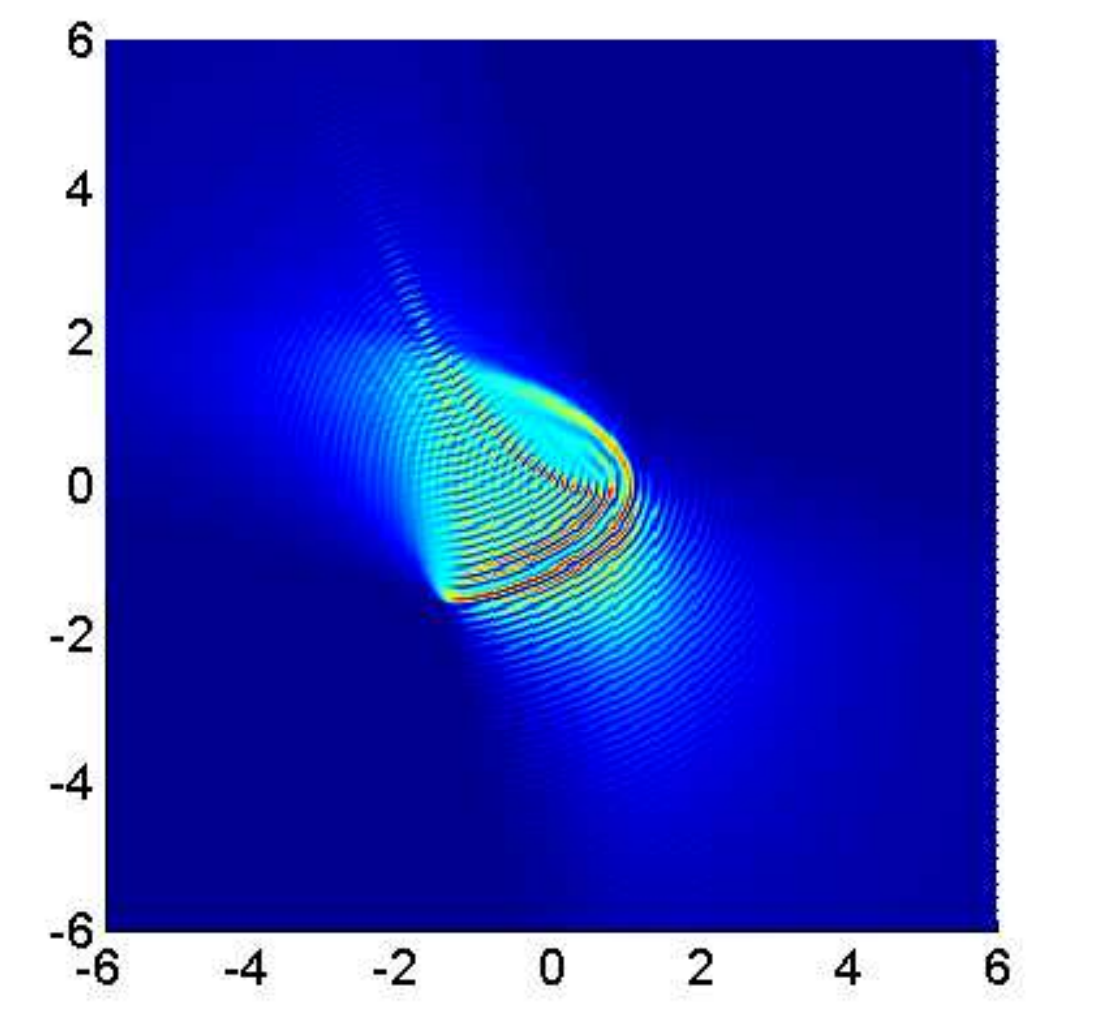}}
\subfigure[$(0,\pi/2)$]{\includegraphics[width=0.24\textwidth]{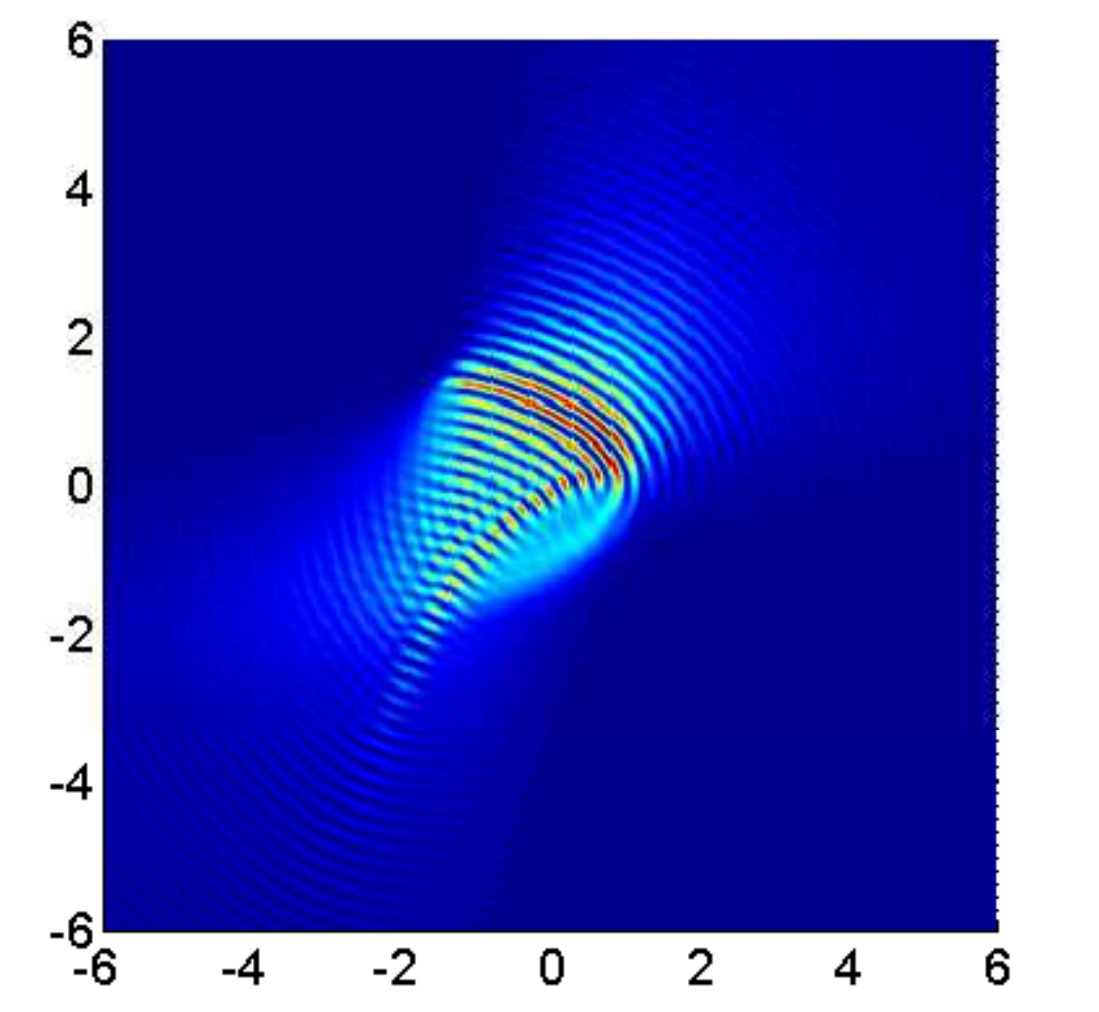}}
\subfigure[$(\pi/2,\pi)$]{\includegraphics[width=0.24\textwidth]{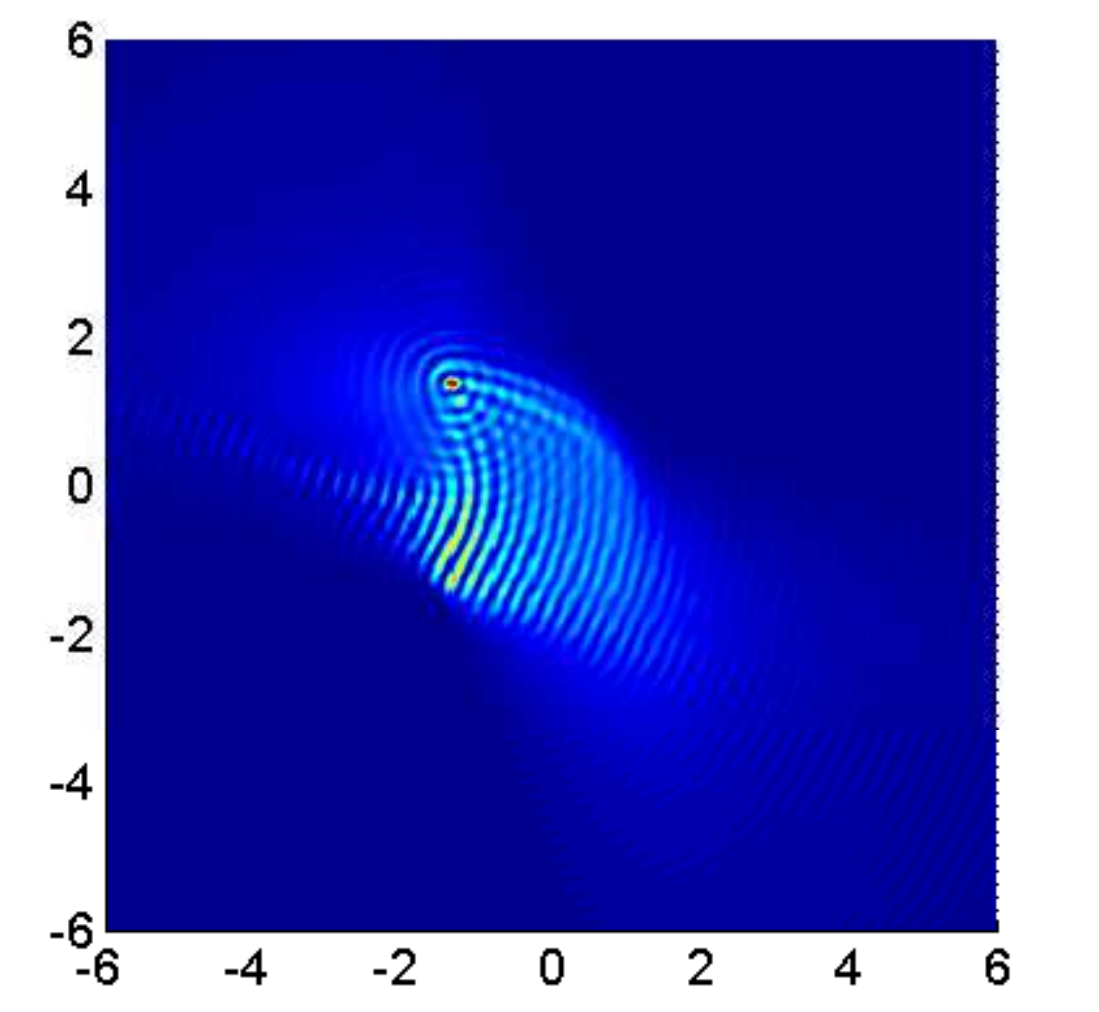}}
\subfigure[$(\pi,3\pi/2)$]{\includegraphics[width=0.24\textwidth]{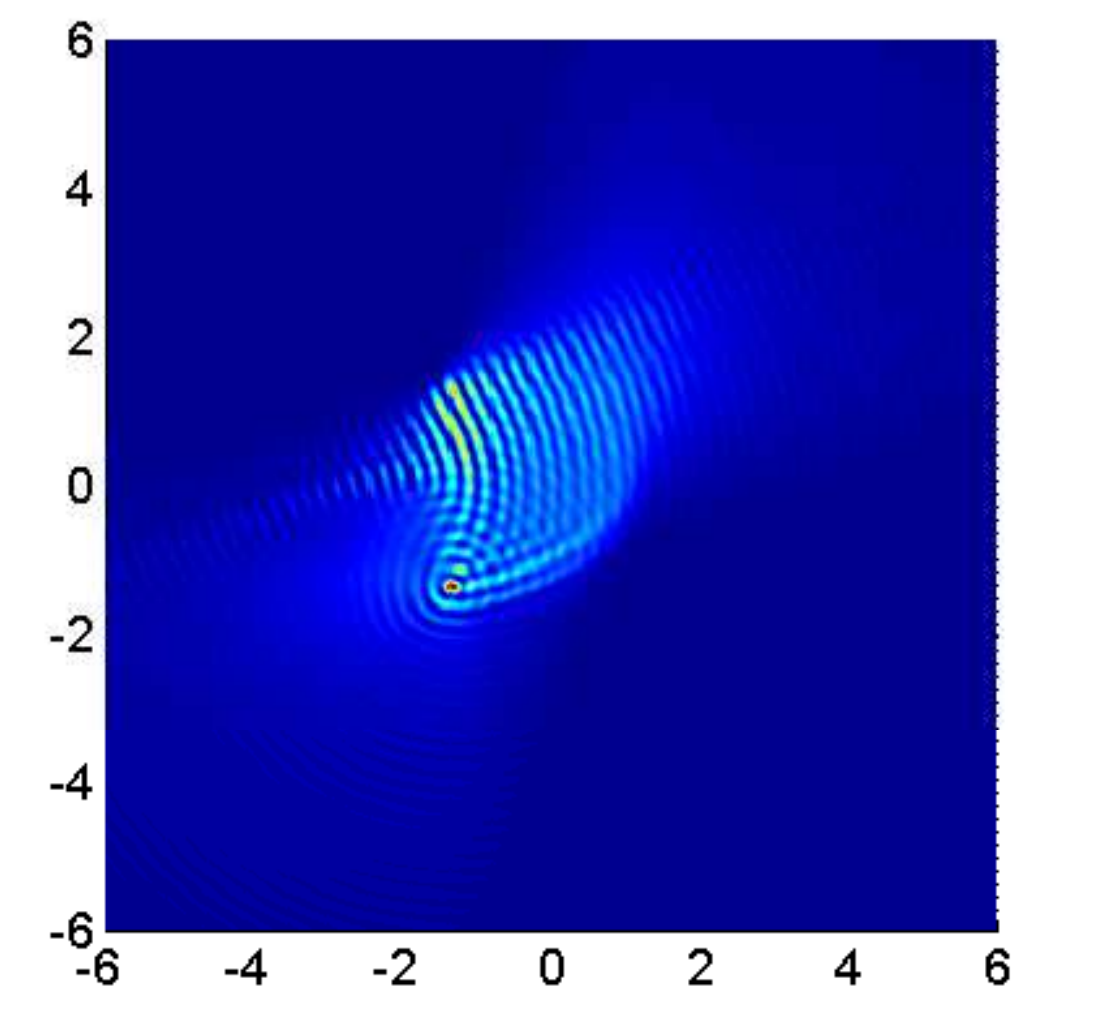}}
\subfigure[$(3\pi/2,2\pi)$]{\includegraphics[width=0.24\textwidth]{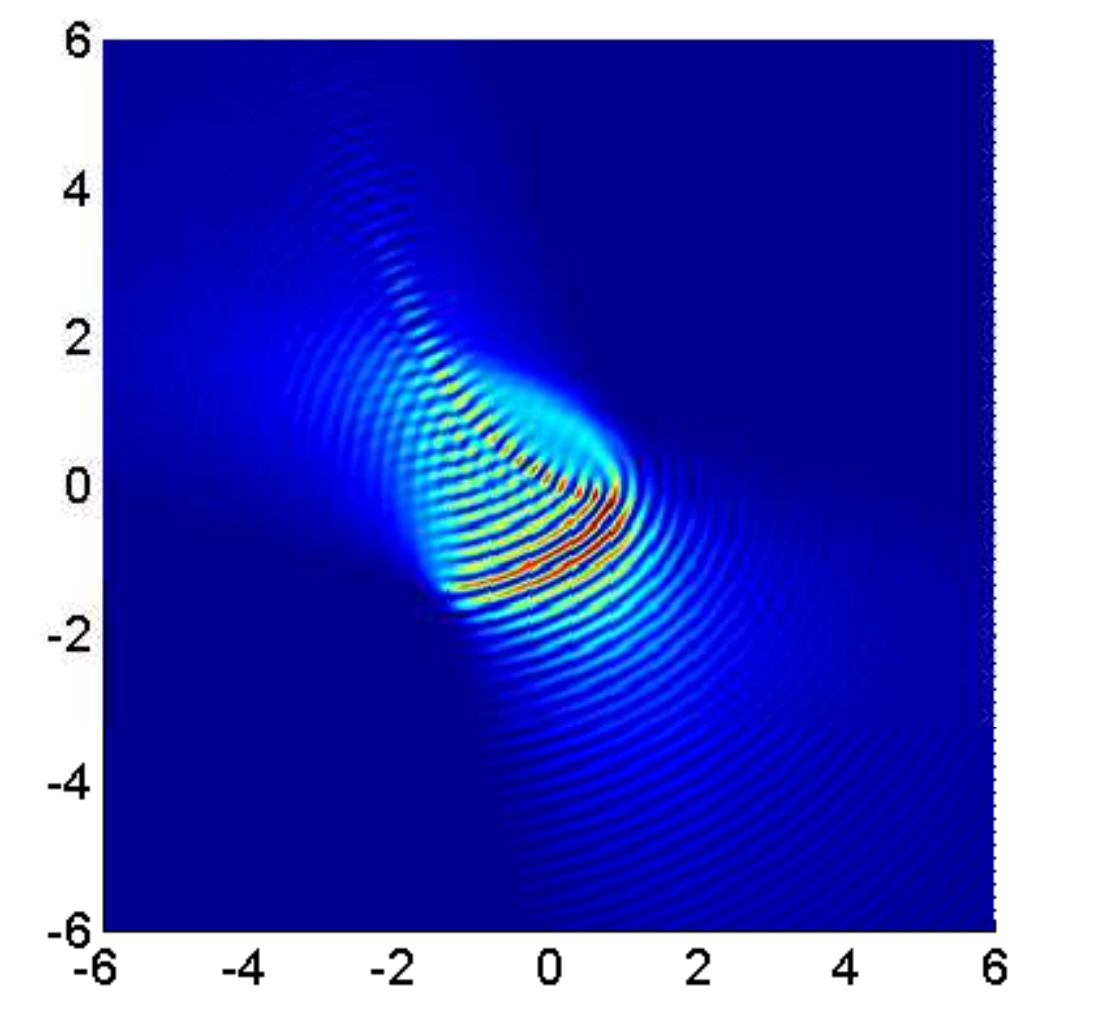}}
\subfigure[$(0,\pi/2)$]{\includegraphics[width=0.24\textwidth]{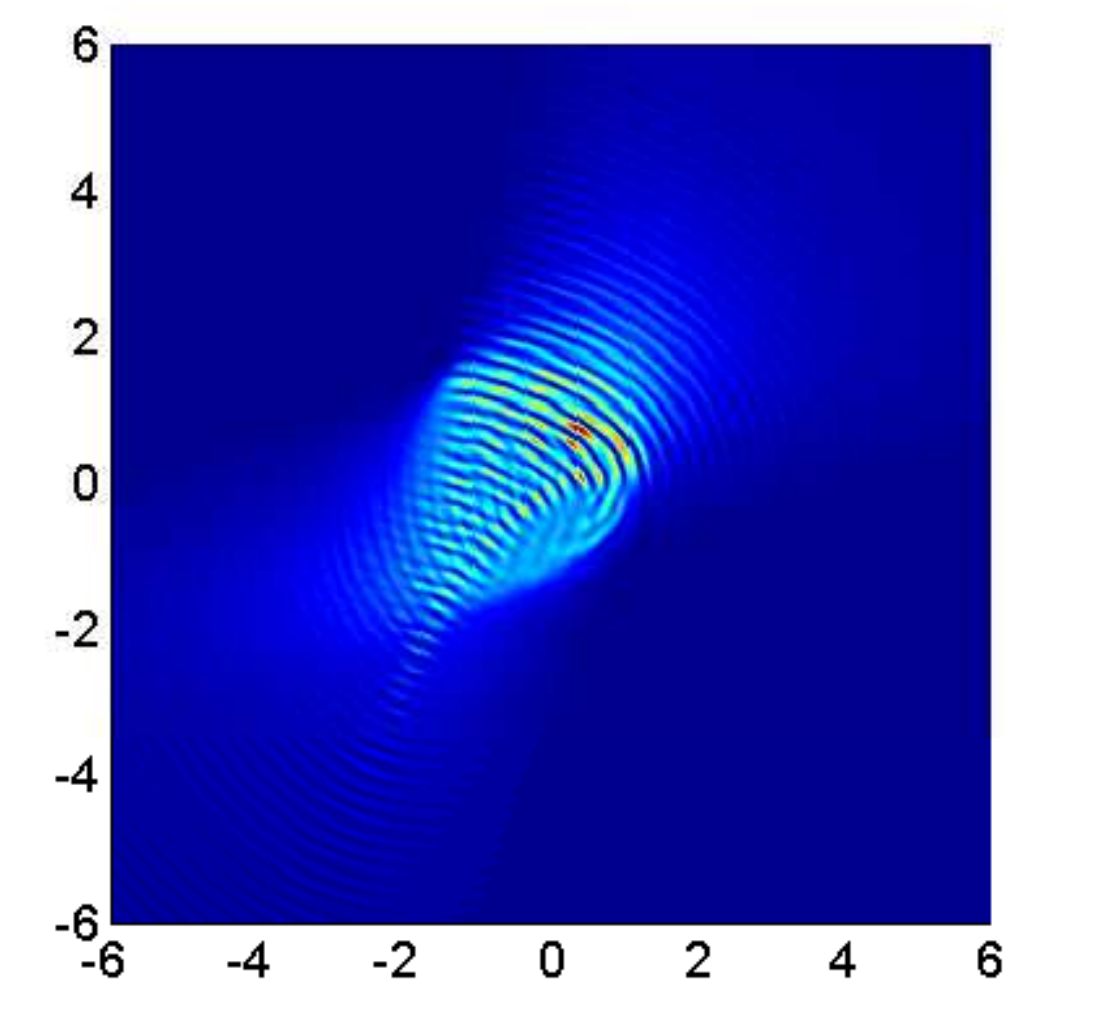}}
\subfigure[$(\pi/2,\pi)$]{\includegraphics[width=0.24\textwidth]{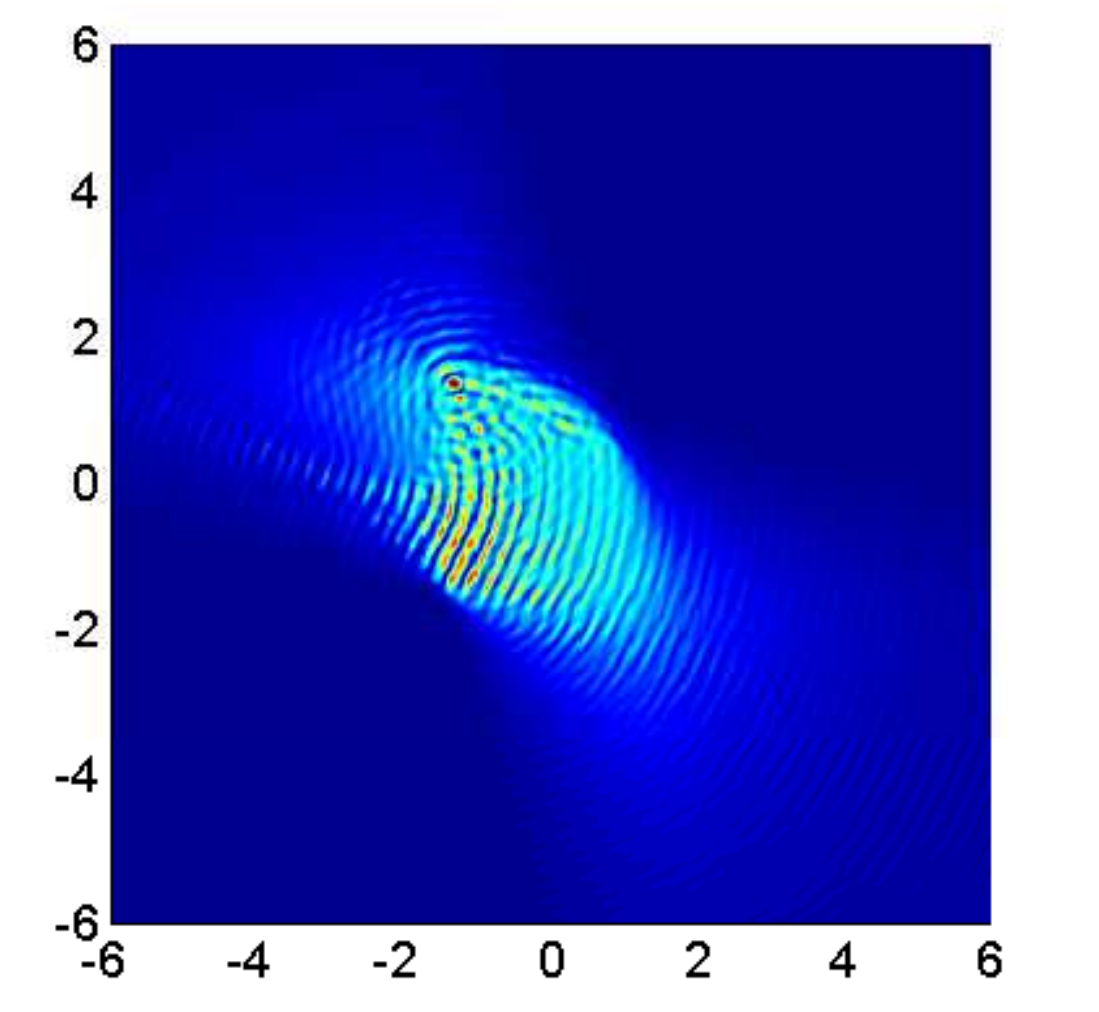}}
\subfigure[$(\pi,3\pi/2)$]{\includegraphics[width=0.24\textwidth]{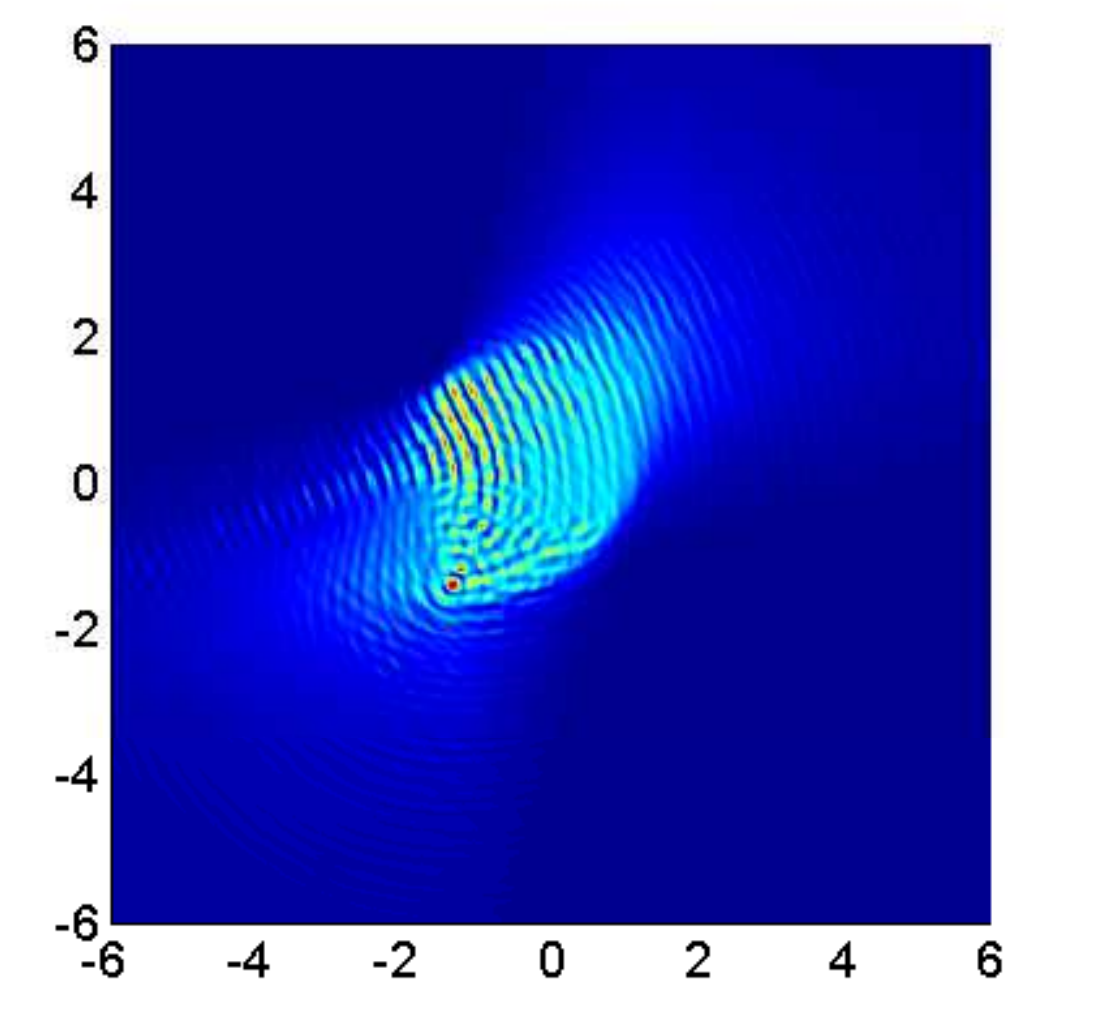}}
\subfigure[$(3\pi/2,2\pi)$]{\includegraphics[width=0.24\textwidth]{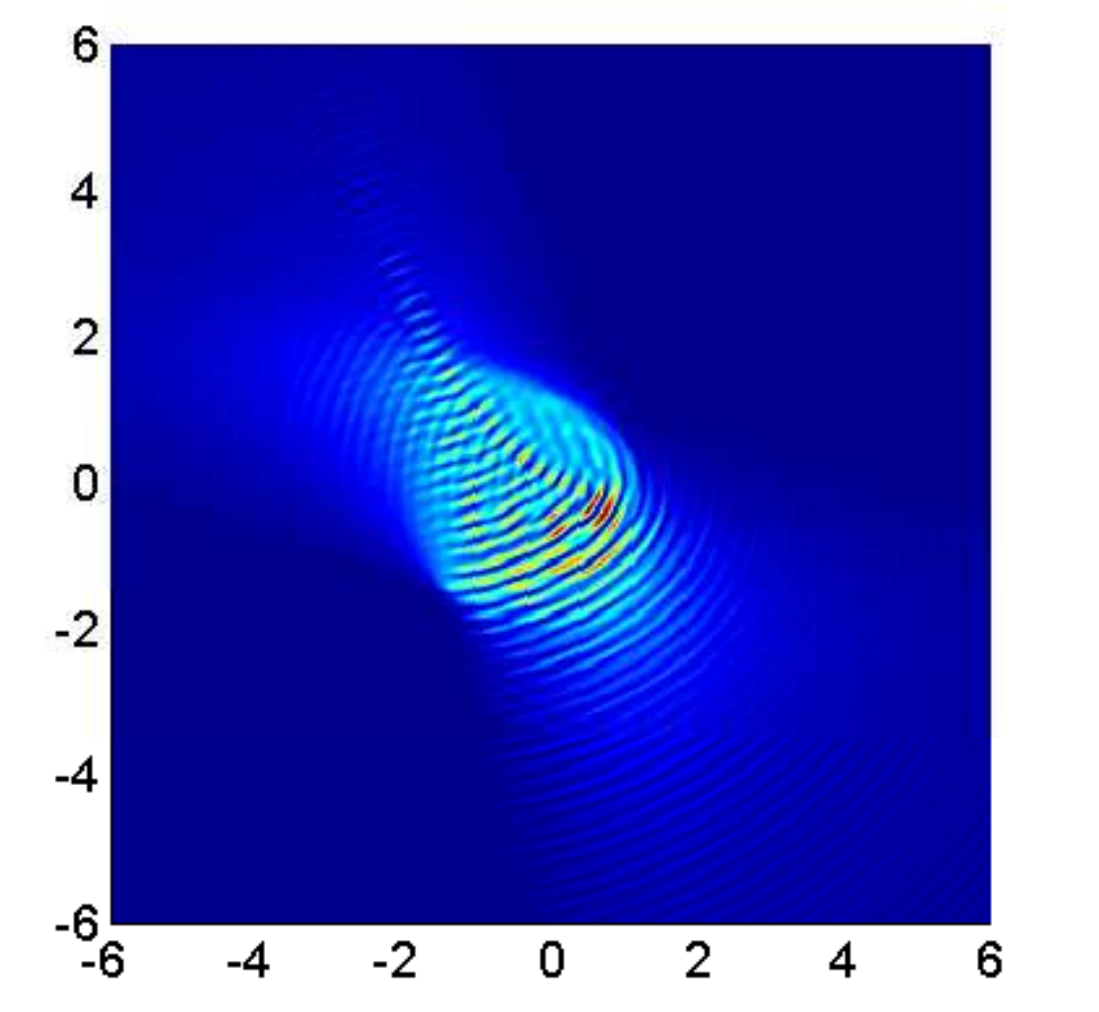}}
\caption{{\bf Example Limited-aperture.} Reconstructions of kite shaped domain with limited-aperture data and no noise. Top row: SS case. Middle row: PP case. Bottom row: FF case.}
\label{fig:kite_limited_quarters}
\end{figure}

If we fix the observation angles distributed over $[0,\pi/2]$.
To obtain more information of the "shadow region", we take, in the final example, a lower frequency $\om=4\pi$.
The top row of Figure \ref{fig:limit_aperture_reconstrucion_I2} shows the results with $10\%$ noise.
Using the data retrieval techniques proposed in the subsection \ref{sec:limited-aperture}, we firstly retrieve those data with observation angles in $(\pi/2, 2\pi)$, and then
reconstruct the kite with the help of the proposed indicators. The artificial domain $B$ is chosen to be a ball centered at the origin with radius $5$.
The bellow row of Figure \ref{fig:limit_aperture_reconstrucion_I2} shows the corresponding results, which are indeed improved than those in the
top row. In particular, the concave part is successfully reconstructed by using the indicator $I_{\rm SS}$.

\begin{figure}[htbp]
\centering
\subfigure[\textbf{$I^{2,limit}_{\rm SS}(z)$}]
   {\includegraphics[width=0.32\textwidth]{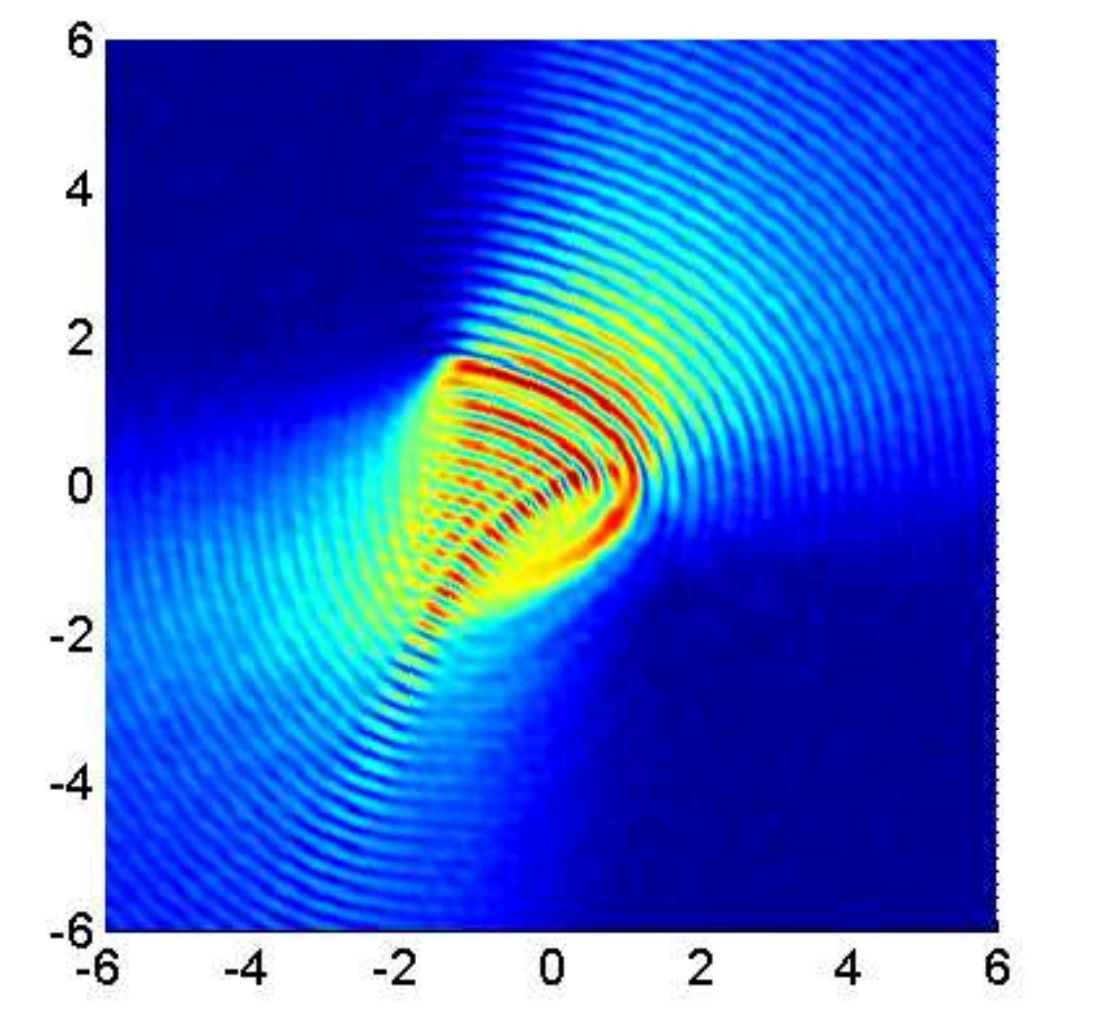}}
\subfigure[\textbf{$I^{2,limit}_{\rm PP}(z)$}]
   {\includegraphics[width=0.32\textwidth]{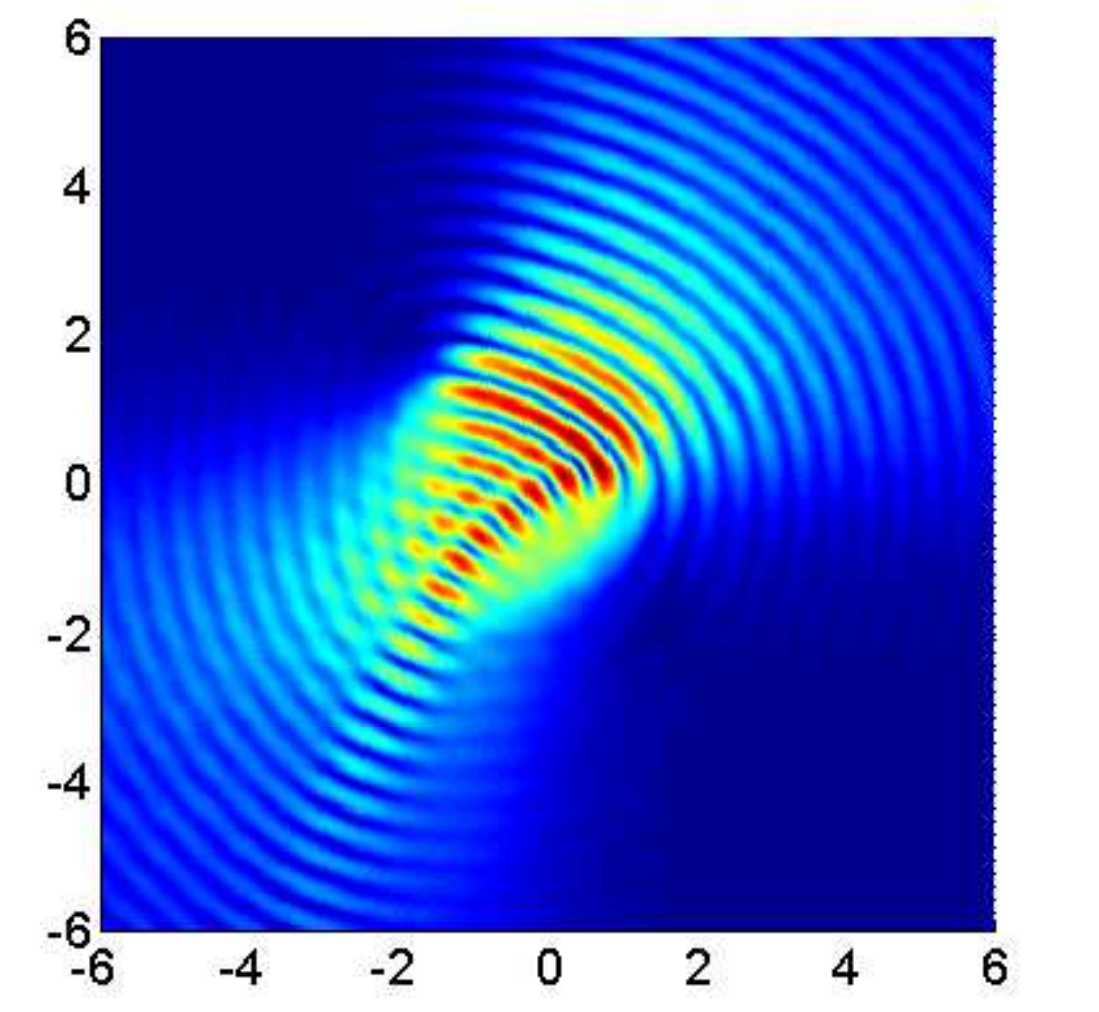}}
\subfigure[\textbf{$I^{2,limit}_{\rm FF}(z)$}]
   {\includegraphics[width=0.32\textwidth]{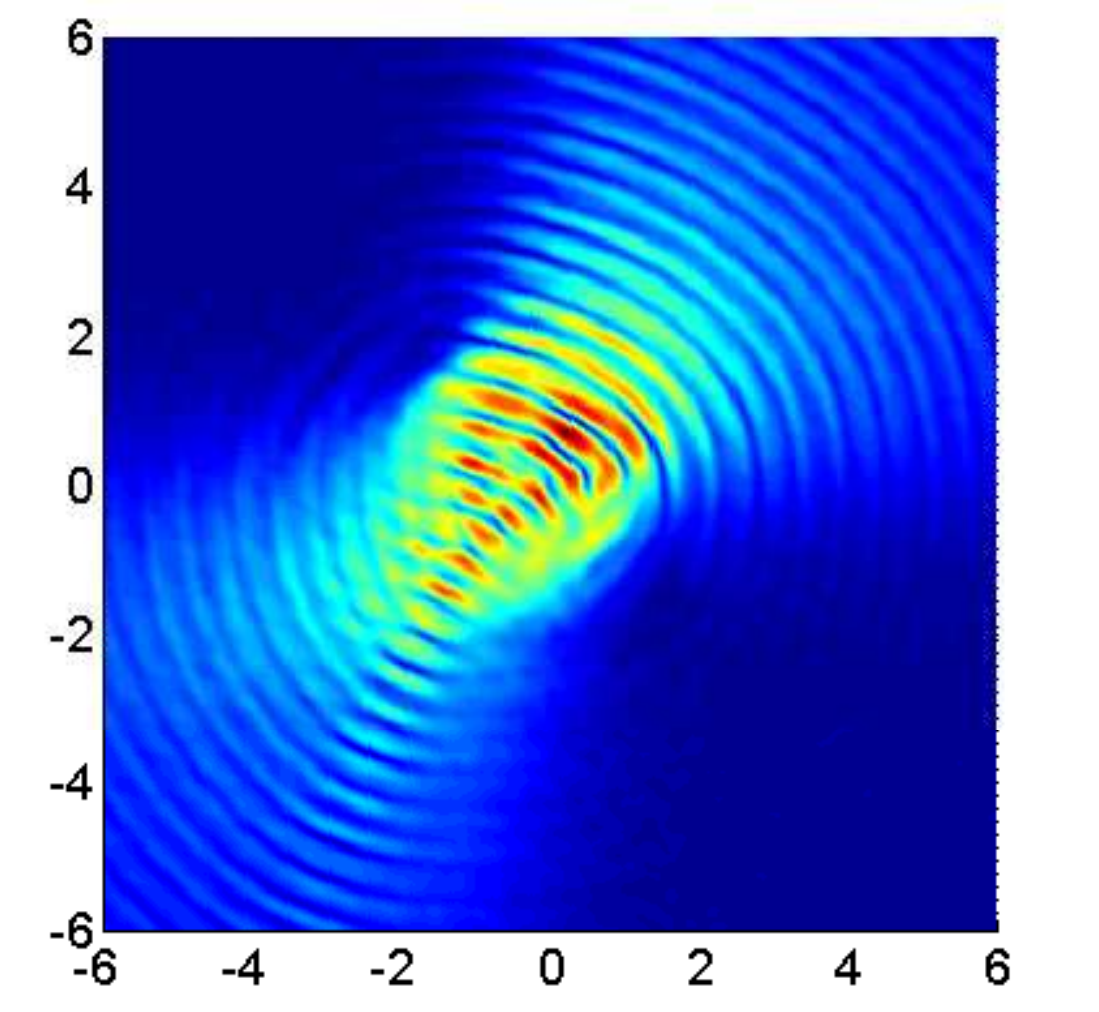}}
\subfigure[\textbf{$\tilde{I}^{2,full}_{\rm SS}(z)$}]
   {\includegraphics[width=0.32\textwidth]{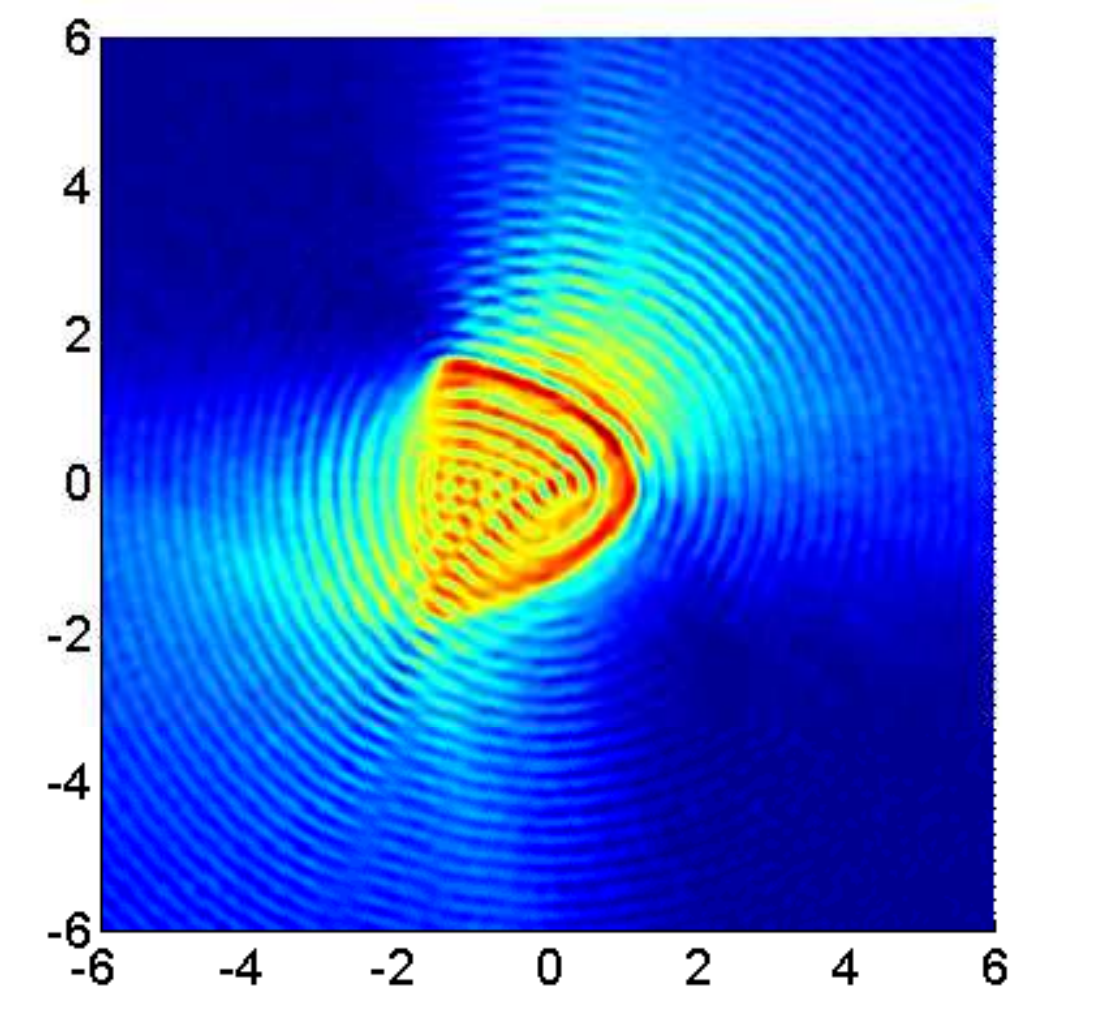}}
\subfigure[\textbf{$\tilde{I}^{2,full}_{\rm PP}(z)$}]
   {\includegraphics[width=0.32\textwidth]{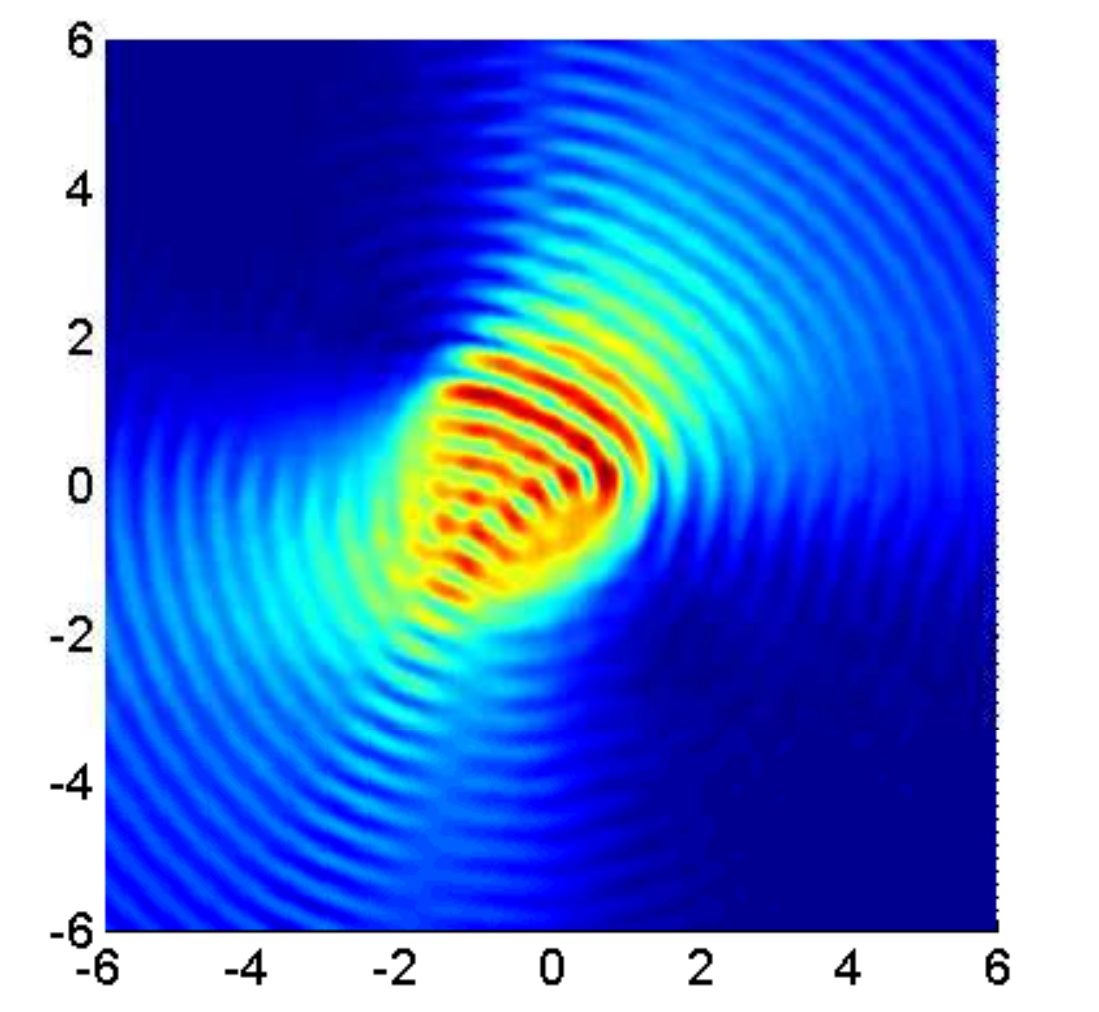}}
\subfigure[\textbf{$\tilde{I}^{2,full}_{\rm FF}(z)$}]
   {\includegraphics[width=0.32\textwidth]{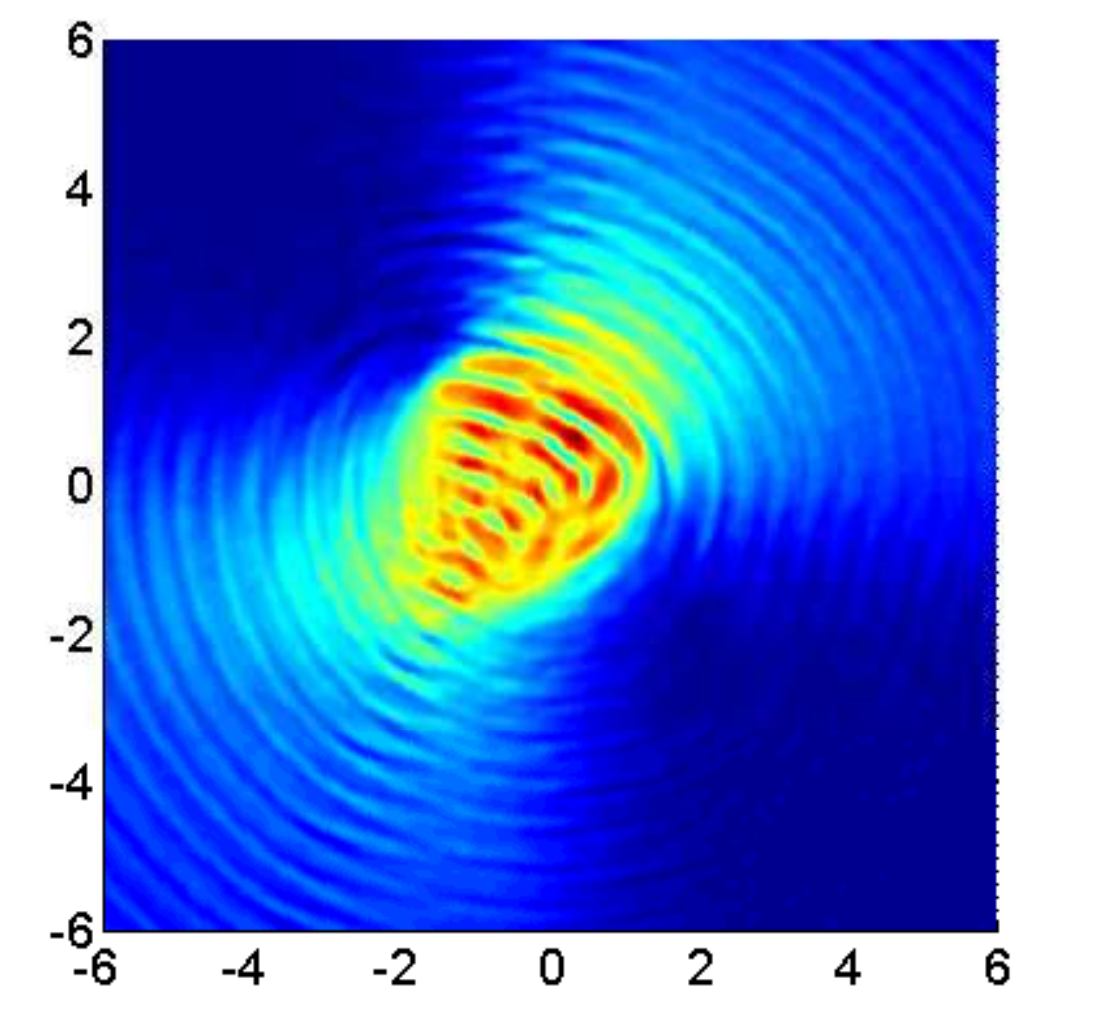}}
\caption{{\bf Example Limited-aperture.} Reconstructions of kite-shaped domain with $10\%$ noise and $w=4\pi$.}
\label{fig:limit_aperture_reconstrucion_I2}
\end{figure}

Finally, we consider the cases with one or several incident directions. The corresponding uniqueness of the inverse problem is still open now.
Figure \ref{fig:limit_aperture_reconstrucion_I2_less_incidentDirection}(a)-(d) show the results with only one incident direction. We observe that the location
is roughly captured. However, the shape information fails to be reconstructed.
With the increase of the incident directions, the shape is well reconstructed better and better, see Figure
\ref{fig:limit_aperture_reconstrucion_I2_less_incidentDirection}(e)-(h).
\begin{figure}[htbp]
\centering
\subfigure[\textbf{$\mathbf{d}=(1,0)^{T}$}]
   {\includegraphics[width=0.24\textwidth]{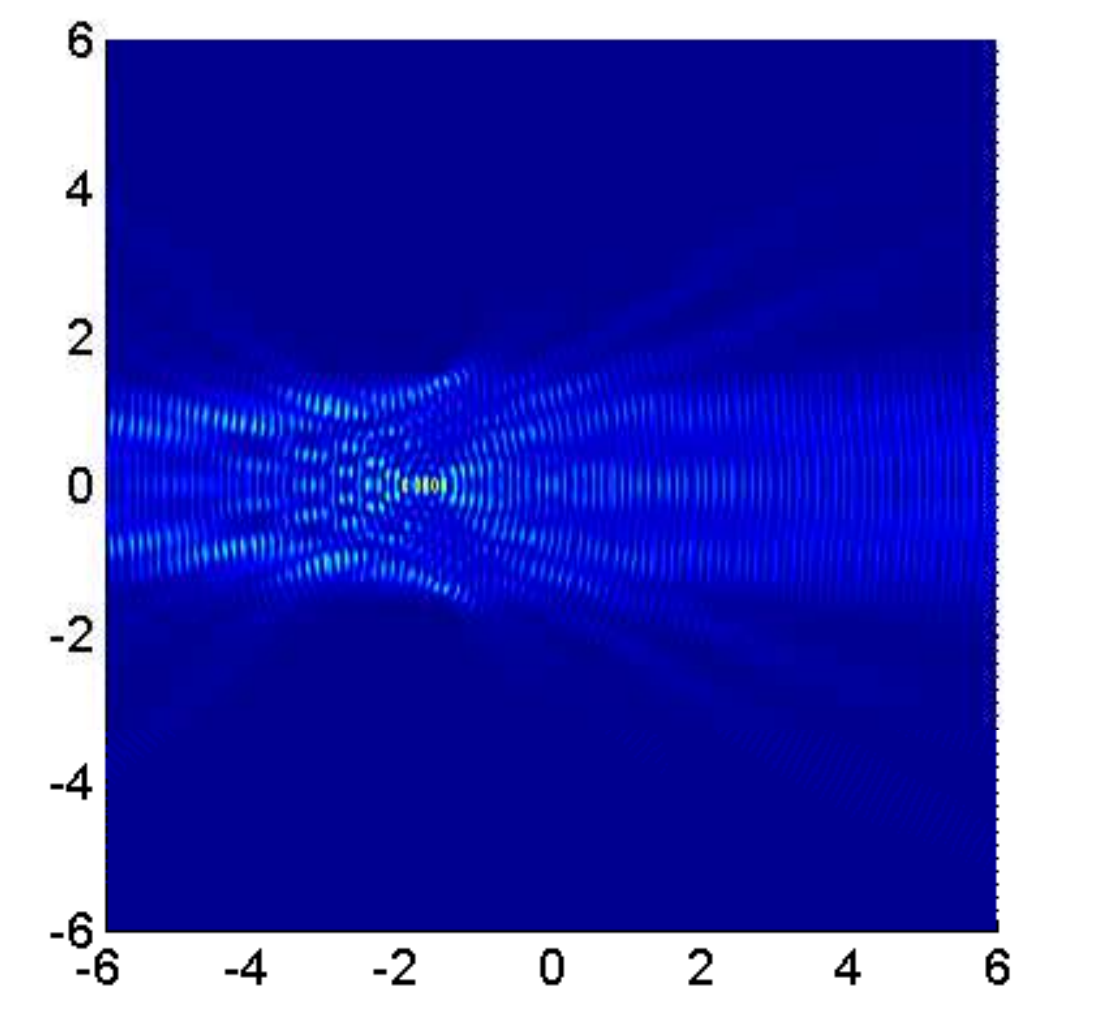}}
\subfigure[\textbf{$\mathbf{d}=(0,1)^{T}$}]
   {\includegraphics[width=0.24\textwidth]{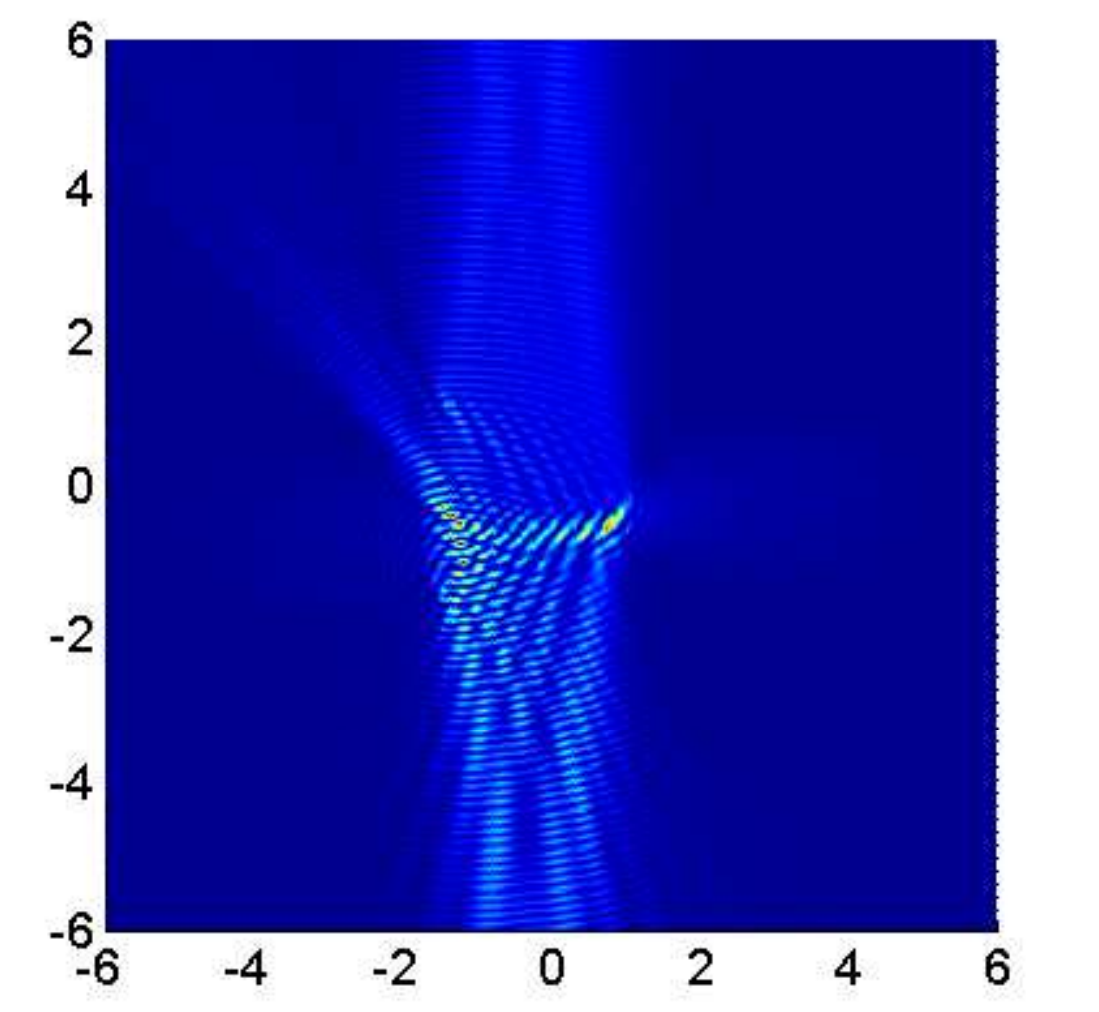}}
\subfigure[\textbf{$\mathbf{d}=(-1,0)^{T}$}]
   {\includegraphics[width=0.24\textwidth]{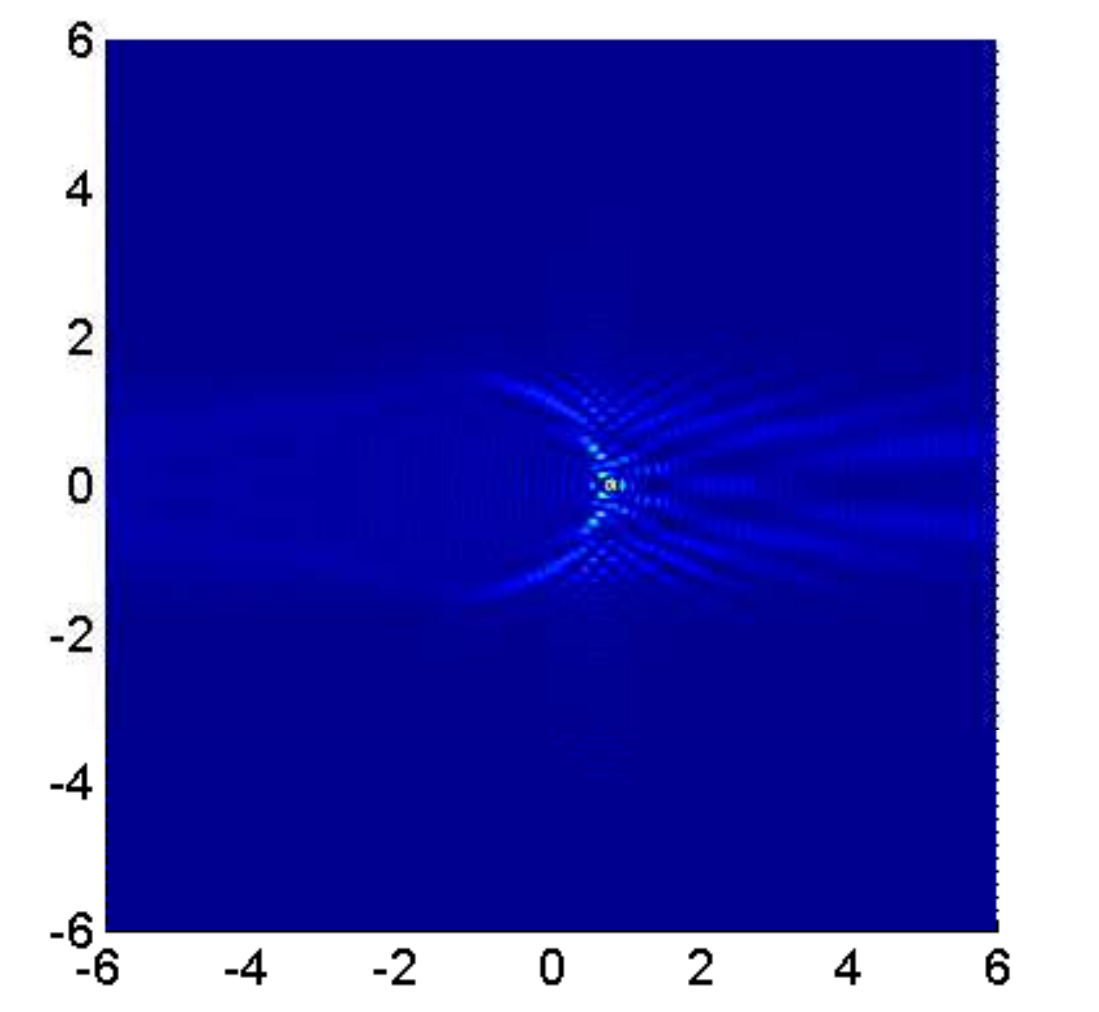}}
\subfigure[\textbf{$\mathbf{d}=(0,-1)^{T}$}]
   {\includegraphics[width=0.24\textwidth]{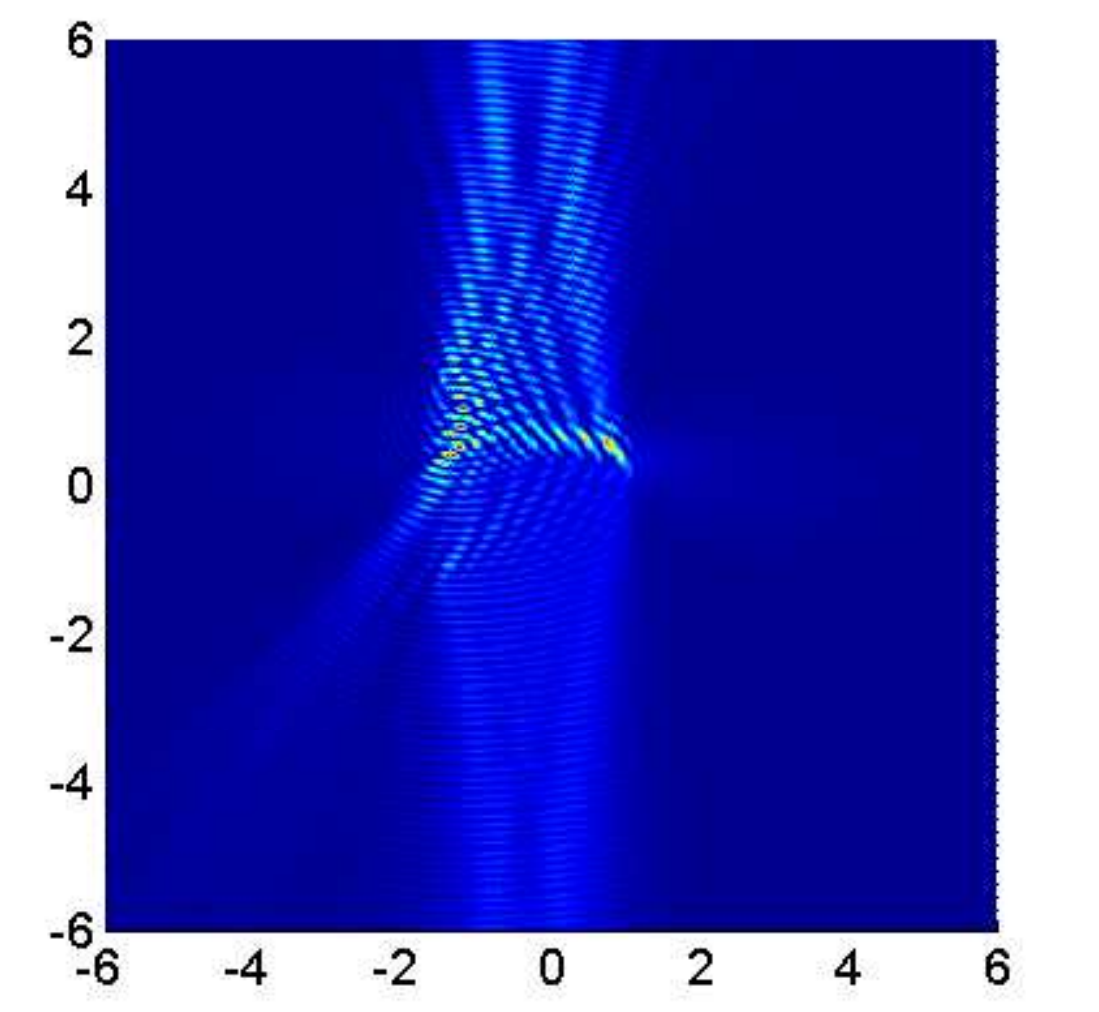}}
\subfigure[\textbf{2 directions}]
   {\includegraphics[width=0.24\textwidth]{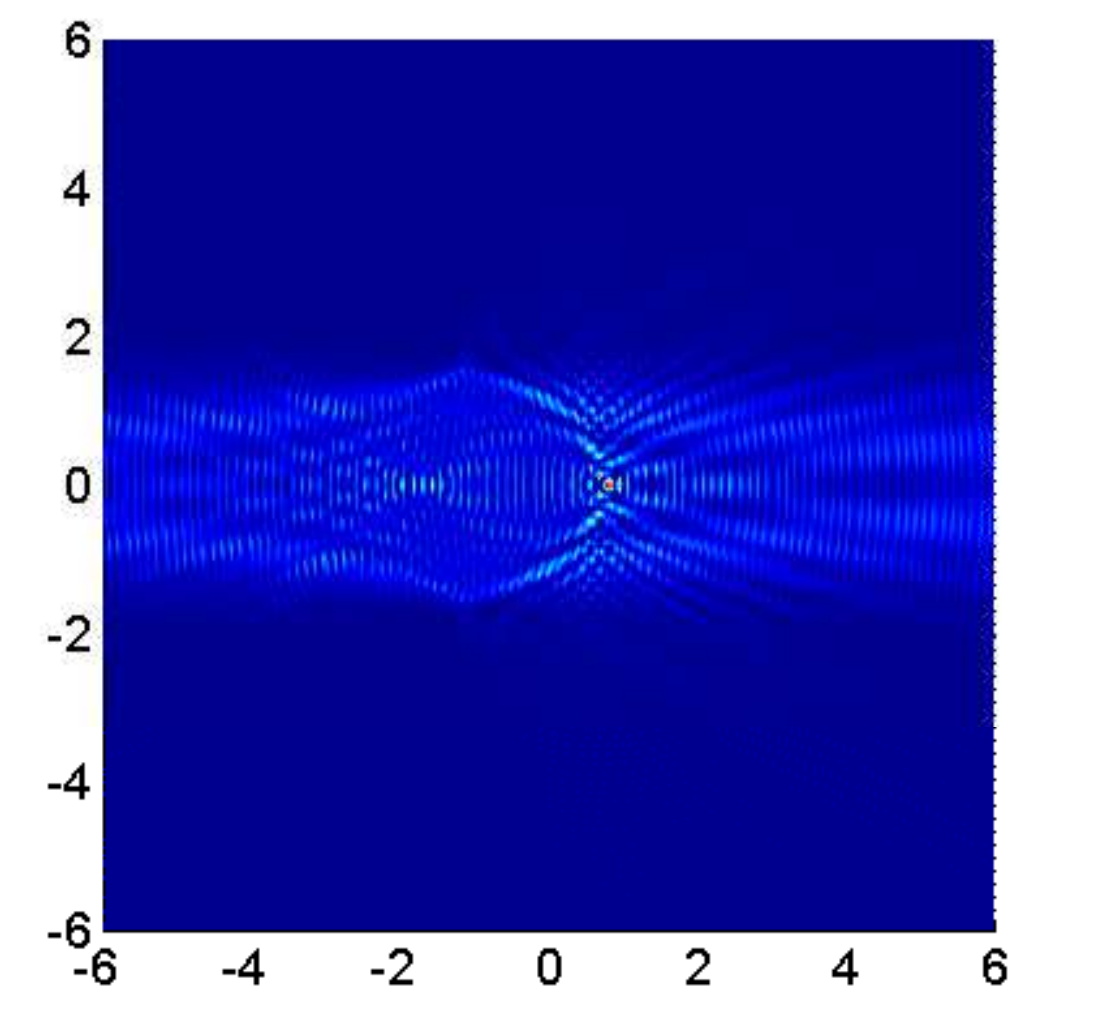}}
  \subfigure[\textbf{4 directions}]
   {\includegraphics[width=0.24\textwidth]{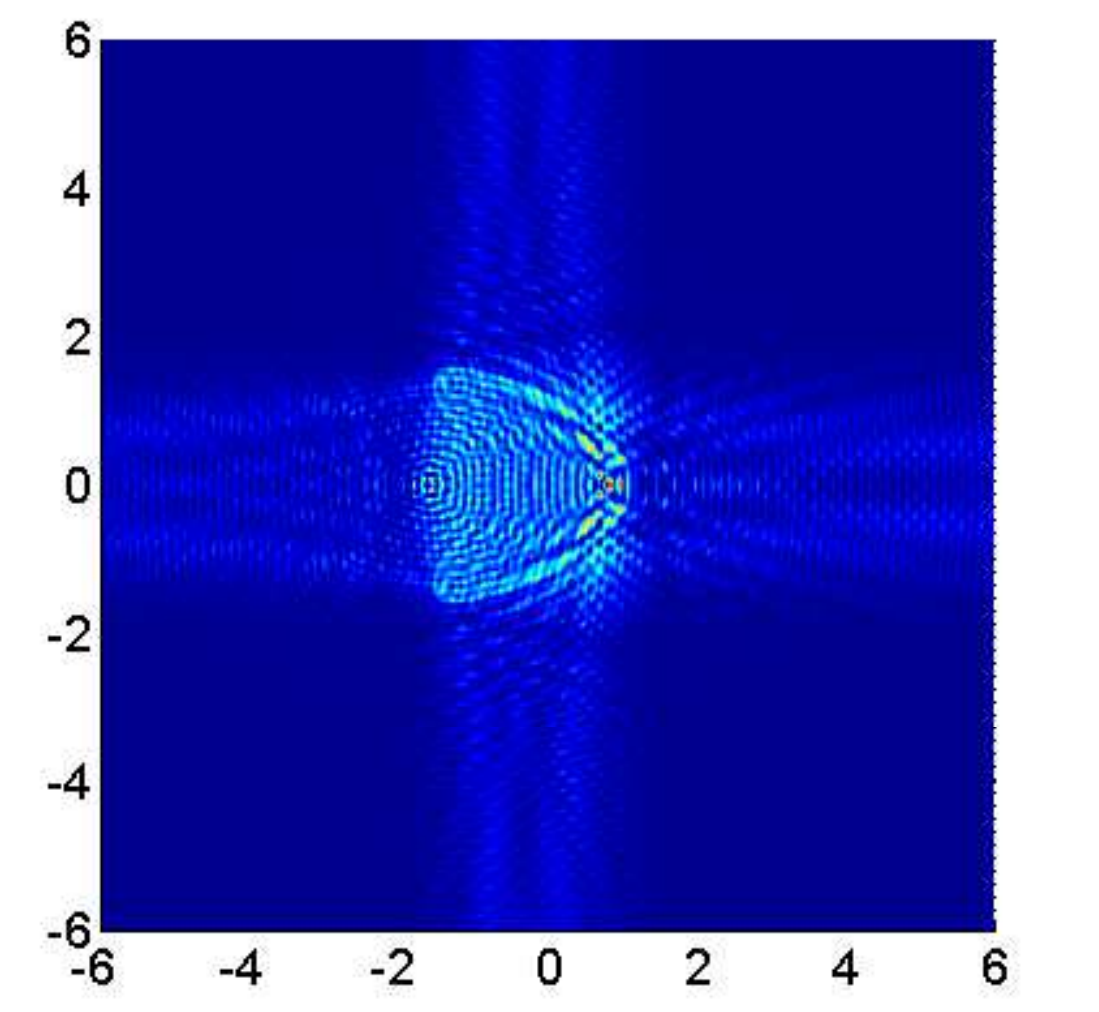}}
\subfigure[\textbf{8 directions}]
   {\includegraphics[width=0.24\textwidth]{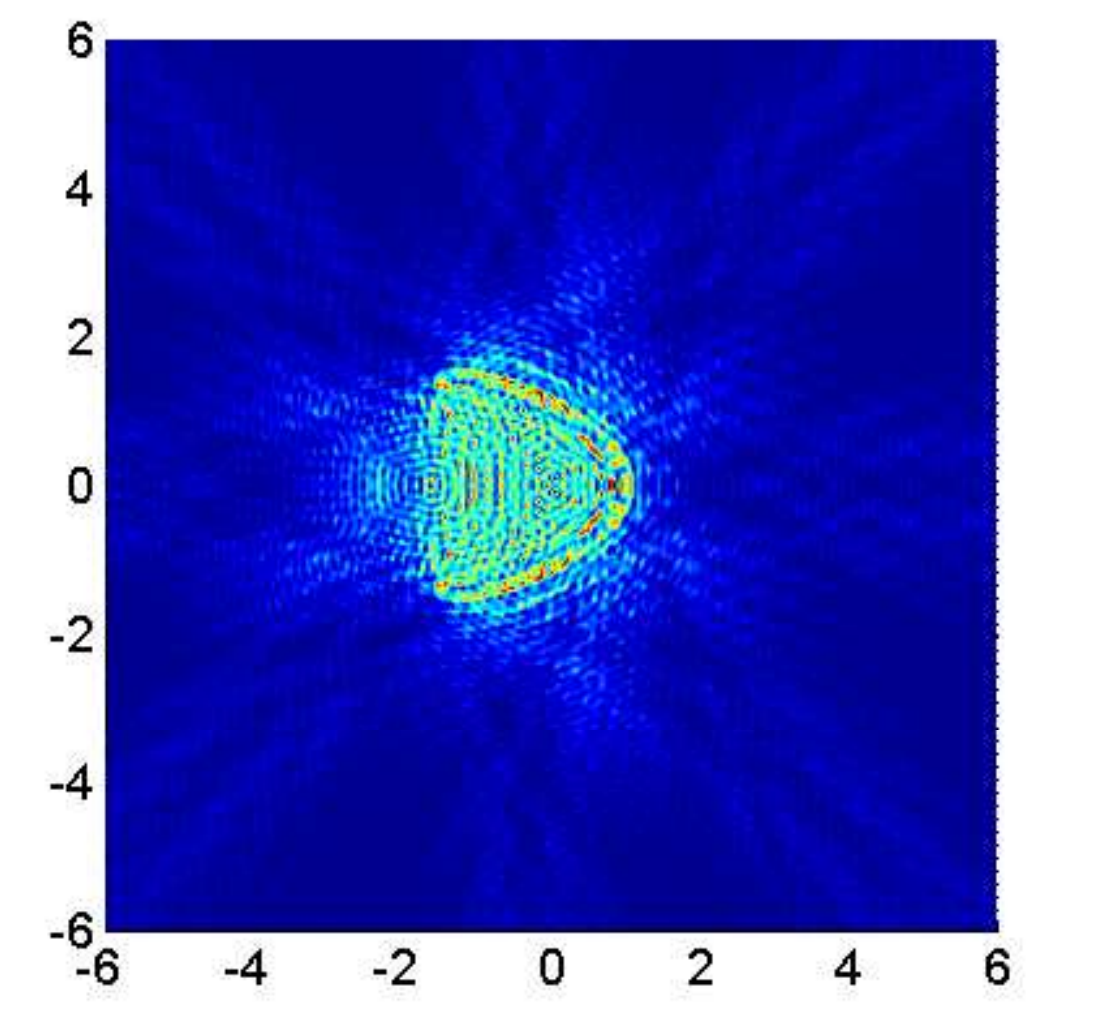}}
\subfigure[\textbf{16 directions}]
   {\includegraphics[width=0.24\textwidth]{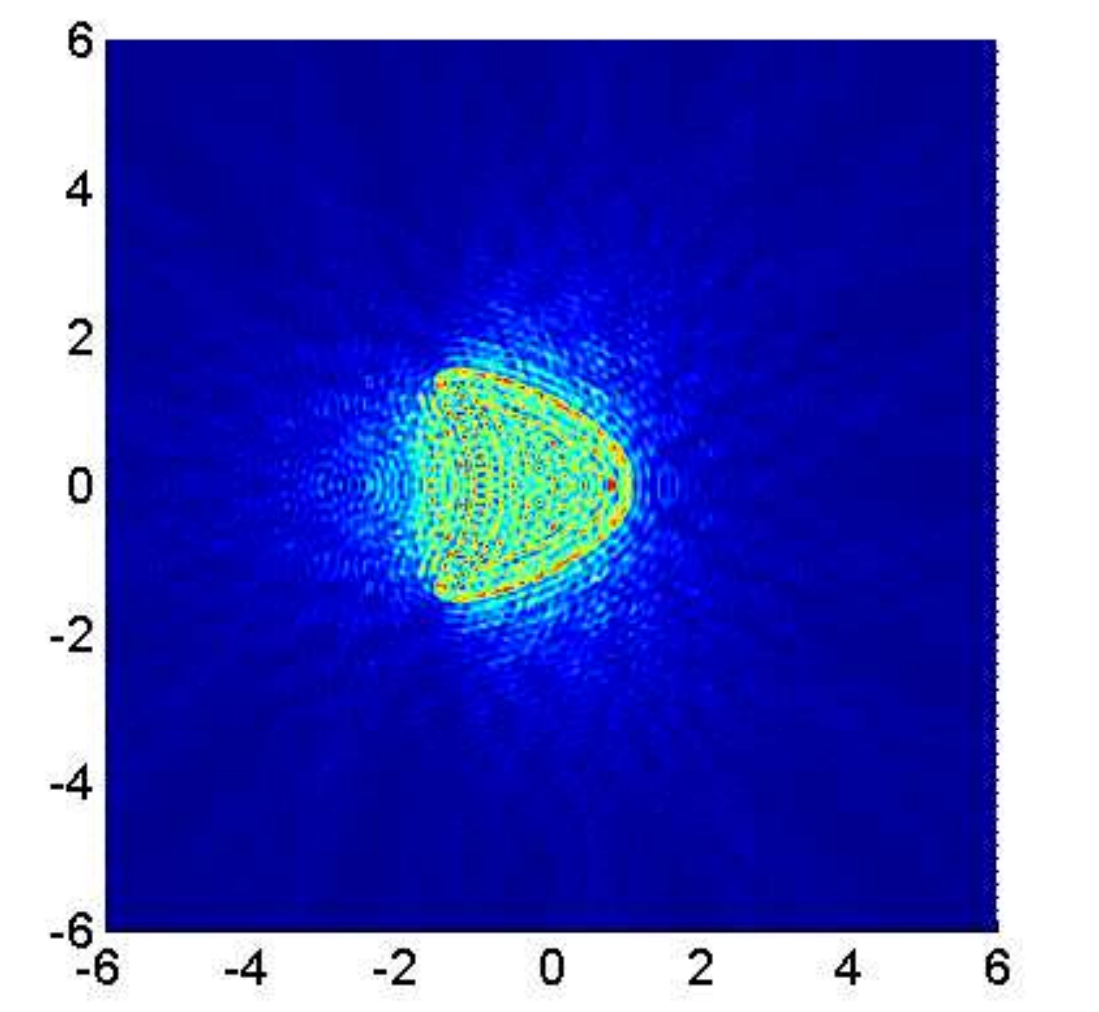}}
\caption{{\bf Example Limited-aperture.} Reconstructions of kite-shaped domain with one or several incident plane shear waves.}
\label{fig:limit_aperture_reconstrucion_I2_less_incidentDirection}
\end{figure}

In all the above examples, we observe that the reconstructions are rather satisfactory,
considering the severe ill-posedness of the inverse scattering problems and the
fact that at least $10\%$ noise is added in the measurements (far field patterns).

\section{Concluding remarks}
\label{sec4}
\setcounter{equation}{0}

In this paper we propose three non-iterative sampling methods for shape identification in inverse elastic scattering problems.
Both the theory foundation and numerical simulations are presented.
Only matrix vector multiplications are involved in the computation, thus
our methods are very fast and robust against measurement noise from the numerical point of view.
The recovering scheme works independently of the physical properties of the underlying scatterers.
There might be several components with different scalar sizes, presented simultaneously.
Our method also allows us to distinguish two components close to each other, which is known to be challenging for numerical reconstruction.
For the limited-aperture problems, some data retrieval techniques are proposed, which have been used for the inverse problems by combining the proposed sampling methods.
An important and interesting  observation is that the reconstructions by using the indicator $I_{\rm SS}$ seem to be the best in all cases, which may
due to the fact that $k_s>k_p$.
Finally, we want to remark that, in all of our numerical simulations,
we have observed that the indicator always takes its maximum on or near the boundary of the scatterer. However, there is still no theory analysis on this fact.

\section*{Acknowledgement}
The research of X. Ji is partially supported by the NNSF of China with Grant Nos. 11271018 and 91630313,
and National Centre for Mathematics and Interdisciplinary Sciences, CAS.
The research of X. Liu is supported by the NNSF of China under grant 11571355 and the Youth Innovation Promotion Association, CAS.

\bibliographystyle{SIAM}%{elsart-num} %{elsart-num-sort} {plain} %

\end{document}